\documentclass[11pt]{amsart}
\usepackage[top=1.5in, bottom=1.5in, left=1.4in, right=1.4in]{geometry} 
\usepackage{graphicx}
\usepackage{amssymb}
\usepackage{epstopdf}
\usepackage{amsthm}
\usepackage{hyperref}
\usepackage[table]{xcolor}
\usepackage{array}
\usepackage{mathtools}
\usepackage{enumerate}
\usepackage{enumitem}
\usepackage{lscape}
\usepackage{verbatim}
\usepackage{color}
\usepackage{hhline}
\usepackage{afterpage}
\usepackage[all]{xy}
\usepackage{graphicx}  
\usepackage{tikz}
\usetikzlibrary{decorations.pathreplacing}
\usepackage{tkz-graph}
\usetikzlibrary{arrows}
\usetikzlibrary{decorations.markings}

\newtheorem{thm}{Theorem}[section]
\newtheorem{lemma}[thm]{Lemma}

\newtheorem{cor}[thm]{Corollary}
\newtheorem{prop}[thm]{Proposition}
\newtheorem{conj}[thm]{Conjecture}

\newtheorem{Definition}[thm]{Definition}
\newenvironment{definition}
  {\begin{Definition}\rm}{\end{Definition}}

\newtheorem{Example}[thm]{Example}
\newenvironment{example}
  {\begin{Example}\rm}{\end{Example}}
  
\newtheorem{Remark}[thm]{Remark}
\newenvironment{remark}
  {\begin{Remark}\rm}{\end{Remark}}
  
 \newtheorem{subclaim}{Subclaim}

 \numberwithin{equation}{section}

\usepackage{etoolbox}
\apptocmd{\sloppy}{\hbadness 10000\relax}{}{}

\allowdisplaybreaks

\pgfarrowsdeclare{mytip}{mytip}
 {
  \pgfarrowsleftextend{-5.59\pgflinewidth}
  \pgfarrowsrightextend{4.5\pgflinewidth}}
 {
  \pgfpathmoveto{\pgfpoint{4.5\pgflinewidth}{0pt}} %Point5
  \pgfpathlineto{\pgfpoint{-0.63\pgflinewidth}{1.09\pgflinewidth}} %Point4
  \pgfpathlineto{\pgfpoint{-5.48\pgflinewidth}{3.06\pgflinewidth}} %Point3
  \pgfpathlineto{\pgfpoint{-5.59\pgflinewidth}{3.0\pgflinewidth}}%Point2
  \pgfpathlineto{\pgfpoint{-3.95\pgflinewidth}{0pt}} %Point1
  \pgfpathlineto{\pgfpoint{-5.59\pgflinewidth}{-3.0\pgflinewidth}}%Point8
  \pgfpathlineto{\pgfpoint{-5.48\pgflinewidth}{-3.06\pgflinewidth}}%Point7
  \pgfpathlineto{\pgfpoint{-0.63\pgflinewidth}{-1.09\pgflinewidth}}%Point6
  \pgfusepathqfill
 }
 
 \pgfarrowsdeclare{mytip2}{mytip2}{%
  \setlength{\arrowsize}{1pt}
  \addtolength{\arrowsize}{.5\pgflinewidth}
  \pgfarrowsrightextend{0}
  \pgfarrowsleftextend{-5\arrowsize}
}{%
  \setlength{\arrowsize}{1pt}
  \addtolength{\arrowsize}{.5\pgflinewidth}
  \pgfpathmoveto{\pgfpoint{-5\arrowsize}{1.5\arrowsize}}
  \pgfpathlineto{\pgfpointorigin}
  \pgfpathlineto{\pgfpoint{-5\arrowsize}{-1.5\arrowsize}}
  \pgfusepathqstroke
}

\def\SetBasicGraph { 	
	\SetVertexMath
	\GraphInit[vstyle=Classic]
	\SetUpVertex[MinSize=0pt]
	\SetVertexLabel
	\tikzset{VertexStyle/.style = {shape = circle,fill = black,minimum size = 0pt,inner sep=0.75pt}}
	\SetUpEdge[color=black]
	\tikzset{->-/.style={decoration={ markings, mark=at position 0.8 with {\arrow{>}}},postaction={decorate}}}
}

\def\SetFancyGraph {
	\SetVertexMath
	\GraphInit[vstyle=Art]
	\SetUpVertex[MinSize=2pt]
	\SetVertexLabel
	\tikzset{VertexStyle/.style = {shape = circle,shading = ball,ball color = black,inner sep = 1.5pt}}
	\SetUpEdge[color=black]
	\tikzset{->-/.style={decoration={ markings, mark=at position 0.8 with {\arrow{>}}},postaction={decorate}}}
	\tikzset{->--/.style={decoration={ markings, mark=at position 0.55 with {\arrow{>}}},postaction={decorate}}}
}

\def\Triangle[#1,#2,#3,#4,#5] {
	\SetBasicGraph
	\Vertex[NoLabel,x=#1+-0.2,y=#2+-0.25]{v_1}
	\Vertex[NoLabel,x=#1+0.15,y=#2+0.25]{v_2}
	\Vertex[NoLabel,x=#1+0.5,y=#2+-0.25]{v_3}
	\ifthenelse{#3=0}{\Edges[style={thick}](v_1,v_2)}{}
	\ifthenelse{#3=1}{\Edges[style={->-,>=mytip2,thick}](v_1,v_2)}{}
	\ifthenelse{#3=-1}{\Edges[style={->-,>=mytip2,thick}](v_2,v_1)}{}
	\ifthenelse{#3=2}{\Edges[style={->,>=mytip2,thick}](v_1,v_2) \Edges[style={->,>=mytip2,thick}](v_2,v_1)}{}
	\ifthenelse{#4=0}{\Edges[style={thick}](v_1,v_3)}{}
	\ifthenelse{#4=1}{\Edges[style={->-,>=mytip2,thick}](v_1,v_3)}{}
	\ifthenelse{#4=-1}{\Edges[style={->-,>=mytip2,thick}](v_3,v_1)}{}
	\ifthenelse{#4=2}{\Edges[style={->,>=mytip2,thick}](v_1,v_3) \Edges[style={->,>=mytip2,thick}](v_3,v_1)}{}
	\ifthenelse{#5=0}{\Edges[style={thick}](v_2,v_3)}{}
	\ifthenelse{#5=1}{\Edges[style={->-,>=mytip2,thick}](v_2,v_3)}{}
	\ifthenelse{#5=-1}{\Edges[style={->-,>=mytip2,thick}](v_3,v_2)}{}
	\ifthenelse{#5=2}{\Edges[style={->,>=mytip2,thick}](v_2,v_3) \Edges[style={->,>=mytip2,thick}](v_3,v_2)}{}
}

\newcommand*{\bigchi}{\mbox{\Large$\chi$}} %big chi

\def\multiset#1#2{\ensuremath{\left(\kern-.3em\left(\genfrac{}{}{0pt}{}{#1}{#2}\right)\kern-.3em\right)}} %multiset notation

\makeatletter
\newcommand{\thickhline}{%
    \noalign {\ifnum 0=`}\fi \hrule height 1pt
    \futurelet \reserved@a \@xhline
}
\newcolumntype{"}{@{\hskip\tabcolsep\vrule width 2pt\hskip\tabcolsep}}
\makeatother

\def\O{\mathcal{O}}
\def\Oref{\mathcal{O}_{\mathrm{ref}}}
\def\Cu{{Cu}}
\def\Cy{{Cy}}
\def\C{{C}}
\def\OC{\overrightarrow{\C}}
\def\OCu{\overrightarrow{\Cu}}
\def\OCy{\overrightarrow{\Cy}}
\def\G{\mathbb{G}}

\definecolor{mygray}{gray}{0.85}

\title{Fourientations and the Tutte Polynomial}
\author{Spencer Backman} 
\email{backman@di.uniroma1.it}
\address{University of Rome ``La Sapienza'', Rome, Italy}
\author{Sam Hopkins}
\email{shopkins@mit.edu}
\address{Massachusetts Institute of Technology, Cambridge, MA, 02139, United States}

\begin{document}

\begin{abstract}
A fourientation of a graph is a choice for each edge of the graph whether to orient that edge in either direction, leave it unoriented, or biorient it.  Fixing a total order on the edges and a reference orientation of the graph, we investigate properties of cuts and cycles in fourientations which give trivariate generating functions that are generalized Tutte polynomial evaluations of the form 
\[(k+m)^{n-1}(k+l)^gT\left(\frac{\alpha k + \beta l + m}{k+m},\frac{\gamma k + l + \delta m}{k+l}\right)\]
for $\alpha,\gamma \in \{0,1,2\}$ and $\beta, \delta \in \{0,1\}$.  We introduce an intersection lattice of~$64$ cut-cycle fourientation classes enumerated by generalized Tutte polynomial evaluations of this form.  We prove these enumerations using a single deletion-contraction argument and classify axiomatically the set of fourientation classes to which our deletion-contraction argument applies.  This work unifies and extends earlier results for fourientations due to Gessel and Sagan~\cite{gessel1996tutte}, and results for partial orientations due to the first author~\cite{backman2014partial}, and the second author and David Perkinson~\cite{hopkins2012bigraphical}, as well as results for total orientations due to Stanley ~\cite{stanley1973acyclic}~\cite{stanley1980decompositions}, Las Vergnas ~\cite{las1980convexity}, Greene and Zaslavsky~\cite{greene1983interpretation}, and Gioan~\cite{gioan2007enumerating}, which were previously unified by Gioan~\cite{gioan2007enumerating}, Bernardi~\cite{bernardi2008tutte}, and Las Vergnas~\cite{las2012tutte}. We conclude by describing how these classes of fourientations relate to geometric, combinatorial, and algebraic objects including bigraphical arrangements, cycle-cocycle reversal systems, graphic Lawrence ideals, Riemann-Roch theory for graphs, zonotopal algebras, and the reliability polynomial.
\end{abstract}

\maketitle

\tableofcontents

\vspace{-1cm}

\section{Introduction}

Throughout we use \emph{graph} to mean finite, undirected graph (although we allow loops and multiple edges). The Tutte polynomial is the most general Tutte-Grothendieck invariant one can associate to a graph; that is, any graph invariant that satisfies a deletion-contraction recurrence is a specialization of the Tutte polynomial. In fact, any graph invariant that satisfies a weighted deletion-contraction recurrence is essentially an evaluation of the Tutte polynomial, as the following theorem, which is sometimes called the \emph{recipe theorem}, makes precise.

\begin{thm}[see~{\cite[Theorem 1]{welsh2000potts}} and~{\cite[Theorem 2.16]{welsh1999tutte}}] \label{thm:gentutte}
Let $\mathcal{G}$ be some set of graphs closed under deletion and contraction, let $\mathbf{k}$ be a field, and let $f\colon \mathcal{G} \to \mathbf{k}$ be some function that is invariant under graph isomorphism. Suppose that $f$ is normalized so that $f(G)=1$ if $G$ has no edges. Suppose further that for every graph $G \in \mathcal{G}$ with at least one edge, there is some edge $e \in E(G)$ such that
\[ f(G) = \begin{cases} af(G / e) + bf(G\setminus e) &\textrm{if $e$ is neither an isthmus nor a loop} \\
 x_0f(G \setminus e)  &\textrm{if $e$ is an isthmus} \\
y_0f(G / e)  &\textrm{if $e$ is a loop},\end{cases}\]
where $G \setminus e$ is graph obtained from $G$ by deleting $e$ and $G/e$ is the graph obtained by contracting $e$. Then for all $G \in \mathcal{G}$ we have 
\[f(G) = a^{n-\kappa}b^{g}T_G\left(\frac{x_0}{a},\frac{y_0}{b}\right),\]
where~$n := |V(G)|$ is the number of vertices of $G$, $\kappa$~is its number of connected components, $g := |E(G)| - |V(G)| + \kappa$~is its cyclomatic number, and $T_G(x,y)$ is its Tutte polynomial.
\end{thm}

In light of Theorem~\ref{thm:gentutte}, we call an expression of the form $a^{n-\kappa}b^{g}T_G(x,y)$ a \emph{generalized Tutte polynomial evaluation}. Note that $n-\kappa$ is the rank of the graphic matroid associated to $G$ and $g$ is its corank. In what follows we assume for simplicity that all graphs are connected. We also write $T(x,y) := T_G(x,y)$ when the graph is implicit.

Our aim in this paper is to systematically exploit Theorem~\ref{thm:gentutte} in order to enumerate various classes of generalized graph orientations via the Tutte polynomial. A fourientation of a graph is a choice for each edge whether to orient that edge in either direction, leave it unoriented, or biorient it. (There are $4^{|E(G)|}$ fourientations of a graph $G$ and thus the name.) A~$(k,l,m)$-fourientation is obtained from a fourientation by assigning each oriented edge one of $k$ colors, each unoriented edge one of $l$ colors, and each bioriented edge one of $m$ colors.  A potential cut (cycle) in a fourientation is the same as a directed cut (cycle) in an ordinary total orientation except that some of the edges may be unoriented (bioriented). In~\S\ref{sec:four} we generate a list of potential cut and cycle properties which mix with one another to give an intersection lattice of $64$ cut-cycle properties of $(k,l,m)$-fourientations such that each associated class is enumerated by a generalized Tutte polynomial evaluation. Moreover, we show that our list of properties is exhaustive: we derive the axioms required for our deletion-contraction proof to apply and show that the set of properties satisfying these axioms consists of precisely the potential cut and cycle properties on our list together with two exceptional cases. In~\S\ref{sec:special} we consider specializations of~$(k,l,m)$ that recover enumerative results about classes of partial orientations and total orientations obtained by many authors, as detailed in~\S\ref{subsec:history} below.

Our axiomatic approach to orientation properties recovers classes of partial orientations which arose in seemingly unrelated contexts and also suggests interesting new avenues of research. In~\S\ref{sec:connections} we outline how several of our cut and cycle properties relate to geometric, combinatorial, and algebraic objects including bigraphical hyperplane arrangements, cycle-cocycle reversal systems, graphic Lawrence ideals, divisors on graphs, zonotopal algebras, and the reliability polynomial. Recent developments in the study of divisors on graphs, including the commutative algebra of the abelian sandpile model~\cite{perkinson2013primer}~\cite{manjunath2014minimal}~\cite{dochtermann2012laplacian}~\cite{mohammadi2013divisors}~\cite{mohammadi2013divisors2}, Riemann-Roch theory for graphs~\cite{baker2007riemann}~\cite{backman2014riemann}, and geometrizations of the Matrix-Tree theorem~\cite{an2014canonical}, highlight the algebraic significance of the relationship between graph orientations and their indegree sequences. The partial orientation classes we define, which we term \emph{min-edge classes}, appear to arise in many situations where one is interested in indegree sequences. A striking example of this phenomenon, described in detail in~\S\ref{subsec:cutin}, is that the acyclic-cut internal partial orientations point the way towards a monomization of the internal power ideal associated to a graph~$G$. While there exist constructions of monomizations of the external~\cite{postnikov1998chern}~\cite{desjardins2010monomization} and central~\cite{postnikov2004trees} power ideal associated to~$G$, there is no such construction for the internal power ideal that works for all~$G$. We arrived at the definition of the min-edge class cut internal only as a result of our abstract machinery, but it conjecturally helps resolve this outstanding problem that we became aware of after we began our research.
 
\subsection{History} \label{subsec:history}
Since at least the seminal work of Stanley~\cite{stanley1973acyclic}, it has been known that the Tutte polynomial counts classes of graph orientations defined in terms of cuts and cycles. Stanley~\cite{stanley1973acyclic} proved that the number of acyclic orientations of a graph is~$T(2,0)$, which is also equal to the chromatic polynomial evaluated at $-1$.  Las Vergnas~\cite{las1980convexity} proved that the number of strongly connected orientations, those with no  directed cut, is~$T(0,2)$.  Greene and Zaslavsky~\cite{greene1983interpretation} showed that the number of acyclic orientations of a graph with a unique source $q$ is~$T(1,0)$.  By fixing a total order on the edges and a reference orientation of the graph, the previous result can be generalized in the following way: the number of acyclic orientations such that the minimum edge in each directed cut is oriented as in the reference orientation is~$T(1,0)$.  These orientations give distinguished representatives for the set of acyclic orientations modulo cut reversals, which can be obtained greedily.  Gioan~\cite{gioan2007enumerating} observed that $T(0,1)$ counts the number of equivalence classes of strongly connected orientations modulo directed cycle reversals, or equivalently indegree sequences of strongly connected orientations (because any two orientations with the same indegree sequence differ by cycle reversals). This theorem is equivalent to the result of Greene and Zaslavsky~\cite{greene1983interpretation} that the number of strongly connected orientations for which the minimum edge in any directed cycle is oriented as in the reference orientation is~$T(0,1)$ since these orientations give distinguished representatives for the set of strongly connected orientations modulo cycle reversals.  Greene and Zaslavaky's result and its equivalence to Gioan's result was rediscovered by Chen, Yang, and Zhang~\cite{chen2008bijection} who investigated it from a bijective perspective.  Stanley~\cite{stanley1980decompositions} observed that the total number of indegree sequences among orientations is counted by~$T(2,1)$, which may be interpreted as a version of the previous result for orientations that are not necessarily strongly connected.  Similarly, Gioan~\cite{gioan2007enumerating} proved that the number of (not necessarily acyclic) $q$-connected orientations is~$T(1,2)$, and the number of indegree sequences of these orientations is~$T(1,1)$.  Trivially, $T(0,0)=0$, the number of strongly connected-acyclic orientations, and $T(2,2)= 2^{|E|}$, the total number of orientations.  Putting all of these enumerations together, Gioan~\cite{gioan2007enumerating} offered a unified framework for interpretations of $T(x,y)$ for all integer values~$0 \leq x,y \leq 2$ in terms of of equivalence classes of orientations.  He presented separate proofs for the evaluations where either $x$ or $y$ is zero, and then obtained the remaining nonzero evaluations via a convolution formula for the Tutte polynomial due to Kook, Reiner, and Stanton~\cite{kook1999convolution}.  Gioan later provided a more unified proof~\cite{gioan2008circuit} making use of matroid duality, while still requiring application of the convolution formula and separate proofs for the evaluations~$T(2,0)$ and~$T(1,0)$.

Gessel and Sagan~\cite{gessel1996tutte} used depth-first search to investigate relationships between the Tutte polynomial and partial orientations (which they call ``suborientations'') and fourientations (which they call ``subdigraphs'') of a graph.  They proved that two generating functions for generalized orientations give generalized Tutte polynomial evaluations with the following specializations:  the number of acyclic partial orientations of a graph is~$2^gT(3,1/2)$, the number of $q$-connected fourientations is~$2^{|E|}T(1,2)$, and the number of acyclic,~$q$-connected partial orientations of a graph is~$2^gT(1,1/2)$. The first result was rediscovered and proven via deletion-contraction by the first author~\cite{backman2014partial}, who also showed that the number of strongly connected partial orientations is~$2^{n-1}T(1/2,3)$.  It was also shown in~\cite{backman2014partial} that the number of partial orientations modulo cut reversals and modulo cycle reversals are~$2^{n-1}T(1,3)$ and~$2^{g}T(3,1)$ respectively.  As in the case of total orientations, the partial orientations for which the minimum edge in every directed cut or cycle is oriented in the same direction as the reference orientation give distinguished representatives for these equivalence classes.  The second author and Perkinson~\cite{hopkins2012bigraphical} observed that the number of regions of a generic bigraphical arrangement is~$2^{n-1}T(3/2,1)$ and the number of bounded regions is~$2^{n-1}T(1/2,1)$. They demonstrated that the regions of a bigraphical arrangement are labeled by a certain class of ``admissible'' acyclic partial orientations.  Using exponential parameters to define a generic arrangement, we give an alternate description of these admissible partial orientations as those for which the minimum edge in any potential cycle is neutral. The partial orientations that label bounded region become those which are in addition strongly connected.

By specializing $(k,l,m)$ in our main Theorem \ref{thm:main} to $(1,1,1)$ (fourientations), $(1,1,0)$ (\emph{type A classes of partial orientations}), $(1,0,1)$ (\emph{type B classes of partial orientations}), and $(1,0,0)$ (total orientations), we obtain tables (see Figure~\ref{fig:fourtables}) in which all of the aforementioned results appear as entries.  Moreover, for $(k,l,m) = (1,0,0)$, our proof of Theorem \ref{thm:main} specializes to a uniform proof of the Tutte polynomial evaluations in Gioan's~$3 \times 3$ square of total orientation classes. Another uniform proof of the Tutte polynomial evaluations in this~$3 \times 3$ square can be obtained from a certain orientation activity expansion of the Tutte polynomial due to Las Vergnas~\cite{las2012tutte} which is closely related to the ``active bijection'' of Gioan-Las Vergnas~\cite{gioan2009active}. We will explore this orientation activity approach in a sequel paper; see~\S\ref{subsec:future} for more details.

\subsection{Acknowledgements} We thank Olivier Bernardi for some enlightening discussions about the Tutte polynomial and for his encouragement with our investigation of fourientations. We thank Criel Merino for explaining the precise statement of the recipe theorem to us.  The second author thanks Farbod Shokrieh for suggesting the use of exponential parameters during the 2013 AIM workshop on generalizations of chip-firing and the critical group. The first author thanks Dave Perkinson for early discussions about partial orientations. We thank Lorenzo Traldi for helpful discussions about (closed forms of) the Tutte polynomial. The first author was partially supported by the European Research Council under the European Unions Seventh Framework Programme (FP7/2007-2013)/ERC Grant Agreement no.~279558, and the Center for Application of Mathematical Principles at the National Institute of Mathematical Sciences in South Korea during the Summer 2014 Program on {\it Applied Algebraic Geometry}.  The second author was partially supported by NSF grant~\#1122374. Finally, the first author would like to thank the MIT Mathematics Department for their hospitality during visits to the second author.

\section{Fourientations and min-edge classes} \label{sec:four}

\subsection{Notation and terminology}
In this section we introduce, enumerate, and classify the main objects of study in this article: {\it min-edge classes of fourientations}.  In order to do so, we begin by first developing some useful notation and terminology, which will be employed throughout the paper.  

Let $G$ be a graph. We use $V(G)$ to denote the vertex set of $G$ and $E(G)$ to denote its edge set. Recall that throughout we will assume that all graphs are connected. (This assumption is justified by the fact that if~$G = G' \sqcup G''$ is the disjoint union of two other graphs then we have $T_G(x,y) = T_{G'}(x,y) \cdot T_{G''}(x,y)$.) We take a moment to describe our notation for the edges of $G$ and for orientations of these edges. Formally,~$V(G)$ is some finite set and~$E(G)$ is a finite set together with a map~$\varphi(G) \colon E(G) \to \multiset{V(G)}{2}$, i.e., a map~$\varphi(G)$ from~$E(G)$ to the set of all multisets of $V(G)$ of size two. By abuse of notation, for an edge $e \in E(G)$ we write~$e = \{u,v\}$ to mean $\varphi(G)(e) = \{u,v\}$; however, note that we may have $e,f \in E$ with~$e = \{u,v\}$ and~$f=\{u,v\}$ but~$e \neq f$ meaning that~$e$ and $f$ are multiple edges between the same vertices; and it is also possible that~$e = \{u,u\}$ is a loop. In order to talk about orientations of a graph it is helpful to have a reference orientation. A \emph{reference orientation} $\Oref$ of~$G$ is a map~$\Oref\colon E(G) \to V(G)\times V(G)$ with~$F \circ \Oref = \varphi(G)$ where~$F\colon V(G) \times V(G) \to \multiset{V}{2}$ is the forgetful map. An \emph{orientation} of~$e \in E(G)$ is a formal symbol $e^{+}$ or~$e^{-}$, where we think of $e^{+}$ as the orientation of $e$ that agrees with $\Oref$ and $e^{-}$ as the orientation that disagrees with~$\Oref$. For $\delta, \varepsilon \in \{+,-\}$ we define~$-\delta$ and~$\delta \cdot \varepsilon$ in the obvious way. When discussing orientations we use the symbols~$\pm$ and~$\mp$ for compactness of notation: any mathematical sentence involving~$\pm$ should be interpreted by replacing all occurrences of $\pm$ with $\delta$, all occurrences of $\mp$ with~$-\delta$, and adding ``for~$\delta \in \{-,+\}$'' at the end of the sentence. Let $\mathbb{E}(G) := \{e^{\pm}\colon e \in E(G)\}$ be the set of orientations of edges of~$G$. We extend~$\Oref$ to a map~$\mathbb{E}(G) \to V(G) \times V(G)$ by setting~\mbox{$\Oref(e^{+}) := \Oref(e)$} and~\mbox{$\Oref(e^{-}) := \Oref(e)^{\mathrm{op}}$}, where~$(u,v)^{\mathrm{op}} := (v,u)$. Again abusing notation, we write~$e^{\pm} = (u,v)$ to mean~$\Oref(e^{\pm}) = (u,v)$; however, note that if~$e = \{u,u\}$ is a loop then $e^{+} = (u,u)$ and~$e^{-} = (u,u)$ but we still treat $e^{+}$ and~$e^{-}$ as different orientations of $e$. We call the pair $(G,\Oref)$ an \emph{oriented graph}.

We also need to review cuts and cycles of graphs as these are fundamental in defining properties of orientations. A \emph{cut} of $G$ is a partition $\Cu = \{U,U^c\}$ for some $U \subseteq V(G)$, where $U^{c} := V(G) \setminus U$ denotes the complement of $U$, such that both~$U$ and $U^{c}$ are nonempty. We define $E(\Cu) := \{e = \{u,v\}\colon u \in U, v\in U^{c}, e \in E(G)\}$. The cut~$\Cu$ is \emph{simple} if the restriction of $G$ to $U$ and the restriction of $G$ to $U^{c}$ both remain connected. (What we call simple cuts are often called ``bonds.'')  An edge~$e \in E(G)$ is an \emph{isthmus} if $E(\Cu) = \{e\}$ for some (necessarily simple) cut~$\Cu$. A \emph{cycle} of $G$ is a list~$\Cy = v_1, e_1, v_2, e_2, \ldots, v_k, e_k$  with~$k \geq 1$, $v_i \in V(G)$, $e_i \in E(G)$, up to rotation and reflection of the indices, such that all edges~$e_i$ are distinct and $e_i = \{v_i, v_{i+1 \mod k}\}$. We define $E(\Cy) := \{e_i\colon 1 \leq i \leq k\}$. The cycle $\Cy$ is \emph{simple} if all the $v_i$ are distinct. An edge~$e \in E(G)$ is a \emph{loop} if~$E(\Cy) = \{e\}$ for some (necessarily simple) cycle~$\Cy$. 

A \emph{directed cut} of~$(G,\Oref)$ is an ordered pair~$\OCu = (U,U^c)$ for some cut~\mbox{$\Cu = \{U,U^c\}$} of $G$; let~$E(\OCu) := E(\Cu)$ and~$\mathbb{E}(\OCu) := \{e^{\pm} = (u,v)\colon u \in U, v\in U^{c}, e \in E(G)\}$; i.e.,~$\mathbb{E}(\OCu)$ is the set of edge orientations from $U$ to $U^c$.  A \emph{directed cycle} of~$(G,\Oref)$ is a list $\OCy = v_1,e^{\delta_1}_1,\ldots,v_k,e^{\delta_k}_k$, with $\delta_i \in \{+,-\}$, up to rotation but not reflection of indices, for some cycle $\Cy = v_1,e_1,\ldots,v_k,e_k$ of $G$ such that~\mbox{$e^{\delta_i}_i = (v_i,v_{i+1 \mod k})$}; let~$E(\OCy) := E(\Cy)$ and~$\mathbb{E}(\OCy) := \{e^{\delta_i}_i\colon 1 \leq i \leq k\}$. Note that each cut (cycle) $\C$ has two associated directed cuts (cycles) $\OC$ and $-\OC$.

\begin{definition}
A \emph{fourientation} $\O$ of an oriented graph $(G,\Oref)$ is a subset of $\mathbb{E}(G)$.
\end{definition}

In other words, a fourientation is a choice for each edge $e$ of a subset of $\{e^{+},e^{-}\}$. If~$e^{+},e^{-} \in \O$ then we say $e$ is \emph{bioriented} in $\O$, and if $e^{+},e^{-} \notin \O$ then we say $e$ is \emph{unoriented} in $\O$.  An \emph{oriented edge} $e$ of $\O$ is one for which $e^{\pm} \in \O$ but $e^{\mp} \notin \O$.  We emphasize that $\Oref$ is not essential in the definition of a fourientation but $\Oref$ has allowed us to introduce the very useful notation $e^+$ and $e^-$ which we employ throughout this paper.  Moreover, when we discuss the properties of cuts and cycles in fourientations which define classes enumerated by generalized Tutte polynomial evaluations the reference orientation will play an indispensable role. Fourientations were introduced and studied in an enumerative context by Gessel and Sagan~\cite{gessel1996tutte} under the name of ``subdigraphs.'' They are also superficially similar to the orientations of signed graphs investgated by Zaslavsky~\cite{zaslavsky1991orientation}; but note that Zaslavsky's notion of a cycle in a signed graph orientation is quite different from the kinds of cycles of fourientations we investigate. It seems plausible that there is some deeper connection between orientations of signed graphs and fourientations and it would be very interesting to find such a relationship.

\begin{definition}
A \emph{potential cut (cycle)} in a fourientation is the same as a directed cut (cycle) in an ordinary total orientation except that some of the edges may be unoriented (bioriented). More formally, a \emph{potential cut (cycle)} of a fourientation is a directed cut (cycle) of the oriented graph such that each edge in that cut (cycle) is either oriented in agreement with the cut (cycle) or is unoriented (bioriented). In symbols,~$\OCu$ is a potential cut belonging to some larger fourientation~$\O$ if $e^{\pm} \in \mathbb{E}(\OCu) \Rightarrow e^{\mp} \notin \O$ for all~$e \in E$, and~$\OCy$ is a potential cycle of~$\O$ if $e^{\pm} \in \mathbb{E}(\OCy) \Rightarrow e^{\pm} \in \O$ for all~$e \in E$.
\end{definition}

\begin{example}
Let $(G,\Oref)$ be an oriented graph and $\O$ be a fourientation of $(G,\Oref)$ as below:
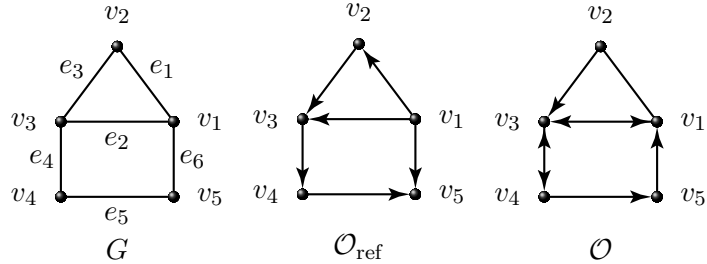
\begin{center}
\begin{tikzpicture}[scale=1]
	\SetFancyGraph
	\Vertex[LabelOut,Lpos=0, Ldist=.1cm,x=0.75,y=0]{v_1}
	\Vertex[LabelOut,Lpos=90, Ldist=.1cm,x=0,y=1]{v_2}
	\Vertex[LabelOut,Lpos=180, Ldist=.1cm,x=-0.75,y=0]{v_3}
	\Vertex[LabelOut,Lpos=180, Ldist=.1cm,x=-0.75,y=-1]{v_4}
	\Vertex[LabelOut,Lpos=0, Ldist=.1cm,x=0.75,y=-1]{v_5}
	\Edges[style={thick}](v_1,v_2)
	\Edges[style={thick}](v_1,v_3)
	\Edges[style={thick}](v_2,v_3)
	\Edges[style={thick}](v_3,v_4)
	\Edges[style={thick}](v_4,v_5)
	\Edges[style={thick}](v_5,v_1)
	\node at (0.6,0.7) {$e_1$};
	\node at (0,-0.25) {$e_2$};
	\node at (-0.6,0.7) {$e_3$};
	\node at (-1,-0.5) {$e_4$};
	\node at (0,-1.25) {$e_5$};
	\node at (1,-0.5) {$e_6$};
	\node at (0,-1.7){$G$};
\end{tikzpicture} \begin{tikzpicture}[scale=1]
	\SetFancyGraph
	\Vertex[LabelOut,Lpos=0, Ldist=.1cm,x=0.75,y=0]{v_1}
	\Vertex[LabelOut,Lpos=90, Ldist=.1cm,x=0,y=1]{v_2}
	\Vertex[LabelOut,Lpos=180, Ldist=.1cm,x=-0.75,y=0]{v_3}
	\Vertex[LabelOut,Lpos=180, Ldist=.1cm,x=-0.75,y=-1]{v_4}
	\Vertex[LabelOut,Lpos=0, Ldist=.1cm,x=0.75,y=-1]{v_5}
	\Edges[style={->,thick,>=mytip}](v_1,v_2)
	\Edges[style={->,thick,>=mytip}](v_1,v_3)
	\Edges[style={->,thick,>=mytip}](v_2,v_3)
	\Edges[style={->,thick,>=mytip}](v_3,v_4)
	\Edges[style={->,thick,>=mytip}](v_4,v_5)
	\Edges[style={->,thick,>=mytip}](v_1,v_5)
	\node at (0,-1.7) {$\Oref$};
\end{tikzpicture} \begin{tikzpicture}[scale=1]
	\SetFancyGraph
	\Vertex[LabelOut,Lpos=0, Ldist=.1cm,x=0.75,y=0]{v_1}
	\Vertex[LabelOut,Lpos=90, Ldist=.1cm,x=0,y=1]{v_2}
	\Vertex[LabelOut,Lpos=180, Ldist=.1cm,x=-0.75,y=0]{v_3}
	\Vertex[LabelOut,Lpos=180, Ldist=.1cm,x=-0.75,y=-1]{v_4}
	\Vertex[LabelOut,Lpos=0, Ldist=.1cm,x=0.75,y=-1]{v_5}
	\Edges[style={thick}](v_1,v_2)
	\Edges[style={->,thick,>=mytip}](v_2,v_3)
	\Edges[style={->,thick,>=mytip}](v_1,v_3)
	\Edges[style={->,thick,>=mytip,dash pattern=on 0pt off 100pt}](v_3,v_1)
	\Edges[style={->,thick,>=mytip}](v_3,v_4)
	\Edges[style={->,thick,>=mytip,dash pattern=on 0pt off 100pt}](v_4,v_3)
	\Edges[style={->,thick,>=mytip}](v_4,v_5)
	\Edges[style={->,thick,>=mytip}](v_5,v_1)
	\node at (0,-1.7) {$\O$};
\end{tikzpicture}
\end{center}
Here $\O = \{e_2^{+},e_2^{-},e_3^{+},e_4^{+},e_4^{-},e_5^{+},e_6^{-}\}$. Observe that $\OCu = (\{v_2\},\{v_1,v_3,v_4,v_5\})$ is a potential cut of $\O$ and $\OCy = v_1,e_2^{+},v_3,e_4^{+},v_4,e_5^{+},v_5,e_6^{-}$ is a potential cycle of $\O$.
\end{example}

In this section we define various classes of fourientations of~$(G,\Oref)$ in terms of potential cuts and cycles. We will need more input data to define these classes. Specifically, we will also need $<$, a total order on the edges of $G$.  Such an edge order often appears in investigations of the Tutte polynomial because it allows one to define the (internal and external) activity of spanning trees of a graph. It may be possible to extend our work to allow for other notions of activity; for instance, see the recent paper~\cite{courtiel2014general} which develops a unified perspective for various kinds of activity. However, we will stick to the most classical case of activity defined in terms of a total edge order here. We call the triple $\G = (G,\Oref,<)$ an \emph{ordered, oriented graph}. A \emph{fourientation} of $\G$ is of course a fourientation of the underlying oriented graph~$(G,\Oref)$.

The classes of fourientations we will define are enumerated by the Tutte polynomial, so we now review deletion and contraction. For $e = \{u,v\} \in E(G)$, the graph obtained by \emph{deleting} $e$ is denoted $G \setminus e$; this graph has~$V(G \setminus e) := V(G)$ and~$E(G\setminus e):= E \setminus \{e\}$. The graph obtained by \emph{contracting} $e$ is denoted $G / e$; now we set~$V(G/ e) := V / \sim$ where~$\sim$ is the identification $u \sim v$, and again $E(G / e) := E \setminus \{e\}$. Of course, we can also form the deletion~$\G \setminus e$ and contraction~$\G / e$ of an ordered, oriented graph~$\G$ by keeping track of the extra data in the obvious way. We similarly define the deletion~$\O \setminus e$ and contraction $\O / e$ of a fourientation~$\O$ (which in fact are both just equal to~$\O \setminus \{e^{+},e^{-}\}$). In particular note that if $e^{+},e^{-} \notin \O$ then we will often treat $\O$ as a fourientation of~$\G \setminus e$ and of~$\G / e$. For a subset of edges~$H \subseteq E(G)$ we also define~$G \setminus H$~($G / H$) to be the graph obtained by deleting (contracting) all the edges in~$H$ in any order. Of course we similarly define $\G \setminus H$, $\G / H$, $\O \setminus H$, and~$\O / H$. For a simple cut~$\Cu$ of $G$, we define the \emph{contraction to $\Cu$}, denoted~$G_{\Cu}$, to be $G / (E(G) \setminus E(\Cu))$. The  contraction to a simple cut always yields a \emph{banana graph}~$B_n$ for some $n \geq 1$, where the family of banana graphs is
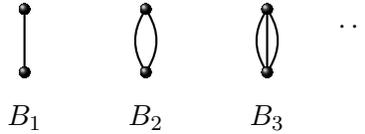
\begin{center}
 \begin{tikzpicture}[scale=1.2]
	\SetFancyGraph
	\Vertex[NoLabel,x=0,y=0.7]{v_2}
	\Vertex[NoLabel,x=0,y=0]{v_1}
	\Edges[style={thick}](v_1,v_2)
	\node at (0,-0.5){$B_1$};
\end{tikzpicture} \qquad \begin{tikzpicture}[scale=1.2]
	\SetFancyGraph
	\Vertex[NoLabel,x=0,y=0.7]{v_2}
	\Vertex[NoLabel,x=0,y=0]{v_1}
	\Edges[style={thick, bend left}](v_1,v_2)
	\Edges[style={thick, bend right}](v_1,v_2)
	\node at (0,-0.5){$B_2$};
\end{tikzpicture} \qquad \begin{tikzpicture}[scale=1.2]
	\SetFancyGraph
	\Vertex[NoLabel,x=0,y=0.7]{v_2}
	\Vertex[NoLabel,x=0,y=0]{v_1}
	\Edges[style={thick}](v_1,v_2)
	\Edges[style={thick, bend left}](v_1,v_2)
	\Edges[style={thick, bend right}](v_1,v_2)
	\node at (0,-0.5){$B_3$};
	\node at (1,0.5){$\cdots$};
\end{tikzpicture}
\end{center}
Similarly, for a simple cycle $\Cy$ of $G$, we define the \emph{restriction to~$\Cy$}, which we denote~$G_{\Cy}$, to be the graph obtained from $G \setminus (E(G) \setminus E(\Cy))$ by removing all isolated vertices. The restriction to a simple cycle always yields a \emph{cycle graph}~$C_n$ for some~$n \geq 1$, where the family of cycle graphs is
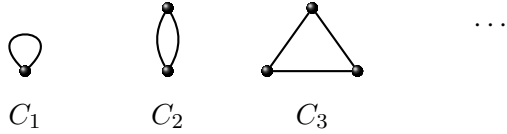
\begin{center}
\begin{tikzpicture}[scale=1.2]
	\SetFancyGraph
	\Vertex[NoLabel,x=0,y=0]{v_1}
	\Edges[style={thick, loop, looseness=20}](v_1,v_1)
	\node at (0,-0.5){$C_1$};
\end{tikzpicture} \qquad \begin{tikzpicture}[scale=1.2]
	\SetFancyGraph
	\Vertex[NoLabel,x=0,y=0.7]{v_2}
	\Vertex[NoLabel,x=0,y=0]{v_1}
	\Edges[style={thick, bend left}](v_1,v_2)
	\Edges[style={thick, bend right}](v_1,v_2)
	\node at (0,-0.5){$C_2$};
\end{tikzpicture} \qquad \begin{tikzpicture}[scale=1.2]
	\SetFancyGraph
	\Vertex[NoLabel,x=0,y=0.7]{v_2}
	\Vertex[NoLabel,x=-0.5,y=0]{v_3}
	\Vertex[NoLabel,x=0.5,y=0]{v_1}
	\Edges[style={thick}](v_1,v_2)
	\Edges[style={thick}](v_1,v_3)
	\Edges[style={thick}](v_2,v_3)
	\node at (0,-0.5){$C_3$};
	\node at (2,0.5){$\cdots$};
\end{tikzpicture}
\end{center}
We define the restriction to a simple cycle $\G_{\Cy}$ and the contraction to a simple cut~$\G_{\Cu}$ of an ordered, oriented graph $\G$ in the obvious way by keeping track of the extra data. We also similarly define the restriction to a simple cycle $\O_{\Cy}$ and contraction to a simple cut $\O_{\Cu}$ for fourientations $\O$.

As an aside, we note that by contracting all the bioriented edges and deleting all the unoriented edges in a fourientation we obtain a total orientation of a graph minor. This procedure seems quite natural as potential cuts and cycles are mapped to directed cuts and cycles, respectively.  Unfortunately, in order to reverse this procedure we must remember how the oriented graph minor was obtained (i.e.,~which edges were contracted and which were deleted) because various fourientations may be mapped to the same oriented graph minor.

A fundamental notion for graph orientations is that of reachability by directed paths. In an ordinary total orientation $\O$ we say that the vertex $v$ is reachable from the vertex $u$ if we can walk from $u$ to $v$ along a series of edges that are oriented in $\O$ consistently with our walk. Because we are viewing a bioriented edge as the union of both orientations of an edge we will allow a bioriented edge to be traversed in either direction. On the other hand, unoriented edges cannot be traversed in either direction (because they are not present). Thus we define a \emph{potential path} $P$ of a fourientation~$\O$ of $(G,\Oref)$ to be a list~$v_1,e^{\delta_i}_1,\ldots,e^{\delta_{k-1}}_{k-1},v_k$ for some $k \geq 1$ such that the~$e_i$ are distinct,~$e^{\delta_i}_i = (v_i,v_{i+1})$, and~$e^{\delta_i} \in \O$ for all~$1 \leq i \leq k$. We say that~$P$ is a potential path \emph{from $v_1$ to $v_k$} and we set $E(P) := \{e_i\colon 1 \leq i \leq k-1\}$. If there is a potential path $P$ of $\O$ from~$u$ to~$v$ then we say~$v$ is \emph{reachable} from~$u$ in $\O$. It is a classical fact, which can be seen by considering reachability classes, that every edge in a total orientation belongs to a directed cycle or a directed cut but not both.  It remains the case in fourientations that every oriented edge belongs to either a potential cut or potential cycle but not both, as we show in Proposition~\ref{prop:ordecomp} below. First we present a more basic lemma that says that potential cuts and cycles of a fourientation are disjoint.

\begin{lemma} \label{lem:disjoint}
Let $\O$ be a fourientation of $(G,\Oref)$. Let~$\OCu$ be a potential cut of~$\O$ and~$\OCy$ a potential cycle of~$\O$. Then $E(\OCu) \cap E(\OCy) = \emptyset$.
\end{lemma}
\begin{proof}
Certainly if $e$ is bioriented in $\O$ then it does not belong to a potential cut and if $e$ is unoriented in $\O$ then it does not belong to a potential cycle. So suppose that~$e^{\pm} = (u,v) \in \O$ but $e^{\mp} \notin \O$ and let~$\OCu = (U,U^c)$ be a potential cut of~$\O$ with~$e \in E(\OCu)$. Note that~$u$ is not reachable from~$v$ because any potential path from~$v$ to~$u$ would have to go through an edge in~$E(U,U^c)$ and these are all either unoriented or oriented from~$U$ into~$U^c$. Thus there is no potential cycle~$\OCy$ of~$\O$ with~$e \in E(\OCy)$.
\end{proof}

\begin{prop} \label{prop:ordecomp}
Let~$\O$ be a fourientation of $(G,\Oref)$ and $e$ an oriented edge of $\O$. Then $e \in E(\OCu)$ for some potential cut $\OCu$ of $\O$ or $e \in E(\OCy)$ for some potential cycle~$\OCy$ of~$\O$ but not both.
\end{prop}

\begin{proof}
Let $e^{\pm} = (u,v) \in \O$ with $e^{\mp} \notin \O$. Let $U$ be set of vertices reachable from $v$ in~$\O$. If $u \in U$ then $e$ belongs to a potential cycle. Otherwise $(U,U^c)$ is a potential cut containing $e$. By  Lemma~\ref{lem:disjoint} we know that $e^{\pm}$ cannot belong to both a potential cut and a potential cycle.
\end{proof}

In general we cannot partition all of the edges of a fourientation into potential cuts and cycles. However, the following two propositions offer two dual ways to extend the partition of the oriented edges in Proposition~\ref{prop:ordecomp} to a decomposition of the entire edge set of our graph.

\begin{prop} \label{prop:cutdecomp}
Let $\O$ be a fourientation of $(G,\Oref)$. Then there is a unique decomposition $E(G) = E_{\mathrm{cu}}(\O) \sqcup E^c_{\mathrm{cu}}(\O)$ such that
\begin{enumerate}[label=(\alph*),leftmargin=1cm]
\item for any $e \in E_{\mathrm{cu}}(\O)$ we have $e \in E(\OCu)$ for some potential cut $\OCu$ of $\O / E^c_{\mathrm{cu}}(\O)$; \label{cond:cutdecompdel}
\item there is no $e \in E^c_{\mathrm{cu}}(\O)$ with $e \in E(\OCu)$ for any potential cut $\OCu$ of $\O \setminus E_{\mathrm{cu}}(\O)$. \label{cond:cutdecompcon}
\end{enumerate}
\end{prop}

\begin{prop} \label{prop:cycdecomp}
Let $\O$ be a fourientation of $(G,\Oref)$. Then there is a unique decomposition $E(G) = E_{\mathrm{cy}}(\O) \sqcup E^c_{\mathrm{cy}}(\O)$ such that
\begin{enumerate}[label=(\alph*),leftmargin=1cm]
\item for any $e \in E_{\mathrm{cy}}(\O)$ we have $e \in E(\OCy)$ for some potential cycle $\OCy$ of $\O \setminus E^c_{\mathrm{cy}}(\O)$;
\item there is no $e \in E^c_{\mathrm{cy}}(\O)$ with $e \in E(\OCy)$ for any potential cycle $\OCy$ of $\O / E_{\mathrm{cy}}(\O)$.
\end{enumerate}
\end{prop}

\begin{proof}[Proof of Proposition~\ref{prop:cutdecomp}]
We first show existence.  Let $E_{\mathrm{cu}}(\O)$ be the set of edges which belong to a potential cut of $\O$ and $E^c_{\mathrm{cu}}(\O)$ the complement of this set.  Clearly condition~\ref{cond:cutdecompdel} is satisfied. To show~\ref{cond:cutdecompcon}, let $e \in E^c_{\mathrm{cu}}(\O)$. If $e$ is bioriented in $\O$ then certainly it does not belong to a potential cut of $\O / E_{\mathrm{cy}}(\O)$. Next suppose $e$ is oriented in $\O$. Then since $e$ does not belong to a potential cut of $\O$, by Proposition~\ref{prop:cutdecomp} it belongs to a potential cycle $\OCy$ of $\O$. Note that $E(\OCy) \cap E_{\mathrm{cu}}(\O) = \emptyset$ by Lemma~\ref{lem:disjoint}. Thus~$\OCy$ persists as a potential cycle in $\O / E_{\mathrm{cy}}(\O)$, so again by Lemma~\ref{lem:disjoint} we have that $e$ belongs to no potential cycle. Finally, suppose that $e = \{u,v\}$ is unoriented in $\O$. Note that because $e$ does not belong to a potential cut, there is a potential path $P$ from~$u$ to~$v$ and another potential path $P'$ from~$v$ to~$u$. All of the edges in $E(P) \cup E(P')$ either belong to potential cycles or are bioriented; at any rate, none of them belong to potential cuts. Thus $P$ and $P'$ persist as potential paths in $\O / E_{\mathrm{cy}}(\O)$. Finally, note that the paths $P$ and $P'$ prevent $e$ from belonging to any potential cut of $\O / E_{\mathrm{cy}}(\O)$. So indeed regardless of how $e$ is fouriented it does not belong to any potential cut of~$\O / E_{\mathrm{cy}}(\O)$.

For proving uniqueness of this decomposition, suppose $E(G) = A\sqcup B$ and every edge of $G/B$ belongs to a potential cut of $\O / B$ while no edges of $G \setminus A$ belong to a potential cut of $\O \setminus A$. First suppose that there exits some $e \in A \setminus E_{\mathrm{cu}}(\O)$.  We know that $e$ does not belong to a potential cut in $\O$, hence it certainly does not belong to a potential cut in $\O /B$, which is a contradiction.  Therefore we may assume that $A\subseteq E_{\mathrm{cu}}(\O)$ and~$E^c_{\mathrm{cu}}(\O) \subseteq B$.  Next suppose there is some edge in $e \in B \setminus E^c_{\mathrm{cu}}(\O)$.  We know that~$e$ belongs to a potential cut in $\O$, and therefore $e$ belongs to a potential cut in $\O \setminus A$, which is a contradiction.  Thus $A  = E_{\mathrm{cu}}(\O)$ and $B = E^c_{\mathrm{cu}}(\O)$.
\end{proof}

\begin{proof}[Proof of Proposition~\ref{prop:cycdecomp}]
The proof is analogous to the proof of Proposition~\ref{prop:cutdecomp}. We define~$E_{\mathrm{cy}}(\O)$ to be the set of edges which belong to a potential cycle.
\end{proof}

\subsection{Tutte fourientation properties and the main theorem}
We now define \emph{Tutte fourientation properties}. These properties are defined axiomatically so as to satisfy exactly those conditions necessary for us to carry out a deletion-contraction argument that invokes Theorem~\ref{thm:gentutte} and proves that they are enumerated by generalized Tutte polynomial evaluations. The key observation that motivates this definition is that if some objects associated to our graph $G$ are enumerated by a generalized Tutte polynomial evaluation, then there is some way of recursively deleting and contracting edges of~$G$ so that at each step our enumeration respects the relevant weighted deletion-contraction recurrence. We may as well assume the order that we delete and contract is dictated by~$<$. Thus we force the weighted deletion-contraction recurrence to hold with respect to the maximum edge of our graph. As we will see later, we can also give explicit descriptions of these properties that focuses instead on the statuses of minimum edges in cuts and cycles.

\begin{definition}
A \emph{fourientation property} is a map
\[\{ (\G, \O) \colon \O \textrm{ is a fourientation of the ordered, oriented graph } \G \} \to \{\mathrm{good},\mathrm{bad}\}\]
that is invariant under isomorphism of the input.\footnote{Formally, we say that $((G^1,\Oref^1,<^1),\O^1)$ is isomorphic to $((G^2,\Oref^2,<^2),\O^2)$ if there exist bijections $\nu\colon V(G^1) \to V(G^2)$ and $\eta\colon E(G^1) \to E(G^2)$ such that for all $e, e' \in E(G^1)$:
\begin{itemize}
\item $\varphi(G^2)(\eta(e)) = \nu(\varphi(G^1)(e))$;
\item $\Oref^2(\eta(e)) = \nu(\Oref^1(e))$;
\item $e <^1 e'$ if and only if $\eta(e) <^2 \eta(e')$;
\item $e^{\pm} \in \O^1$ if and only if $\eta(e)^{\pm} \in \O^2$.
\end{itemize}} When $(\G, \O)$ is good we say that $\O$ is a good fourientation of $\G$ with respect to the property, and similarly when $( \G, \O)$ is bad. In general we identify a fourientation property with its set of good fourientations. We call a fourientation property a \emph{cut (cycle) property} if $\O$ is a good fourientation of~$\G$ if and only if $\O_{\C}$ is a good fourientation of $\G_{\C}$ for each simple cut (cycle) $\C$ of $\G$. We call a cut (cycle) property a \emph{Tutte cut (cycle) property} if for all ordered, oriented graphs $\G$ we have
\begin{enumerate}[label=\textbf{T\arabic*},ref=\textbf{T\arabic*}]
\item $\O$ is a bad fourientation of $\G$ only if $\O$ has a potential cut (cycle); \label{cond:tuttepot}
\item if the maximum edge $e$ of $\G$ is neither an isthmus nor a loop, then \label{cond:tuttemax}
\begin{enumerate}[label=\textbf{(\alph*)}, ref=\textbf{\theenumi(\alph*)}]
\item if $\O$ is a good fourientation of $\G \setminus e$ ($\G / e$) then $\O \cup \{e^{+}\}$ and $\O \cup \{e^{-}\}$ are both good fourientations of $\G$; \label{cond:tuttegoodor}
\item if $\O$ is a bad fourientation of $\G \setminus e$ ($\G / e$) but $\O \cup \{e^{+},e^{-}\}$ ($\O$) is a good fourientation of $\G$, then exactly one of $\O \cup \{e^{+}\}$ or $\O \cup \{e^{-}\}$ is a good fourientation of $\G$; \label{cond:tuttebador}
\item $\O$ is a good fourientation of $\G \setminus e$ ($\G / e$) if and only if~$\O$ ($\O \cup \{e^{+},e^{-}\}$) is a good fourientation of $\G$. \label{cond:tutteun}
\end{enumerate}
\end{enumerate}
A \emph{Tutte cut-cycle property} is the intersection of a Tutte cut property with a Tutte cycle property. 
\end{definition}

The following lemma says that condition~\ref{cond:tutteun} applies to Tutte cut-cycle properties as well as individual Tutte cut or cycle properties.

\begin{lemma}\label{lem:goodfourup}
Let $e$ be the maximum edge of $\G$, and assume $e$ is not an isthmus or loop. Then $\O$ is a good fourientation of $\G\setminus e$ with respect to some Tutte cut-cycle property if and only if $\O$ is good for $\G$. Similarly, $\O / e$ is good for~$\G/e$ if and only if~$\O \cup \{e^{+},e^{-}\}$ is good for $\G$.
\end{lemma}

\begin{proof}
Suppose $\O$ is bad for $\G \setminus e$ with respect to the cycle property: then there is a simple cycle $\Cy$ of $\G \setminus e$ so that $\O_{\Cy}$ is bad for the cycle property; this cycle $\Cy$ persists in $\G$ and in fact $\G_{\Cy}$ is isomorphic to $(\G\setminus e)_{\Cy}$; therefore $\O$ is also bad for~$\G$. Similarly, if $\O$ is bad for $\G$ with respect to the cycle property and $e^{+},e^{-} \notin \O$, then there is a simple cycle $\Cy$ of $\G$ so that $\O_{\Cy}$ is bad for the cycle property. If we had~$e \in E(\Cy)$, then $\O_{\Cy}$ could not have a potential cycle and so could not be bad by condition~\ref{cond:tuttepot}. So indeed $e \notin E(\Cy)$, and thus~$\Cy$ remains a simple cycle of~$\G \setminus e$, and thus $\O$ remains bad for $\G \setminus e$. On the other hand, that $\O$ is bad for~$\G \setminus e$ with respect to the cut property if and only if $\O$ is bad for~$\G$ is exactly condition~\ref{cond:tutteun}. The proof for $\G/e$ is analogous.
\end{proof}

\begin{lemma}\label{lem:goodfourdown}
Let $e$ be the maximum edge of $\G$ and let $\O$ be a good fourientation of~$\G$ with respect to some Tutte cut-cycle property. Then either $\O \setminus e$ is good for~$\G \setminus e$ or~$\O / e$ is good for~$\G / e$.
\end{lemma}

\begin{proof} If $e$ is an isthmus or loop then $\O \setminus e$ and $\O/e$ are both clearly good. So assume that~$e$ is neither an isthmus nor a loop. By Lemma~\ref{lem:goodfourup}, if $e$ is unoriented in $\O$ then $\O \setminus e$ is good and if~$e$ is bioriented in $\O$ then~$\O/e$ is good. So assume further that~$e^{\delta} \in \O$ and~$e^{-\delta} \notin \O$ for some $\delta \in \{+,-\}$. Now suppose to the contrary that both~$\O \setminus e$ and~$\O/e$ are bad. Note as in the proof of Lemma~\ref{lem:goodfourup} that~$\O \setminus e$ is certainly good for~$\G \setminus e$ with respect to the cycle property; and~$\O/e$ is certainly good for $\G/e$ with respect to the cut property. So~$\O\setminus e$ is bad with respect to the cut property and~$\O/e$ is bad with respect to the cycle property. Then by condition~\ref{cond:tuttebador} we can extend~$\O \setminus e$ to a bad fourientation of $\G$ with respect to the cut property by orienting $e$ in a certain direction, and it must be that $\O' := (\O \setminus e) \cup \{e^{-\delta}\}$ is this extension as we know $\O$ is good. Analogously we know that $\O' = (\O / e) \cup \{e^{-\delta}\}$ is bad for $\G$ with respect to the cycle property. But that means that there must be some cut $\Cu$ of $\G$ so that~$\O'_{\Cu}$ is bad with respect to the cut property. It must be that~$e \in E(\Cu)$ or else $\O_{\Cu}$ would also be bad. And it must be that there is a way of directing $\Cu$ to become a potential cut $\OCu$ of~$\O'$ or else the contraction $\O'_{\Cu}$ could not be bad by condition~\ref{cond:tuttepot}. Analogously we can find a potential cycle $\OCy$ of $\O'$ with $e \in E(\OCy)$. However, this contradicts Lemma~\ref{lem:disjoint} which says the edge sets of potential cuts and potential cycles are disjoint.
\end{proof}

The deletion-contraction argument we use to count the good fourientations with respect to some Tutte cut-cycle property will actually work for fourientations that are weighted by the number of oriented, bioriented, and unoriented edges they contain. Thus we define the following chromatic extension of fourientations.

\begin{definition}
For $k, l, m \in \mathbb{Z}_{\geq 0}$, a \emph{(k,l,m)-fourientation} is obtained from a fourientation by assigning each of the oriented edges one of $k$ colors, each of the unoriented edges one of $l$ colors, and each of the bioriented edges one of $m$ colors.  If a variable is set equal to zero, we naturally require that the associated fourientations have no edges with the corresponding fourientation type. We say a~$(k,l,m)$-fourientation is good with respect to a fourientation property if the underlying fourientation is good.
\end{definition}

\begin{definition}
We define the \emph{bad isthmus set} $X \subseteq \{ \emptyset, \{-\}, \{+\}\}$ of a Tutte cut property in the following way: let $(B_1,\Oref,<)$ be the unique (up to isomorphism) ordered, oriented graph whose underlying graph is the banana graph~$B_1$ and let~$e$ denote the unique edge of $B_1$; define $X$ to be the set of all~$S$ such that $\{e^{\varepsilon}\colon \varepsilon \in S\}$ is a bad fourientation for $(B_1,\Oref,<)$ with respect to the cut property. (Observe that by~\ref{cond:tuttepot} it is impossible for $\{e^{+},e^{-}\}$ to be bad with respect to a Tutte cut property.) For a Tutte cycle property we define its \emph{bad loop set} $Y  \subseteq \{ \{-\}, \{+\},\{-,+\}\}$ analogously in terms of the fourientations that are bad for~$C_1$. The bad isthmus (loop) set of a Tutte cut-cycle property is the bad isthmus (loop) set of its underlying cut (cycle) property.
\end{definition}

\begin{thm}\label{thm:main}
For a fourientation $\O$, let $\O^{o}$ denote the set of edges that are oriented in $\O$, $\O^{u}$ the set of unoriented edges, and $\O^{b}$ the set of bioriented edges. Fix a Tutte cut-cycle property with bad isthmus set $X$ and bad loop set~$Y$. Then for $\G$, an ordered, oriented graph whose underlying graph $G$ has $n$ vertices and cyclomatic number~$g$, we have
\begin{equation} \label{eqn:tutteeval}
\sum_{\O} k^{|\O^{o}|} l^{|\O^{u}|} m^{|\O^{b}|} = (k+m)^{n-1}(k+l)^{g} T_G\left(\frac{x_0}{k+m},\frac{y_0}{k+l}\right)
\end{equation}
where the sum is over all good fourientations $\O$ of $\G$, and where
\begin{align*}
x_0 &:= \delta(\{+\} \notin X)k + \delta(\{-\} \notin X)k + \delta(\emptyset \notin X)l + m  \\
y_0 &:= \delta(\{+\} \notin Y)k + \delta(\{-\} \notin Y)k  + l + \delta(\{+,-\} \notin Y)m
\end{align*}
and $\delta(P)$ is $1$ if $P$ is true and $0$ if $P$ is false. In other words, the number of good $(k,l,m)$-fourientations of $\G$ is given by~(\ref{eqn:tutteeval}).
\end{thm}

\begin{proof}
Fix a Tutte cut-cycle property. For an ordered, oriented graph $\G$, define 
\[ T(\G) := \{\O\colon \textrm{$\O$ is a good fourientation of $\G$ with respect to our property} \}.\]
For a set $O$ of fourientations, define $\widehat{O} := \sum_{\O \in O} k^{|\O^{o}|} l^{|\O^{u}|} m^{|\O^{b}|}$. Let $f(G) := \widehat{T(\G)}$ where $\G$ has $G$ as its underlying graph. The proof will also establish inductively that~$f(G)$ is well-defined, that is, that this generating function does not depend on what reference orientation and edge order we give $G$.

Let $G$ be a graph and set $\G := (G,\Oref,<)$ for arbitrary $\Oref$ and $<$. If $G$ has no edges then certainly $f(G) = 1$. So assume $G$ has an edge and let $e$ be the maximum edge of $\G$. If $e$ is an isthmus, then the simple cycles of $G$ are the same as those of $G\setminus e$ and there is one additional simple cut: namely, the cut that has $e$ as its only edge. So any good fourientation $\O$ of $\G\setminus e$ extends to a good fourientation $\O \cup \{e^{\varepsilon}\colon \epsilon \in S\}$ of~$\G$ as long as $S \notin X$, and we get all good fourientations of $\G$ this way. Therefore if $e$ is an isthmus then $f(G) = x_0 f(G\setminus e)$. Similarly if $e$ is a loop then $f(G) = y_0f(G/e)$. So from now on assume that $e$ is neither an isthmus nor loop. 

We want to show $f(G) = (k+l)f(G \setminus e) +(k+m)f(G/e)$ from which the result follows via Theorem~\ref{thm:gentutte}. For $\O \in T(\G\setminus e)$, set $De(\O) := \{\O' \in T(\G)\colon \O' \setminus e = \O\} $. Similarly, for~$\O \in T(\G/e)$, set $C\O^{o} := \{\O' \in T(\G)\colon \O'/e = \O\}$. Set~$De := \bigcup_{\O \in T(\G\setminus e)} De(\O)$ and~$Co := \bigcup_{\O \in T(\G/e)}C\O^{o}$. We claim that for $\O \in T(\G\setminus e)$,
\begin{enumerate}[label=(\roman*)]
\item either $De(\O) \subseteq De \setminus Co$ and $De(\O) = \{\O, \O \cup \{e^{\varepsilon}\}\}$ for some $\varepsilon \in \{-,+\}$;
\item or $De(\O) \subseteq De \cap Co$ and $De(\O) = \{\O, \O \cup \{e^\mathrm{+}\}, \O \cup \{e^\mathrm{-}\}, \O \cup \{e^{+},e^{-}\}\}$.
\end{enumerate}
From this claim it follows that $f(G\setminus e) = \frac{\widehat{De \setminus Co}}{k+l}+ \frac{\widehat{De \cap Co}}{2k+l+m}$. So let us prove the claim. First of all, by Lemma~\ref{lem:goodfourup} we know that $\O \in T(\G)$ which agrees with our claim. Assume first~$ \O \cup \{e^{+},e^{-}\} \in De(\O)$; we need to show~$\O \in T(\G/e)$ and~$\O \cup \{e^{\varepsilon}\} \in T(\G)$ for any~$\varepsilon \in \{\mathrm{-},\mathrm{+}\}$. For any simple cut $\Cu$ of~$\G/e$ we have a corresponding simple cut~$\Cu'$ of $\G\setminus e$ so that $(\G/e)_{\Cu}$ and $(\G\setminus e)_{\Cu'}$ are isomorphic. Thus $\O$ is good for~$\G/e$ with respect just to the cut property. Because~$ \O \cup \{e^{+},e^{-}\}$ is good for $\G$ with respect to the cycle property, by condition~\ref{cond:tutteun} we get that~$(\O \cup \{e^{+},e^{-}\}) / e=\O$ is also good for $\G/e$ with respect to the cycle property and therefore~$\O \in T(\G/e)$. Using condition~\ref{cond:tuttegoodor} with respect to the cut and cycle properties gives~$\O \cup \{e^{+}\} \in T(\G)$ and $\O \cup \{e^{-}\} \in T(\G)$ so we are done. Next assume~$\O \cup \{e^{+},e^{-}\} \notin De(\O)$; we need to show~$\O \notin T(\G/e)$ and~$\O \cup \{e^{\varepsilon}\} \in T(\G)$ for a unique $\varepsilon \in \{\mathrm{-},\mathrm{+}\}$. Lemma~\ref{lem:goodfourup} gives~$\O \notin T(\G/e)$. We know that~$\O \cup \{e^{+}\}$ and $\O \cup \{e^{-}\}$ are good for $\G$ with respect to just the cut property by condition~\ref{cond:tuttegoodor}. As before, we know~$\O$ is good for $\G/e$ with respect to the cut property. So it must be that~$\O$ is bad for $\G/e$ with respect to the cycle property. Condition~\ref{cond:tutteun} says that since~$\O$ is not good for $\G/e$ with respect to the cycle property, exactly one of $\O \cup \{e^{+}\}$ or~$\O \cup \{e^{-}\}$ is good for $\G$ with respect to the cycle property. So the claim is proved.

We can similarly show that $f(G / e) = \frac{\widehat{Co \setminus De}}{k+m}+ \frac{\widehat{De\cap Co}}{2k+l+m}$. Then by Lemma~\ref{lem:goodfourdown},
\begin{align*}
f(G) &= \widehat{De \setminus Co}+ \widehat{Co \setminus De}+ \widehat{De \cap Co} \\
&= (k+l)f(G\setminus e) + (k+m)f(G/e),
\end{align*}
thus completing the proof.
\end{proof}

\begin{remark}
Theorem~\ref{thm:main} says that there are the same number of good fourientations of $(G,\Oref^1,<^1)$ and $(G,\Oref^2,<^2)$ with respect to some fixed Tutte cut-cycle property. It would be interesting to find a simple bijection between these sets of fourientations; that is, it would be interesting to understand how the set of good fourientations changes as we modify the reference orientation and total order.
\end{remark}

\subsection{Min-edge classes}
A priori it is not clear that there are any nontrivial Tutte cut properties. We will now show that there exist Tutte cut properties for all bad isthmus sets $X \subseteq \{ \emptyset, \{+\}, \{-\}\}$. Moreover, we will show that the Tutte cut properties are almost classified by $X$ (specifically, for each choice of $X$ there are one, two, or three Tutte cut properties with bad isthmus set~$X$). Of course the situation is analogous for Tutte cycle properties.

\begin{definition}
A \emph{min-edge cut (cycle) property} is defined by~$X \subseteq \{\emptyset,\{-\},\{+\}\}$  ($\{ \{-\},\{+\}, \{-,+\}\}$) and~$\delta \in \{+,-\}$.  A potential cut $\OC$ (cycle $\OC$) of a fourientation~$\O$ of an ordered, oriented graph $\G$ is \emph{bad} with respect to the min-edge cut (cycle) property defined by~$(X,\delta)$ if it satisfies both of the following conditions:
\begin{enumerate}[label=(\roman*),ref=(\roman*)]
\item $\{\varepsilon\colon e_{\mathrm{min}}^{\varepsilon} \in \O\} = S$ for some $S \in X$, where $e_{\mathrm{min}}$ is the minimum edge in $E(\OC)$ ;
\item if $e_{\mathrm{min}}$ is unoriented (bioriented) in $\O$ then $e_{\mathrm{min}}^{\delta} \in \mathbb{E}(\OC)$. \label{cond:minhalf}
\end{enumerate}
If the potential cut (cycle) is not bad, then it is \emph{good}. A fourientation $\O$  of $\G$ is good with respect to the min-edge cut (cycle) property defined by~$(X,\delta)$ if and only if all of its potential cuts (cycles) are good.
\end{definition}

A heuristic explanation for the emphasis on the statuses of minimum edges in potential cuts is that in checking whether a cut is good or bad with respect to a Tutte cut property we repeatedly peel off maximal edges until we reduce to the base case where the cut's minimum edge becomes an isthmus. The point of $\delta$ is that in order to satisfy~\ref{cond:tuttebador} we want one of the ways of extending a bad cut by an oriented edge to be bad and the other way to be good: if the bad cut consists only of unoriented edges then both ways of extending it by an oriented edge still yield potential cuts and so could potentially be bad by~\ref{cond:tuttepot}; in this case $\delta$ tells us which of these in fact is bad.

There are $12$ essentially different min-edge cut properties because the choice of $\delta$ is relevant only if~$\emptyset \in X$. Let $-X$ denote the set we get by swapping $+ \leftrightarrow -$ and define~$-\O$ similarly. Clearly $\O$ is a good fourientation of $\G$ with respect to the min-edge cut property defined by $(X,\delta)$ if and only if $-\O$ is good with respect to~$(-X,-\delta)$. So we may as well fix $\delta = -$ and thus reduce the list to the following eight properties, which we call the \emph{min-edge cut classes} of fourientations:
\begin{enumerate}
\setlength{\itemindent}{-1em}
\item {\bf Cut general} ($X = \emptyset$):  There are no restrictions on potential cuts.
\item {\bf Cut directed} ($X = \{\emptyset\}$):   For each potential cut, if the minimum edge of the cut is unoriented then cut contains some oriented edge directed in agreement with the reference orientation of this minimum edge.
\item {\bf Cut negative} ($X = \{ \{-\} \} $):  The minimum edge in each potential cut is unoriented or is oriented in agreement with its reference orientation.
\item {\bf Cut positive} ($X = \{\{+\}\} $):  The minimum edge in each potential cut is unoriented or is oriented in disagreement with its reference orientation.
\item {\bf Cut connected} ($X = \{\emptyset,\{-\}\}$):  Each potential cut contains an oriented edge directed in agreement with the reference orientation of the minimum edge in the cut.
\item {\bf Cut co-connected} ($X = \{\emptyset,\{+\}\}$):  For each potential cut, either the minimum edge of the cut is unoriented and the cut contains an oriented edge directed in agreement with the reference orientation of this minimum edge, or the minimum edge is oriented in disagreement its the reference orientation.
\item {\bf Cut neutral} ($X = \{\{-\},\{+\}\}$):  The minimum edge in each potential cut is unoriented.
\item {\bf Cut internal} ($X = \{\emptyset,\{-\},\{+\}\}$):  The minimum edge in each potential cut is unoriented and the cut contains an oriented edge directed in agreement with the reference orientation of this minimum edge.
\end{enumerate}
We omit the description of the \emph{min-edge cycle classes} which are exactly analogous (with ``cycle external'' being dual to cut internal). Observe that the poset of the above eight classes ordered by containment is of course isomorphic to the Boolean lattice on three elements. A \emph{min-edge class} of fourientations is the intersection of a min-edge cut class and min-edge cycle class. Note that an arbitrary intersection of min-edge classes remains a min-edge class.

\begin{thm} \label{thm:minaretutte}
Any  min-edge cut (cycle) property is a Tutte cut (cycle) property.
\end{thm}

\begin{proof} Let us work with cut properties; the cycle properties are exactly analogous. First let us prove that a min-edge cut property is a Tutte cut property. Fix some min-edge cut property. Clearly the property is defined in an isomorphism invariant way and so it is indeed a fourientation property. A potential simple cut being good with respect to our min-edge cut property is the same as the corresponding contraction to a banana graph being good (and in light of \ref{cond:tuttepot} we only care about potential cuts). So certainly if a fourientation is good, its contractions to its simple cuts are good. Conversely, assume the fourientation $\O$ is bad. Then there is a bad potential cut~$\OCu$ for~$\O$. In fact we have~$\mathbb{E}(\OCu) = \mathbb{E}(\OCu_1) \sqcup \cdots \sqcup \mathbb{E}(\OCu_k)$ for potential cuts~$\OCu_i$ whose underlying undirected cuts~$\Cu_i$ are simple. So if~$e_{\mathrm{min}}^{\delta} \in \mathbb{E}(\OCu)$ with~$e_{\mathrm{min}}$ being the minimum edge of~$E(\OCu)$ then $e_{\mathrm{min}}^{\delta} \in \mathbb{E}(\OCu_i)$ for some $i$, which means $\O_{\Cu_i}$ is bad. Thus our min-edge cut property is indeed a cut property. What remains to check are the conditions~\ref{cond:tuttepot} and~\ref{cond:tuttemax}. Condition~\ref{cond:tuttepot} holds trivially. Now let us deal with condition~\ref{cond:tuttemax}: so let $e$ be the maximal edge of $\G$ with $e$ neither an isthmus nor a loop, and let $\O$ be a fourientation of $\G\setminus e$. Condition~\ref{cond:tuttegoodor} holds because in this case $e$ cannot be the minimum edge in any potential cut, so any good potential cuts of $\O$ which it becomes a part of remain good in $\O \cup \{e^{\pm}\}$. Condition~\ref{cond:tutteun} holds for much the same reason: since $e$ is not the minimum edge in any potential cut, any good potential cuts for~$\O$ which it becomes a part of remain good potential cuts in~$\O$ when considered as a fourientation of $\G$ and any bad potential cuts remain bad potential cuts. 

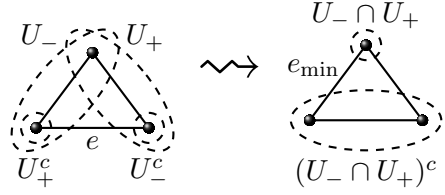
\begin{figure}
\begin{tikzpicture}
	\SetFancyGraph
	\Vertex[NoLabel,x=0,y=1]{v_2}
	\Vertex[NoLabel,x=-0.75,y=0]{v_1}
	\Vertex[NoLabel,x=0.75,y=0]{v_3}
	\Edges[style={thick}](v_1,v_2)
	\Edges[style={thick}](v_2,v_3)
	\Edges[style={thick}](v_1,v_3)
	\draw[rotate=-40,dashed,thick] (-0.6,0.15) ellipse (.4 and 1);
	\draw[dashed,thick] (-0.75,0) ellipse (.2 and .2);
	\draw[rotate=40,dashed,thick] (0.6,0.15) ellipse (.4 and 1);
	\draw[dashed,thick] (0.75,0) ellipse (.2 and .2);
	\node at (0,-0.2) {$e$};
	\node at (-0.75,-0.6) {$U^c_{+}$};
	\node at (-0.7,1.2) {$U_{-}$};
	\node at (0.75,-0.6) {$U^c_{-}$};
	\node at (0.7,1.2) {$U_{+}$};
	\node at (0,-0.7){};
\end{tikzpicture} \parbox{0.4in}{\begin{center} \vspace{-1.25in} \Huge $\leadsto$\end{center}} \begin{tikzpicture}
	\SetFancyGraph
	\Vertex[NoLabel,x=0,y=1]{v_2}
	\Vertex[NoLabel,x=-0.75,y=0]{v_1}
	\Vertex[NoLabel,x=0.75,y=0]{v_3}
	\Edges[style={thick}](v_1,v_2)
	\Edges[style={thick}](v_2,v_3)
	\Edges[style={thick}](v_1,v_3)
	\draw[dashed,thick] (0,1) ellipse (.2 and .2);
	\draw[dashed,thick] (0,0) ellipse (1 and .4);
	\node at (-0.7,0.7) {$e_{\mathrm{min}}$};
	\node at (0,1.4) {$U_{-} \cap U_{+}$};
	\node at (0,-0.7) {$(U_{-} \cap U_{+})^c$};
\end{tikzpicture}
\caption{A visual aid for the proof of Theorem~\ref{thm:minaretutte}.} \label{fig:badglue}
\end{figure}

The least obvious condition is~\ref{cond:tuttebador}. First of all, if $\O$ is bad for $\G \setminus e$ then one of~$\O \cup \{e^{+}\}$ or $\O \cup \{e^{-}\}$ is bad because $\G \setminus e$ has at least one bad potential cut and so if we orient $e$ as $e^{\pm}$ in a way consistent with this cut it will remain a bad potential cut in~$\O \cup \{e^{\pm}\}$. In order to complete the proof that \ref{cond:tuttebador} holds, we claim that if~$\O \cup \{e^{+},e^{-}\}$ is good for $\G$ then at least one of~$\O \cup\{e^{+}\}$ or~$\O \cup\{e^{-}\}$ is good. Suppose that to the contrary $\O \cup \{e^{+},e^{-}\}$ is good but $\O \cup\{e^{+}\}$ and $\O \cup\{e^{-}\}$ are both bad. Since~$\O \cup \{e^{+},e^{-}\}$ is good, it cannot be that there is a bad potential cut~$\OCu$ of $\O \cup \{e^{\pm}\}$ with $e \notin E(\OCu)$. So it must be that there is a bad potential cut~$\OCu_{+} = (U_{+},U_{+}^c)$ of $\O \cup \{e^{+}\}$ and a bad potential cut~$\OCu_{-} = (U_{-},U_{-}^c)$ of $\O \cup \{e^{-}\}$ with $e \in E(\OCu_{+}) \cap E(\OCu_{-})$. The idea, as depicted in Figure~\ref{fig:badglue}, is to glue the cuts~$\OCu_{+}$ and $\OCu_{-}$ together to find a bad potential cut which does not involve $e$. Let $e_{\mathrm{min}}$ be the minimum edge of $E(\OCu_{+}) \cup E(\OCu_{-})$. By supposition that $e$ is not an isthmus, we have $e \neq e_{\mathrm{min}}$. Suppose by symmetry that~$e_{\mathrm{min}}^{\delta} \in \mathbb{E}(\OCu_{+})$ for some $\delta \in \{+,-\}$. We claim that $e_{\mathrm{min}}^{\delta} \neq (u,v)$ with $u \in U_{+} \cap U_{-}^{c}$ and $v \in U_{+}^{c} \cap U_{-}$. Suppose to the contrary. Then $e_{\mathrm{min}} \in E(\OCu_{+}) \cap E(\OCu_{-})$, which means that in order for $\OCu_{+}$ to be a potential cut of $\O \cup \{e^{+}\}$ and $\OCu_{-}$ a potential cut of $\O \cup \{e^{-}\}$ we need $e_{\mathrm{min}}$ to be unoriented in~$\O$. But then we have $e_{\mathrm{min}}^{\delta} \in \mathbb{E}(\OCu_{+})$ and~$e_{\mathrm{min}}^{-\delta} \in \mathbb{E}(\OCu_{-})$, so by part~\ref{cond:minhalf} of the min-edge cut definition we cannot have that~$\OCu_{+}$ and $\OCu_{-}$ are both bad potential cuts, a contradiction. So indeed $e_{\mathrm{min}}^{\delta} \neq (u,v)$ for any $u \in U_{+} \cap U_{-}^{c}$ and $v \in U_{+}^{c} \cap U_{-}$. One consequence of this is that~$U_{+} \neq U_{-}^{c}$. So at least one of $U_{+} \cap U_{-}$ or $U_{+}^{c} \cap U_{-}^{c}$ is nonempty. And on the other hand if $e = \{w,x\}$ then~$\{w,x\} \subseteq (U_{+} \cap U_{-})^c \cap (U_{+}^{c} \cap U_{-}^{c})^c$, so $(U_{+} \cap U_{-})^c$ and $(U_{+}^{c} \cap U_{-}^{c})^c$ are both nonempty. Thus if we define~$\OCu' := (U_{+} \cap U_{-},(U_{+} \cap U_{-})^c)$ and~$\OCu'' := (U_{+}^{c} \cap U_{-}^{c},(U_{+}^{c} \cap U_{-}^{c})^c)$ at least one of these must genuinely be a directed cut. Our discussion of $e_{\mathrm{min}}$ also establishes that~$e_{\mathrm{min}}^{\delta} \in \mathbb{E}(\OCu')$ or~$e_{\mathrm{min}}^{\delta} \in \mathbb{E}(\OCu'')$. Whichever of~$\OCu'$ or~$\OCu''$ it belongs to is a potential cut of $\O \cup \{e^{+},e^{-}\}$ because we have~$e \notin E(\OCu') \cap E(\OCu'')$. But then one of~$\OCu'$ or $\OCu''$ is a bad potential cut of~$\O \cup \{e^{+},e^{-}\}$, contradicting our assumption that~$\O \cup \{e^{+},e^{-}\}$ is good. \end{proof}

In order to give a near converse to Theorem~\ref{thm:minaretutte} and to completely classify Tutte cut properties we need to introduce two anomalous properties: {\bf cut weird} and {\bf cut co-weird}. The cut weird fourientations of an ordered, oriented graph are those such that each potential cut contains at least one oriented edge and the minimum oriented edge in the cut is oriented in agreement with its reference orientation. The cut co-weird fourientations are those such that each potential cut contains at least one oriented edge and the minimum oriented edge in the cut is oriented in disagreement with its reference orientation.

\begin{thm} \label{thm:tuttearemin}
Any Tutte cut property is either a min-edge cut property or is cut weird or cut co-weird. There is a completely analogous classification for Tutte cycle properties.
\end{thm}

\begin{proof} 
Again we work with the cut case. Fix some Tutte cut property. Because it is a cut property, this property is determined by the values it takes on ordered, oriented banana graphs. It is not hard to see that if our Tutte property agrees with some min-edge cut property~$(X,\delta)$ on all ordered, oriented banana graphs then it is agrees with~$(X,\delta)$ on all graphs (here we again use the fact that a cut decomposes into simple cuts). Our goal is to find~$(X,\delta)$. Clearly we should define $X$ to be bad isthmus set of our Tutte cut property. In order to define $\delta$ we need to consider some small banana graphs. Let us view the banana graph $B_n$ as having vertex set~$V(B_n) := \{u,v\}$ and edge set~$E(B_n) := \{e_1,\ldots,e_n\}$ where $e_1 := \cdots := e_n := \{u,v\}$. Define the edge order~$<$ by~$e_1 <\cdots < e_n$. If $\emptyset \notin X$, then we define~$\delta$ arbitrarily. If $\emptyset \in X$, then we define~$\delta$ as follows: define a reference orientation~$\Oref^2$ by~$\Oref^2(e^{+}_1) := \Oref^2(e^{+}_2) := (u,v)$; then let~$\delta \in \{+,-\}$ be so that~$\O^2 := \{e_2^{\delta}\}$ is a bad fourientation of~$(B_2,\Oref^2,<)$. We need to check that our property agrees with the min-edge cut property $(X,\delta)$ on all banana graphs. So let $(B_n,\Oref,<)$ be an ordered, oriented banana graph and assume~$n > 1$ since we know our Tutte property agrees with $(X,\delta)$ for $n=1$.  Let $\O$ be a fourientation of $(B_n,\Oref,<)$. If~$\O$ has any bioriented edges, we know it is good by condition~\ref{cond:tuttepot} because it has no potential cuts and this agrees with $(X,\delta)$. So now assume $\O$ has no bioriented edges. If~$\O \setminus e_n$ is good for~$(B_n,\Oref,<) \setminus e_n$ then we know by conditions~\ref{cond:tuttegoodor} and~\ref{cond:tutteun} that~$\O$ is good no matter how $e_n$ is fouriented, which again agrees with~$(X,\delta)$. If $\O \setminus e_n$ is bad but contains at least one oriented edge, then we know by conditions~\ref{cond:tuttepot},~\ref{cond:tuttebador}, and~\ref{cond:tutteun} that~$\O$ is good if and only if $e$ is oriented to disagree with that oriented edge and make it so that~$\O$ has no potential cuts. This agrees with $(X,\delta)$. So finally assume that~$\O \setminus e_n$ is bad and contains no oriented edges. Note by repeated application of~\ref{cond:tutteun} that this is possible only if $\emptyset \in X$. Certainly by~\ref{cond:tutteun} if $e_n^{+},e_n^{-} \notin \O$ then $\O$ is bad; so the status of $\O$ is only at issue if $e_n^{\varepsilon} \in \O$ for some~$\varepsilon \in \{-,+\}$. We claim that in this case the status of $\O$ is still consistent with~$(X,\delta)$ unless we are in an exceptional case that we will address at the end. 

\begin{figure}
\def\arraystretch{1.5}
\begin{tabular}{c | c " c | c | c}
\multicolumn{5}{c}{ \parbox{4in} {\centering Assume throughout \\ by symmetry $\delta = -$: \\ \vspace{.1cm} 
\begin{tikzpicture}
	\SetFancyGraph
	\Vertex[NoLabel,x=0,y=1]{v_2}
	\Vertex[NoLabel,x=0,y=0]{v_1}
	\Edges[style={thick,bend left,->--}](v_1,v_2)
	\Edges[style={thick,bend left,->,>=mytip}](v_2,v_1)
	\Edges[style={dash pattern=on 0pt off 100pt,thick,bend right,->--}](v_1,v_2);
\end{tikzpicture} bad \vspace{.2cm}} }  \\ \hline
%%%%%%%%%%%%%%%%%
\multicolumn{2}{c"}{ \parbox{2in} { \vspace{.1cm} \centering Case I: \\ \vspace{.1cm} 
\begin{tikzpicture}
	\SetFancyGraph
	\Vertex[NoLabel,x=0,y=1]{v_2}
	\Vertex[NoLabel,x=0,y=0]{v_1}
	\Edges[style={thick,bend left,->--}](v_1,v_2)
	\Edges[style={thick,bend left,->,>=mytip}](v_2,v_1)
	\Edges[style={dash pattern=on 0pt off 100pt,thick,bend left,->--}](v_2,v_1);
\end{tikzpicture} bad \vspace{.2cm} }} &
%%%%%%%%%%%%%%%%%
 \multicolumn{3}{c}{ \parbox{2in} { \vspace{.1cm} \centering Case II: \\ \vspace{.1cm} 
\begin{tikzpicture}
	\SetFancyGraph
	\Vertex[NoLabel,x=0,y=1]{v_2}
	\Vertex[NoLabel,x=0,y=0]{v_1}
	\Edges[style={thick,bend left,->--}](v_1,v_2)
	\Edges[style={thick,bend left,->,>=mytip}](v_2,v_1)
	\Edges[style={dash pattern=on 0pt off 100pt,thick,bend left,->--}](v_2,v_1);
\end{tikzpicture} good \vspace{.2cm} }} \\ \hline
%%%%%%%%%%%%%%%%%
Subclaim~\ref{sub:b3}: &  Subclaim~\ref{sub:caseimain}: &  Subclaim~\ref{sub:minusinx}: & Subclaim~\ref{sub:plusnotinx}: & Subclaim~\ref{sub:caseiimain}: \\
%%%%%%%%%%%%%%%%%
\parbox{0.75in} { \centering 
\begin{tikzpicture}
	\SetFancyGraph
	\Vertex[NoLabel,x=0,y=1]{v_2}
	\Vertex[NoLabel,x=0,y=0]{v_1}
	\Edges[style={thick,bend left=40,->--}](v_1,v_2)
	\Edges[style={thick,->--}](v_2,v_1)
	\Edges[style={dash pattern=on 0pt off 100pt,thick,bend left=40,->--}](v_2,v_1);
	\Edges[style={thick,bend left=40,->,>=mytip}](v_2,v_1)
\end{tikzpicture} good  \\ $\Downarrow$ \\  \begin{tikzpicture}
	\SetFancyGraph
	\Vertex[NoLabel,x=0,y=1]{v_2}
	\Vertex[NoLabel,x=-0.5,y=0]{v_1}
	\Vertex[NoLabel,x=0.5,y=0]{v_3}
	\Edges[style={thick,bend left,->--}](v_1,v_2)
	\Edges[style={thick,->,>=mytip}](v_1,v_2)
	\Edges[style={dash pattern=on 0pt off 100pt,thick,->--}](v_2,v_1);
	\Edges[style={thick,bend right,->--}](v_2,v_3)
	\Edges[style={thick,->--}](v_3,v_2)
	\Edges[style={thick,->,>=mytip, bend right=60}](v_3,v_2)
	\Edges[style={dash pattern=on 0pt off 100pt,thick,bend right=60,->--}](v_3,v_2)
	\node at (-0.6,0.6) {\tiny $e_1$};
	\node at (-0.2,0.25) {\tiny $e_3$};
	\node at (0.15,0.1) {\tiny $e_2$};
	\node at (0.47,0.4) {\tiny $e_4$};
	\node at (0.8,0.7) {\tiny $e_5$};	
\end{tikzpicture} good \\ $\Downarrow$ \\  \begin{tikzpicture}
	\SetFancyGraph
	\Vertex[NoLabel,x=0,y=1]{v_2}
	\Vertex[NoLabel,x=0,y=0]{v_1}
	\Edges[style={thick,bend left=80,->--}](v_1,v_2)
	\Edges[style={thick,bend right=35,->--}](v_2,v_1)
	\Edges[style={thick,->,>=mytip}](v_1,v_2)
	\Edges[style={dash pattern=on 0pt off 100pt,thick,->--}](v_2,v_1);
	\Edges[style={thick,->--,bend right=35}](v_1,v_2)
	\Edges[style={thick,->,>=mytip, bend right=80}](v_1,v_2)
	\Edges[style={dash pattern=on 0pt off 100pt,thick,bend right=80,->--}](v_1,v_2)
\end{tikzpicture} good \\ $\Downarrow$ \\  \begin{tikzpicture}
	\SetFancyGraph
	\Vertex[NoLabel,x=0,y=1]{v_2}
	\Vertex[NoLabel,x=0,y=0]{v_1}
	\Edges[style={thick,bend left=40,->--}](v_1,v_2)
	\Edges[style={thick,->--}](v_2,v_1)
	\Edges[style={thick,->,>=mytip,bend right=40}](v_1,v_2)
	\Edges[style={dash pattern=on 0pt off 100pt,thick,->--, bend left=40}](v_2,v_1);
\end{tikzpicture} good  \\ $\Rightarrow \Leftarrow$} & 
%%%%%%%%%%%%%%%%%
\parbox{0.75in} { \centering 
\begin{tikzpicture}
	\SetFancyGraph
	\Vertex[NoLabel,x=0,y=1]{v_2}
	\Vertex[NoLabel,x=0,y=0]{v_1}
	\Edges[style={thick,bend left=40,->--}](v_1,v_2)
	\Edges[style={thick,->--}](v_2,v_1)
	\Edges[style={dash pattern=on 0pt off 100pt,thick,bend left=40,->--}](v_2,v_1);
	\Edges[style={thick,bend left=40,->,>=mytip}](v_2,v_1)
\end{tikzpicture} bad  \\ $\Downarrow$ \\  
\begin{tikzpicture}
	\SetFancyGraph
	\Vertex[NoLabel,x=0,y=1]{v_2}
	\Vertex[NoLabel,x=0,y=0]{v_1}
	\tikzset{VertexStyle/.style = {shape = circle,fill = black,minimum size = 0pt,inner sep=0pt}}
	\Vertex[NoLabel,x=0.75,y=0.5]{v_3}
	\Edges[style={thick,bend left=30,->--}](v_1,v_2)
	\Edges[style={thick,->--}](v_2,v_1)
	\Edges[style={dash pattern=on 0pt off 100pt,thick,bend left=30,->--}](v_2,v_1);
	\Edges[style={thick,bend left=30,->,>=mytip}](v_2,v_1)
	\Edges[style={thick,bend left=50}](v_2,v_3)
	\Edges[style={thick,bend left=50,->,>=mytip}](v_3,v_1)
	\node at (0.55,0.5) {\footnotesize $\cdots$};
\end{tikzpicture} bad \\ $\Downarrow$ \\  
\begin{tikzpicture}
	\SetFancyGraph
	\Vertex[NoLabel,x=0,y=1]{v_2}
	\Vertex[NoLabel,x=0,y=0]{v_1}
	\Vertex[NoLabel,x=1,y=1]{v_3}
	\Edges[style={thick,bend left=50,->--}](v_1,v_2)
	\node at (0.05,0.5) {\footnotesize $\cdots$};
	\Edges[style={thick,bend left=50,->,>=mytip}](v_2,v_1)
	\Edges[style={thick,->--,bend right=30}](v_2,v_3)
	\Edges[style={dash pattern=on 0pt off 100pt,thick,bend left=30,->--}](v_2,v_3);
	\Edges[style={thick,bend left=30,->,>=mytip}](v_2,v_3)
	\draw[dashed,thick] (0.5,1) ellipse (.7 and .5);
	\node[color=green] at (0.7,1.5) {\LARGE $\checkmark$};
\end{tikzpicture} bad \\ $\Downarrow$ \\  
\begin{tikzpicture}
	\SetFancyGraph
	\Vertex[NoLabel,x=0,y=1]{v_2}
	\Vertex[NoLabel,x=0,y=0]{v_1}
	\Edges[style={thick,bend left=50,->--}](v_1,v_2)
	\node at (0.05,0.5) {\footnotesize $\cdots$};
	\Edges[style={thick,bend left=50,->,>=mytip}](v_2,v_1)
\end{tikzpicture} bad } & 
%%%%%%%%%%%%%%%%%
\parbox{0.75in} { \centering 
\begin{tikzpicture}
	\SetFancyGraph
	\Vertex[NoLabel,x=0,y=1]{v_2}
	\Vertex[NoLabel,x=0,y=0]{v_1}
	\Edges[style={dash pattern=on 0pt off 100pt,thick,->--}](v_1,v_2);
	\Edges[style={thick,->,>=mytip}](v_2,v_1)
\end{tikzpicture} good  \\ $\Downarrow$ \\  \begin{tikzpicture}
	\SetFancyGraph
	\Vertex[NoLabel,x=0,y=1]{v_2}
	\Vertex[NoLabel,x=0.5,y=0]{v_1}
	\Vertex[NoLabel,x=-0.5,y=0]{v_3}
	\Edges[style={dash pattern=on 0pt off 100pt,thick,->--}](v_1,v_2);
	\Edges[style={thick,->,>=mytip}](v_2,v_1)
	\Edges[style={bend right,thick,->--}](v_2,v_3);
	\Edges[style={thick,->,>=mytip,bend left}](v_2,v_3)
	\Edges[style={dash pattern=on 0pt off 100pt,bend left,thick,->--}](v_2,v_3);
	\node at (-0.5,0.8) {\tiny $e_1$};
	\node at (0.05,0.25) {\tiny $e_3$};
	\node at (0.35,0.8) {\tiny $e_2$};
\end{tikzpicture} good \\ $\Downarrow$ \\  \begin{tikzpicture}
	\SetFancyGraph
	\Vertex[NoLabel,x=0,y=1]{v_2}
	\Vertex[NoLabel,x=0,y=0]{v_1}
	\Edges[style={dash pattern=on 0pt off 100pt,thick,->--}](v_1,v_2);
	\Edges[style={thick,->,>=mytip}](v_2,v_1)
	\Edges[style={bend right=40,thick,->--}](v_2,v_1);
	\Edges[style={thick,->,>=mytip,bend left=40}](v_2,v_1)
	\Edges[style={dash pattern=on 0pt off 100pt,bend left=40,thick,->--}](v_2,v_1);
\end{tikzpicture} good \\ $\Downarrow$ \\  \begin{tikzpicture}
	\SetFancyGraph
	\Vertex[NoLabel,x=0,y=1]{v_2}
	\Vertex[NoLabel,x=0,y=0]{v_1}
	\Edges[style={dash pattern=on 0pt off 100pt,thick,->--,bend right = 40}](v_1,v_2);
	\Edges[style={thick,->,>=mytip, bend left = 40}](v_2,v_1)
	\Edges[style={bend right=40,thick,->--}](v_2,v_1);
\end{tikzpicture} good  \\ $\Rightarrow \Leftarrow$ } &
%%%%%%%%%%%%%%%%%
\parbox{0.75in} { \centering 
\begin{tikzpicture}
	\SetFancyGraph
	\Vertex[NoLabel,x=0,y=0.8]{v_2}
	\Vertex[NoLabel,x=0.5,y=0]{v_1}
	\Vertex[NoLabel,x=-0.5,y=0]{v_3}
	\Edges[style={dash pattern=on 0pt off 100pt,thick,->--}](v_2,v_1)
	\Edges[style={thick,->,>=mytip}](v_2,v_1)
	\Edges[style={thick,->--}](v_2,v_3);
	\node at (-0.35,0.6) {\tiny $e_1$};
	\node at (0.35,0.6) {\tiny $e_2$};
\end{tikzpicture} bad \\ and \\ \vspace{.1cm}  \begin{tikzpicture}
	\SetFancyGraph
	\Vertex[NoLabel,x=0,y=0.8]{v_2}
	\Vertex[NoLabel,x=0.5,y=0]{v_1}
	\Vertex[NoLabel,x=-0.5,y=0]{v_3}
	\Edges[style={dash pattern=on 0pt off 100pt,thick,->--}](v_2,v_1)
	\Edges[style={thick,->,>=mytip}](v_2,v_1)
	\Edges[style={thick,->--}](v_2,v_3);
	\Edges[style={thick,->,>=mytip}](v_1,v_3)
	\Edges[style={dash pattern=on 0pt off 100pt,thick,->,>=mytip}](v_3,v_1)
	\Edges[style={dash pattern=on 0pt off 100pt,thick,->--}](v_3,v_1)
	\node at (-0.35,0.6) {\tiny $e_1$};
	\node at (0.35,0.6) {\tiny $e_2$};
	\node at (0,-0.2) {\tiny $e_3$};
\end{tikzpicture} good \\ $\Downarrow$ \\ \vspace{.1cm} \begin{tikzpicture}
	\SetFancyGraph
	\Vertex[NoLabel,x=0,y=0.8]{v_2}
	\Vertex[NoLabel,x=0.5,y=0]{v_1}
	\Vertex[NoLabel,x=-0.5,y=0]{v_3}
	\Edges[style={dash pattern=on 0pt off 100pt,thick,->--}](v_2,v_1)
	\Edges[style={thick,->,>=mytip}](v_2,v_1)
	\Edges[style={thick,->--}](v_2,v_3);
	\Edges[style={thick,->,>=mytip}](v_1,v_3)
	\Edges[style={dash pattern=on 0pt off 100pt,thick,->--}](v_3,v_1)
	\draw[rotate=45,dashed,thick] (-0.1,0.3) ellipse (.4 and .5);
	\node[color=red] at (-0.45,0.4) {\Huge $\times$};
\end{tikzpicture} good \\ or \\ \vspace{.1cm} \begin{tikzpicture}
	\SetFancyGraph
	\Vertex[NoLabel,x=0,y=0.8]{v_2}
	\Vertex[NoLabel,x=0.5,y=0]{v_1}
	\Vertex[NoLabel,x=-0.5,y=0]{v_3}
	\Edges[style={dash pattern=on 0pt off 100pt,thick,->--}](v_2,v_1)
	\Edges[style={thick,->,>=mytip}](v_2,v_1)
	\Edges[style={thick,->--}](v_2,v_3);
	\Edges[style={thick,->,>=mytip}](v_3,v_1)
	\Edges[style={dash pattern=on 0pt off 100pt,thick,->--}](v_3,v_1)
	\draw[rotate=-45,dashed,thick] (0.1,0.3) ellipse (.4 and .5);
	\node[color=green] at (0.7,0.5) {\LARGE $\checkmark$};
\end{tikzpicture} good} &
%%%%%%%%%%%%%%%%%
 \parbox{1in} { \centering Induct on $n$: \\
 \begin{tikzpicture}
	\SetFancyGraph
	\Vertex[NoLabel,x=0,y=1]{v_2}
	\Vertex[NoLabel,x=0,y=0]{v_1}
	\Vertex[NoLabel,x=1,y=1]{v_3}
	\Edges[style={thick,bend left=90}](v_1,v_2)
	\node at (-0.125,0.5) {\footnotesize $\cdots$};
	\Edges[style={thick,bend right=10,dashed}](v_1,v_2)
	\Edges[style={thick,bend left=40,->,>=mytip}](v_2,v_1)
	\Edges[style={dash pattern=on 0pt off 100pt,thick,bend left=40,->--}](v_2,v_1);
	\Edges[style={thick}](v_2,v_3)
	\Edges[style={dash pattern=on 0pt off 100pt,thick,bend left=60,->--}](v_2,v_3);
	\Edges[style={thick,bend left=60,->,>=mytip}](v_2,v_3)
	\node at (-0.6,0.6) {\tiny $e_1$};
	\node at (0.6,0.45) {\footnotesize $e_{n}$};
	\node at (0.6,0.85) {\footnotesize $e_{n-1}$};
	\node at (0.55,1.5) {\footnotesize $e_{n+1}$};
\end{tikzpicture} good \\ $\Downarrow$ \\ \begin{tikzpicture}
	\SetFancyGraph
	\Vertex[NoLabel,x=0,y=1]{v_2}
	\Vertex[NoLabel,x=0,y=0]{v_1}
	\Edges[style={thick,bend left=90}](v_1,v_2)
	\node at (-0.125,0.5) {\footnotesize $\cdots$};
	\Edges[style={thick,bend right=10}](v_1,v_2)
	\Edges[style={thick,bend left=40,->,>=mytip}](v_2,v_1)
	\Edges[style={dash pattern=on 0pt off 100pt,thick,bend left=40,->--}](v_2,v_1);
	\Edges[style={dash pattern=on 0pt off 100pt,thick,bend left=90,->--}](v_2,v_1);
	\Edges[style={thick,bend left=90,->,>=mytip}](v_2,v_1)
\end{tikzpicture} good \\ $\Downarrow$ \\ \begin{tikzpicture}
	\SetFancyGraph
	\Vertex[NoLabel,x=0,y=1]{v_2}
	\Vertex[NoLabel,x=0,y=0]{v_1}
	\Edges[style={thick,bend left=90}](v_1,v_2)
	\node at (-0.125,0.5) {\footnotesize $\cdots$};
	\Edges[style={thick,bend right=10}](v_1,v_2)
	\Edges[style={thick,bend left=40,->,>=mytip}](v_2,v_1)
	\Edges[style={dash pattern=on 0pt off 100pt,thick,bend left=40,->--}](v_2,v_1);
\end{tikzpicture} good }
%%%%%%%%%%%%%%%%%
\end{tabular}
\caption{A visual aid for the proof of Theorem~\ref{thm:tuttearemin}. The smaller arrows in the middle of an edge indicate the reference orientation (if there are no arrows in the middle of an edge then the reference orientation is arbitrary). The larger arrows at the end of an edge are edge orientations that belong to the fourientation. In general edges are ordered from left-to-right (with leftmost being minimal) but edge labels are included when this is not the case and the order is important.} \label{fig:tutteareminaid}
\end{figure}
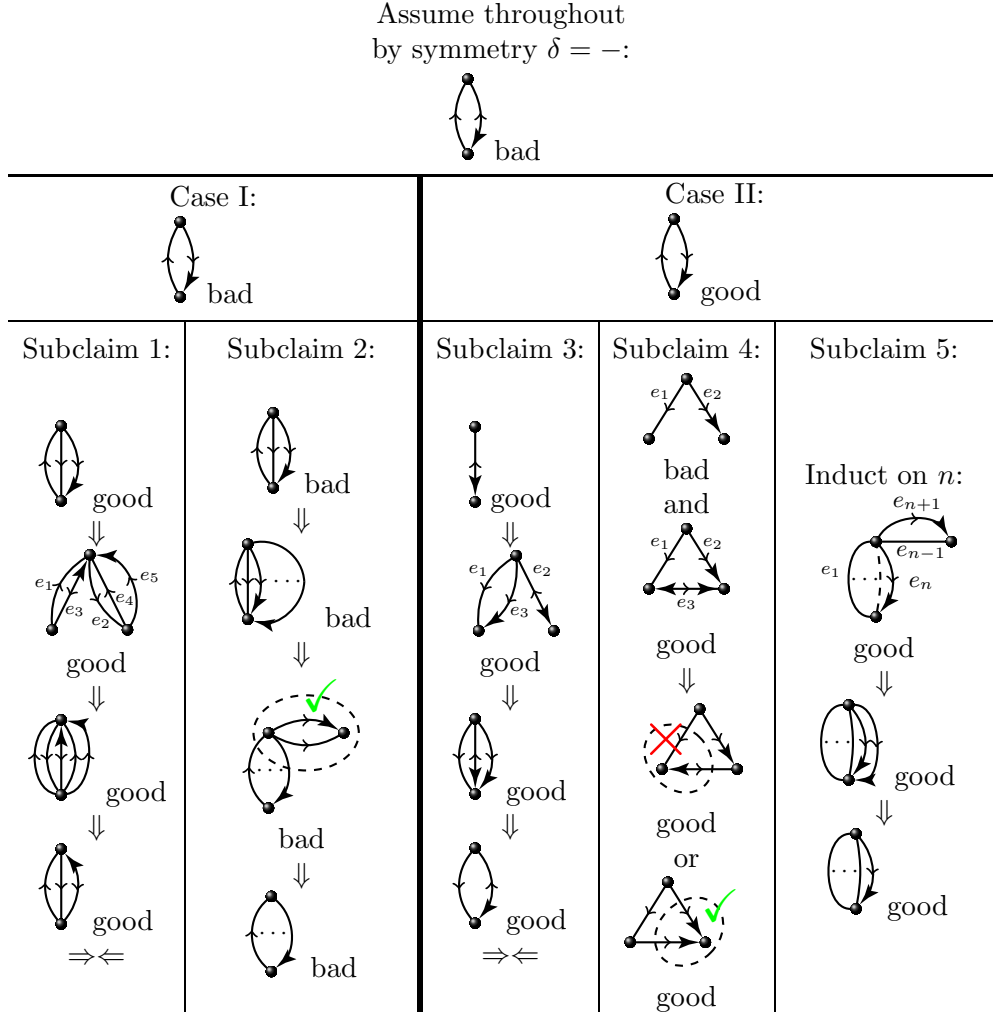

From now on assume $\emptyset \in X$ (or else we cannot have that $\O \setminus e_n = \emptyset$ is bad). Using the $+ \leftrightarrow -$ symmetry assume additionally from now on that $\delta = -$. The proof that follows is technical and requires constructing several auxillary graphs; Figure~\ref{fig:tutteareminaid} offers a pictorial aid for the various subclaims made below. We must now consider how our Tutte property behaves with respect to the other reference orientation for $B_2$. Define~$\Oref^{2'}$ by $\Oref^{2'}(e^{+}_1) := \Oref^{2'}(e^{-}_2) := (u,v)$ and define $\O^{2'} := \{e_2^{+}\}$. There are two cases to address: either $\O^{2'}$ is bad for $(B_2,\Oref^{2'},<)$ or it is good.

\medskip
\noindent {\bf Case I: $\O^{2'}$ is a bad fourientation of $(B_2,\Oref^{2'},<)$.}
\medskip

Note that this case is consistent with the min-edge cut property defined by $(X,\delta)$. We will show that indeed our Tutte property is this min-edge cut property.

\begin{subclaim} \label{sub:b3}
Set~$\Oref^3(e^+_1) := \Oref^3(e^-_2) := \Oref^3(e^-_3) := (u,v)$. Then~$\O^3 := \{e_3^{+}\}$ is a bad fourientation of~$(B_3,\Oref^3,<)$.
\end{subclaim}

\begin{proof} Suppose to the contrary. Define the auxiliary graph~$G^{*}$ by $V(G^{*}) := \{u,v,w\}$ and~$E(G^{*}) := \{e_1,e_2,e_3,e_4,e_5\}$ where $e_1 := e_3 := \{u,w\}$ and $e_2 := e_4 := e_5 := \{u,v\}$. Set $\Oref^{*}(e_1^{-}) := \Oref^{*}(e_3^{+}) := (u,w)$ and $\Oref^{*}(e_2^{+}) := \Oref^{*}(e_4^{-}) := \Oref^{*}(e_5^{-}) := (u,v)$. Then $\O^{*} := \{e_3^{-},e_5^{+}\}$ must be good for $\G^{*} := (G^{*},\Oref^{*},<)$: the graph $G^{*}$ has two simple cuts $\Cu_1 := \{ \{u,v\}, \{w\}\}$ and $\Cu_2 := \{ \{u,w\}, \{v\}\}$; and the contraction to these cuts are~$(\G^{*}_{\Cu_1},\O^{*}_{\Cu_1}) \simeq ((B_2,\Oref^{2'},<),\O^{2'})$ and~$(\G^{*}_{\Cu_2},\O^{*}_{\Cu_2}) \simeq ((B_3,\Oref^{3},<),\O^3)$, both of which are good by supposition. Let $G^{*'}$ be the graph obtained from $G^{*}$ by adding an edge $e_6 := \{v,w\}$ and let~$\Oref^{*'}$ be any extension of $\Oref^{*}$. By the Tutte condition~(2c), we have that $\O^{*}$ remains good for~$\G^{*'} := (G^{*'},\Oref^{*'},<)$. Set $\Cu_3 := \{ \{u\}, \{v,w\}\}$, a cut of~$G^{*'}$. Since we are working with a cut property, we know the contraction $(\G^{*'}_{\Cu_3},\O^{*}_{\Cu_3})$ is good; by removing~$e_5$ and $e_4$ from this contraction using conditions~\ref{cond:tuttegoodor} and~\ref{cond:tutteun} we get that something isomorphic to $((B_3,\Oref^{3},<),-\O^3)$ is good. But $\O^3$ and $-\O^3$ both being good for~$(B_3,\Oref^{3},<)$ contradicts the Tutte condition~\ref{cond:tuttebador}. So indeed it must have been that~$\O^3$ was bad. \end{proof}

\begin{subclaim} \label{sub:caseimain}
Let $n > 1$. Fix some~$(B_n,\Oref,<)$. Suppose $\O = \{e_n^{\varepsilon}\}$ for $\varepsilon \in \{-,+\}$ with~$\Oref(e^{+}_1) = \Oref(e_n^{-\varepsilon})$. Then $\O$ is bad for $(B_n,\Oref,<)$.
\end{subclaim}

\begin{proof} Assume without loss of generality that $\Oref(e^{+}_1) = (u,v)$. Define a reference orientation $\Oref^{n+3}$ of $B_{n+3}$  by~$\Oref^{n+3}(e_1^{+}) := \Oref^{n+3}(e_2^{-}) := \Oref^{n+3}(e_3^{-}) := (u,v)$ and~$\Oref^{n+3}(e^{+}_i) := \Oref(e^{+}_{i-3})$ for all $4 \leq i \leq n+3$. Since $\O^3$ is bad for $(B_3,\Oref^3,<)$, repeated use of condition~\ref{cond:tutteun} and one application of~\ref{cond:tuttebador} says that~$\O^{*} := \{e_3^{+},e_{n+3}^{\varepsilon}\}$ is bad for~$(B_{n+3},\Oref^{n+3},<)$. Define the auxiliary graph~$G^{*}$ by $V(G^{*}) := \{u,v,w\}$ and~$E(G^{*}) := \{e_1,\ldots,e_{n+4}\}$, where $e_1 := e_4 := \ldots := e_{n+3} := \{u,v\}$,~$e_{n+4} := \{v,w\}$ and~$e_2 := e_3 := \{u,w\}$. Let $\Oref^{*}$ be any extension of $\Oref^{n+3}$. Because the contraction of~$( (G^{*}, \Oref^{*},<),\O^{*})$ to the cut $\{\{u\},\{v,w\}\}$ is isomorphic to~$((B_{n+3},\Oref^{n+3},<),\O^{*})$, we get that $\O^{*}$ is bad for $\G^{*} := (G^{*}, \Oref^{*},<)$. So by condition~\ref{cond:tutteun}, $\O^{*}\setminus e_{n+4}$ is bad for $\G^{*} \setminus e_{n+4}$. Note that $G^{*} \setminus e_{n+4}$ has two simple cuts:~$\Cu_1 := \{\{u,v\},\{w\}\}$ and~$\Cu_2 := \{\{u,w\},\{v\}\}$. The contraction of~$( \G^{*} \setminus e_{n+4},\O^{*} \setminus e_{n+4})$ to~$\Cu_1$ is isomorphic to $((B_2,\Oref^2,<),\O^2)$, which is good. So it must be that the contraction of~$(\G^{*}  \setminus e_{n+4},\O^{*} \setminus e_{n+4})$ to~$\Cu_2$, which is isomorphic to $((B_n,\Oref,<),\O)$, is bad. \end{proof}

\noindent Under the assumptions of the previous subclaim we have by condition~\ref{cond:tuttebador} that~$-\O$ is good. Recall by considerations at the beginning of the proof that the status of any fourientation $\O$ was only at issue if $\O = \{e^{\varepsilon}\}$ for some $\varepsilon \in \{+,-\}$ and $\O \setminus e_n$ was bad. But we just showed that in this case the status of $\O$ still agrees with the min-edge cut property defined by $(X,\delta)$. So we are done with Case I.

\medskip
\noindent {\bf Case II: $\O^{2'}$ is a good fourientation of $(B_2,\Oref^{2'},<)$.}
\medskip

Note that this case is in contradiction with the min-edge cut property defined by~$(X,\delta)$. We claim that our Tutte property must be cut weird. 

\begin{subclaim} \label{sub:minusinx}
 We have $\{-\} \in X$.
\end{subclaim}
\begin{proof} Suppose to the contrary. Define the auxiliary graph $G^{*}$ by~$V(G^{*}) := \{u,v,w\}$ and $E(G^{*}) := \{e_1,e_2,e_3\}$ where $e_1 := e_3 := \{u,v\}$ and $e_2 := \{u,w\}$. Define $\Oref^{*}$ by~$\Oref^{*}(e_1^{+}) := \Oref^{*}(e_3^{+}) := (u,v)$ and $\Oref^{*}(e_2^{+}) := (w,u)$. Then~$\O^{*} := \{e_2^{-},e_3^{+}\}$ is good for $\G^{*} := (G^{*},\Oref^{*},<)$: the graph $G^{*}$ has two simple cuts $\Cu_1 := \{ \{u,v\}, \{w\}\}$ and~$\Cu_2 := \{ \{u,w\}, \{v\}\}$ and we have that~$(\G^{*}_{\Cu_1},\O^{*}_{\Cu_1}) \simeq ((B_1,\Oref^{1},<),\{e_1^{-}\})$ and~$(\G^{*}_{\Cu_2},\O^{*}_{\Cu_2}) \simeq ((B_2,\Oref^{2},<),\O^2)$, both of which are good by supposition. Let~$G^{*'}$ be the graph obtained from $G^{*}$ by adding an edge $e_4 := \{v,w\}$ and let~$\Oref^{*'}$ be any extension of $\Oref^{*}$. Then~$\O^{*}$ is good for $\G^{*'} := (G^{*'},\Oref^{*'},<)$ by condition~\ref{cond:tuttegoodor}.  Set~$\Cu_3 := \{ \{u\}, \{v,w\}\}$, a cut of $G^{*'}$. The contraction~$(\G^{*'}_{\Cu_3},\O^{*}_{\Cu_3})$ is good; by removing $e_3$ from this contraction using~\ref{cond:tuttegoodor} we get that something isomorphic to~$((B_2,\Oref^{2'},<),-\O^{2'})$ is good. But $\O^{2'}$ and $-\O^{2'}$ both being good for~$(B_2,\Oref^{2'},<)$ contradicts~\ref{cond:tuttebador}. So $\{-\} \in X$.
\end{proof}

\begin{subclaim} \label{sub:plusnotinx}
 We have $\{+\} \notin X$.
\end{subclaim}
\begin{proof} Define the auxiliary graph $G^{*}$ by~$V(G^{*}) := \{u,v,w\}$ and $E(G^{*}) := \{e_1,e_2\}$ where $e_1 := \{u,v\}$ and $e_2 := \{u,w\}$.  Define~$\Oref^{*}(e_1^{+}) := (u,v)$ and $\Oref^{*}(e_2^{+}) := (u,w)$. Set $\O^{*} := \{e_2^{+}\}$. Then~$\O^{*}$ is bad for $\G^{*} := (G^{*},\Oref^{*},<)$ because its contraction to~$\Cu_1:=\{\{u,w\},\{v\}\}$ is bad.  Let $G^{*'}$ be the graph obtained from $G^{*}$ by adding an edge $e_3 := \{v,w\}$, let~$\Oref^{*'}$ be the extension of $\Oref^{*}$ with $\Oref^{*'}(e_3^{+}) := (v,w)$, and let~$\O^{*'} := \O^{*} \cup \{e_3^{+},e_3^{-}\}$. Note that $\O^{*'}$ is good for~$\G^{*'} := (G^{*'},\Oref^{*'},<)$: the contractions to $\Cu_1$ and $\Cu_2 := \{\{u,v\},\{w\}\}$ no longer give potential cuts, and the contraction to~$\Cu_3 := \{\{u\},\{v,w\}\}$ is isomorphic to~$((B_2,\Oref^{2},<),\O^2)$.  So by condition~(2b), one of~$\O^{*}\cup \{e_3^{+}\}$ or~$\O^{*}\cup \{e_3^{-}\}$ must be good for $\G^{*'}$. Note that $\O^{*}\cup \{e_3^{-}\}$ is not good because the contraction of~$(\G^{*'},\O^{*}\cup \{e_3^{-}\})$ to $\Cu_1$ is isomorphic to~$((B_2,\Oref^{2'},<),-\O^{2'})$, which is bad by condition \ref{cond:tuttebador} since~$((B_2,\Oref^{2'},<),\O^{2'})$ is good. So $\O^{*}\cup \{e_3^{+}\}$ must be good; but then $((B_2,\Oref^2,<),\{e_1^{+},e_2^{+}\})$, which is isomorphic to the contraction of~$(\G^{*'},\O^{*}\cup \{e_3^{+}\})$ to $\Cu_2$, is good. Then by~\ref{cond:tuttebador} and~\ref{cond:tuttepot} we get~$\{+\} \notin X$. \end{proof}

Therefore we must have~$X = \{\emptyset,\{-\}\}$. This indeed is possible. In this case, the good fourientations are the cut weird fourientations. To see that these are exactly the good fourientations, again we can just check agreement on banana graphs. The only case not addressed by above considerations is when $\O$ is a fourientation of $(B_n,\Oref,<)$ for some~$n > 1$ where~$e_n^{\varepsilon} \in \O$ for~$\varepsilon \in \{-,+\}$ and $\O \setminus e_n = \emptyset$ is bad. 

\begin{subclaim} \label{sub:caseiimain}
Let $n > 1$. Set $\O := \{e_n^{+}\}$. Then $\O$ is good for any $(B_n,\Oref,<)$.
\end{subclaim}

\begin{proof} We prove this by induction on $n$. The case~\mbox{$n=2$} is true by our suppositions. So assume $n > 2$ and the result holds for smaller~$n$. Assume without loss of generality that~$\Oref(e^{+}_1) = (u,v)$. Suppose $\Oref(e_{n-1}^{\gamma}) = \Oref(e_{n}^{\gamma'}) = (u,v)$ for~$\gamma, \gamma' \in \{+,-\}$. Define the auxiliary graph $G^{*}$ by~$V(G^{*}) := \{u,v,w\}$ and~$E(G^{*}) := \{e_1,\ldots,e_{n+1}\}$ where~$e_1 := \ldots := e_{n-2} := e_{n} := \{u,v\}$ and~\mbox{$e_{n-1} := e_{n+1} := \{u,w\}$}. Define~$\Oref^{*}$ by~$\Oref^{*}(e_{i}) := \Oref(e_i)$ for all $1 \leq i \leq n-2$ and~\mbox{$\Oref^{*}(e_{n-1}^{\gamma}) := \Oref^{*}(e_{n+1}^{\gamma'}) := (u,w)$} and~$\Oref^{*}(e_n^{\gamma'}) := (u,v)$. Let $\O^{*} := \{e_{n}^{+},e_{n+1}^{+}\}$. Then~$\O^{*}$ is good for~$\G^{*} := (G^{*},\Oref^{*},<)$: the graph~$G^{*}$ has two simple cuts~$\Cu_1 := \{ \{u,v\}, \{w\}\}$ and~$\Cu_2 := \{ \{u,w\}, \{v\}\}$; the contraction to~$\Cu_1$ is isomorphic to~$((B_2,\Oref^2,<),\O^2)$ or to~$((B_2,\Oref^2,<),\O^{2'})$, which are good, and the contraction to~$\Cu_2$ is isomorphic to~$((B_{n-1},\Oref^{'},<),\{e_{n-1}^{+}\})$, which is good by our inductive hyptothesis. Let $G^{*'}$ be the graph obtained from $G^{*}$ by adding an edge~$e_{n+2} := \{v,w\}$ and let $\Oref^{*'}$ be any extension of~$\Oref^{*}$. By condition~\ref{cond:tutteun}, we have that~$\O^{*}$ remains good for~$\G^{*'} := (G^{*'},\Oref^{*'},<)$. Set~$\Cu_3 := \{ \{u\}, \{v,w\}\}$, a cut of~$G^{*'}$. The contraction~$(\G^{*'}_{\Cu_3}, \O^{*}_{\Cu_3})$ is good; by removing~$e_{n+1}$ from this contraction using condition~\ref{cond:tuttegoodor} we get that something isomorphic to~$((B_n,\Oref^{n},<),\O)$ is good. \end{proof}

\noindent Under the assumptions of the previous subclaim we have by condition~\ref{cond:tuttebador} that~$-\O$ is bad. So indeed any property that lands in Case II would have to be cut weird. By mimicking the proof of Theorem~\ref{thm:minaretutte} one can show that cut weird actually defines a consistent Tutte cut property. Finally note that if $\delta = +$ then by a completely symmetric argument either our Tutte property is still a min-edge cut property or we arrive at the other exceptional case where our property is cut co-weird.
\end{proof}

\begin{remark} 
Define a \emph{signed, ordered, oriented graph} to be~$(G,\Oref,<,\sigma)$, where the triple~$(G,\Oref,<)$ is an ordered, oriented graph, and $\sigma\colon E(G) \to \{+,-\}$ is any map from the edges of $G$ to $\{+,-\}$. We could extend our notion of fourientation property to take as input fourientations of signed, ordered, oriented graphs and only require invariance under isomorphism of these more decorated structures. Then we could extend the min-edge cut (cycle) property defined by $(X,\delta)$ to signed, ordered, oriented graphs by saying that a potential cut~$\OC$ (cycle $\OC$) of a fourientation~$\O$ of~$(G,\Oref,<,\sigma)$ is bad if it satisfies both of the following conditions:
\begin{enumerate}[label=(\roman*$'$)]
\item $\{\varepsilon\colon e_{\mathrm{min}}^{\varepsilon} \in \O\} = S$ for some $S \in X$, where $e_{\mathrm{min}}$ is the minimum edge in $E(\OC)$ ;
\item if $e_{\mathrm{min}}$ is unoriented (bioriented) in $\O$ then $e_{\mathrm{min}}^{\delta\cdot\sigma(e_{\mathrm{min}})} \in \mathbb{E}(\OC)$.
\end{enumerate}
The arguments already given in this section establish that the number of good $(k,l,m)$-fourientations of $(G,\Oref,<,\sigma)$ with respect to the intersection of the min-edge cut property defined by $(X,\delta_1)$ and the min-edge cycle property defined by $(Y,\delta_2)$ is still given by formula~(\ref{eqn:tutteeval}) in the statement of Theorem~\ref{thm:main}. However, a classification of Tutte properties where we allow the extra decoration $\sigma$ appears to be significantly more involved than Theorem~\ref{thm:tuttearemin}, and it is unclear what is gained by this extra level of generality. It would certainly be interesting to find a simple bijection from the good fourientations of~$(G,\Oref,<,\sigma_1)$ to the good fourientations of~$(G,\Oref,<,\sigma_2)$ with respect to some fixed min-edge cut property $(X,\delta)$.
\end{remark}

\section{Specializations} \label{sec:special}

In this section we consider $(k,l,m)$-fourientations for special values of~$(k,l,m)$. Let us call a fourientation with no bioriented edges a \emph{Type A fourientation}, and a fourientation with no unoriented edges a \emph{Type B fourientation}. In other words, a Type~A foruientation is a $(1,1,0)$-fourientation and a Type B fourientation is a $(1,0,1)$-fourientation. The fourientations that are both Type A and Type B, the $(1,0,0)$-fourientations, are precisely the total orientations. The impetus for this research was actually to unify the study of various classes of partial orientations. We now explain how Tutte fourientation properties give rise to many interesting classes of partial orientations.

\subsection{Partial orientations} \label{subsec:partial}

\begin{definition}
A \emph{partial orientation} of $(G,\Oref)$ is a subset $\O$ of $\mathbb{E}(G)$ such that for each $e \in E(G)$ at least one of $e^{+}$ or $e^{-}$ is not in $\O$. If $e^{+} \notin \O$ and $e^{-} \notin \O$ then we say~$e$ is \emph{neutral} in $\O$ and we write $e \notin \O$. If $e^{\pm} \in \O$ then we say $e$ is \emph{oriented} in $\O$.
\end{definition}

So a partial orientation is just a Type~A fourientation where we call the unoriented edges neutral. However, when studying partial orientations we actually want to consider Type A and Type~B fourientations ``simultaneously.'' Let us call the images of the min-edge classes of fourientations under the identity map from Type~A fourientations to partial orientations the \emph{Type~A classes} of partial orientations. There is also an obvious bijection from the set of Type~B fourientations of~$\G$ to the set of partial orientations of~$\G$ where we treat bioriented edges as neutral. Let us call the images of the min-edge classes of fourientations under this second bijection the \emph{Type~B classes} of partial orientations. Many (but not all) of the min-edge classes of partial orientations have been studied before, as we detail in~\S\ref{sec:connections}. In order to explicitly describe the Type~A and~B classes of partial orientations, let us give some preliminary definitions.

\begin{definition}
By abuse of language, a \emph{directed cut} (\emph{cycle}) of a partial orientation is a directed cut (cycle) of the underlying oriented graph for which all edges are oriented in agreement with the cut (cycle). A \emph{potential cut} (\emph{cycle}) of a partial orientation is a directed cut (cycle) of the underlying oriented graph for which all oriented edges are oriented consistently with the cut (cycle), but neutral edges are allowed. In symbols,~$\OC$ is a directed cut (cycle) of $\O$ if $e^{\pm} \in \mathbb{E}(\OC) \Rightarrow e^{\pm} \in \O$ for all~$e \in E$, whereas $\OC$ is a potential cut (cycle) of $\O$ if $e^{\pm} \in \mathbb{E}(\OC) \Rightarrow e^{\mp} \notin \O$ for all~$e \in E$.
\end{definition}

\noindent Using these notions of potential and directed cuts and cycles, we give the following names to the Type~A classes of partial orientations of an ordered, oriented graph:
\begin{enumerate}
\item {\bf Cut/cycle general}: There are no restrictions on cuts/cycles.
\item {\bf Cycle minimal}: The minimum edge in each directed cycle is oriented in agreement with its reference orientation.
\item {\bf Cycle maximal}:  The minimum edge in each directed cycle is oriented in disagreement its the reference orientation.
\item {\bf Acyclic}: There are no directed cycles.
\item {\bf Cut directed}: For each potential cut, if the minimum edge of the cut is neutral then the cut contains an oriented edge directed in agreement with the reference orientation of this minimum edge.
\item {\bf Cut negative}: The minimum edge in each potential cut is neutral or is oriented in agreement with its reference orientation.
\item {\bf Cut positive}: The minimum edge in each potential cut is neutral or is oriented in disagreement with its reference orientation.
\item {\bf Cut connected}: Each potential cut contains an oriented edge directed in agreement with the reference orientation of the minimum edge in the cut.
\item {\bf Cut co-connected:} For each potential cut, either the minimum edge of the cut is neutral and the cut contains an oriented edge directed in agreement with the reference orientation of this minimum edge, or the minimum edge is oriented in disagreement with its reference orientation.
\item {\bf Cut neutral}: The minimum edge in each potential cut is~neutral.
\item {\bf Cut internal}: The minimum edge in each potential cut is neutral and the cut contains an oriented edge directed in agreement with the reference orientation of this minimum edge.
\end{enumerate}
The names of the Type~B classes of partial orientations are similar (with ``strongly connected'' being dual to acyclic). The point of considering Type~A and Type~B classes simultaneously is that there are interesting containment relations between classes across types: Figure~\ref{fig:poclasses} depicts these relations. Theorem~\ref{thm:main} tells us that all Type~A and Type~B classes of partial orientations are enumerated by generalized Tutte polynomial evaluations (but note that it is not true in general that an intersection of a Type~A and a Type~B class is enumerated by a generalized Tutte polynomial evaluation). Figure~\ref{fig:fourtables} below displays these specific evaluations. 

\begin{remark}
The containment relations depicted in Figure~\ref{fig:poclasses} imply inequalities among the generalized Tutte polynomial evaluations in Figure~\ref{fig:fourtables}: for instance, for any graph~$G$ on~$n$ vertices with cyclomatic number~$g$ we have~$2^{g}\cdot T_G(2,\frac{3}{2}) \leq 2^{n-1}\cdot T_G(1,3)$.
\end{remark}

\begin{figure}
\setlength{\tabcolsep}{0pt}
\begin{tabular}{cc}
	\begin{tikzpicture}[scale=0.65]
		\node at (0,-0.72) {\parbox{0.75in}{\begin{center} Cut internal \end{center}}};
		\node at (0,-1.7) {\textrm{(A)}};
		\node at (-3.5,2) {\parbox{0.75in}{\begin{center} Cut \\ co-con. \end{center}}};
		\node at (-3.5,1) {\textrm{(A)}};
		\node at (3.5,2) {\parbox{0.75in}{\begin{center} Cut \\ connected \end{center}}};
		\node at (3.5,1) {\textrm{(A)}};
		\node at (0,2.2) {\parbox{0.8in}{\begin{center} Cut neutral \end{center}}};
		\node at (0,1.5) {\textrm{(A)}};
		\node at (-3.5,5) {\parbox{0.75in}{\begin{center} Cut \\ positive \end{center}}};
		\node at (-3.5,4) {\textrm{(A)}};
		\node at (3.5,5) {\parbox{0.75in}{\begin{center} Cut \\ negative \end{center}}};
		\node at (3.5,4) {\textrm{(A)}};
		\node at (0,5) {\parbox{0.75in}{\begin{center} Strongly \\ connected \end{center}}};
		\node at (0,4) {\textrm{(B)}};
		\node at (0,7.5) {\parbox{0.75in}{\begin{center} Cut \\ directed \end{center}}};
		\node at (0,6.5) {\textrm{(A)}};
		\node at (-2.5,8) {\parbox{0.75in}{\begin{center} Cut \\ maximal \end{center}}};
		\node at (-2.5,7) {\textrm{(B)}};
		\node at (2.5,8) {\parbox{0.75in}{\begin{center} Cut \\ minimal \end{center}}};
		\node at (2.5,7) {\textrm{(B)}};
		\node at (0,9.5) {\parbox{0.75in}{\begin{center} Cut \\ General \end{center}}};	
		\draw (0,1) -- (0,0);
		\draw (-3,1) -- (-0.5,0);
		\draw (3,1) -- (0.5,0);
		\draw (0,3.5) -- (0,2.5);
		\draw (-3,3.8) -- (-0.5,2.5);
		\draw (3,3.8) -- (0.5,2.5);
		\draw (-3.5,3.5) -- (-3.5,2.5);
		\draw (3.5,3.5) -- (3.5,2.5);
		\draw (-2.25,6.5) -- (-0.5,5.6);
		\draw (2.25,6.5) -- (0.5,5.6);
		\draw (-2.75,6.5) -- (-3.5,5.7);
		\draw (2.75,6.5) -- (3.5,5.7);
		\draw (-0.5,6.5) -- (-3.25,2.7);
		\draw (0.5,6.5) -- (3.25,2.7);
		\draw (0,8.8) -- (0,8);
		\draw (-0.5,8.8) -- (-1.5,8.2);
		\draw (0.5,8.8) -- (1.5,8.2);
	\end{tikzpicture} & 
	\begin{tikzpicture}[scale=0.65]
		\node at (0,-0.72) {\parbox{0.75in}{\begin{center} Cycle external \end{center}}};
		\node at (0,-1.7) {\textrm{(B)}};
		\node at (-3.5,2) {\parbox{0.75in}{\begin{center} Cycle \\ co-con. \end{center}}};
		\node at (-3.5,1) {\textrm{(B)}};
		\node at (3.5,2) {\parbox{0.75in}{\begin{center} Cycle \\ connected \end{center}}};
		\node at (3.5,1) {\textrm{(B)}};
		\node at (0,2.2) {\parbox{1in}{\begin{center} Cycle neutral \end{center}}};
		\node at (0,1.5) {\textrm{(B)}};
		\node at (-3.5,5) {\parbox{0.75in}{\begin{center} Cycle \\ positive \end{center}}};
		\node at (-3.5,4) {\textrm{(B)}};
		\node at (3.5,5) {\parbox{0.75in}{\begin{center} Cycle \\ negative \end{center}}};
		\node at (3.5,4) {\textrm{(B)}};
		\node at (0,5.2) {\parbox{0.75in}{\begin{center} Acyclic \end{center}}};
		\node at (0,4.5) {\textrm{(A)}};
		\node at (0,7.5) {\parbox{0.75in}{\begin{center} Cycle \\ directed \end{center}}};
		\node at (0,6.5) {\textrm{(B)}};
		\node at (-2.5,8) {\parbox{0.75in}{\begin{center} Cycle \\ maximal \end{center}}};
		\node at (-2.5,7) {\textrm{(A)}};
		\node at (2.5,8) {\parbox{0.75in}{\begin{center} Cycle \\ minimal \end{center}}};
		\node at (2.5,7) {\textrm{(A)}};
		\node at (0,9.5) {\parbox{0.75in}{\begin{center} Cycle \\ General \end{center}}};	
		\draw (0,1) -- (0,0);
		\draw (-3,1) -- (-0.5,0);
		\draw (3,1) -- (0.5,0);
		\draw (0,4) -- (0,2.5);
		\draw (-3,3.8) -- (-0.5,2.5);
		\draw (3,3.8) -- (0.5,2.5);
		\draw (-3.5,3.5) -- (-3.5,2.5);
		\draw (3.5,3.5) -- (3.5,2.5);
		\draw (-2.25,6.5) -- (-0.5,5.6);
		\draw (2.25,6.5) -- (0.5,5.6);
		\draw (-2.75,6.5) -- (-3.5,5.7);
		\draw (2.75,6.5) -- (3.5,5.7);
		\draw (-0.5,6.5) -- (-3.25,2.7);
		\draw (0.5,6.5) -- (3.25,2.7);
		\draw (0,8.8) -- (0,8);
		\draw (-0.5,8.8) -- (-1.5,8.2);
		\draw (0.5,8.8) -- (1.5,8.2);
	\end{tikzpicture}
\end{tabular}
\caption{The min-edge cut and cycle classes of partial orientations ordered by containment.} \label{fig:poclasses}
\end{figure}
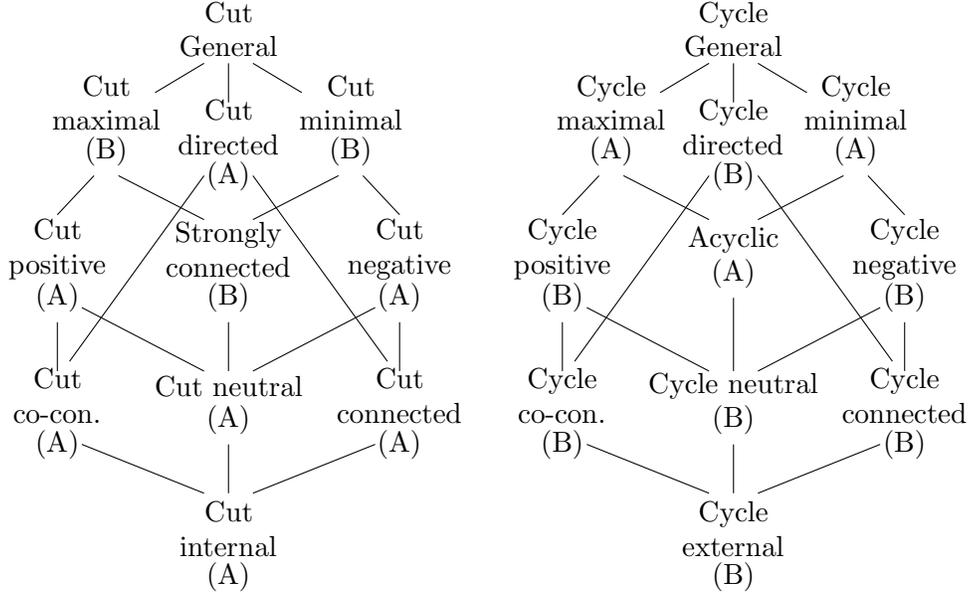

\subsection{Total orientations}

Of course, we can also set $(k,l,m) := (1,0,0)$. The $(1,0,0)$-fourientations of $\G$ are precisely the total orientations. The min-edge cut classes for total orientations are:
\begin{enumerate}
\item {\bf Cut general}: There are no restrictions on cuts.
\item {\bf Cut minimal}: The minimum edge in each directed cut is oriented in agreement with its reference orientation.
\item {\bf Cut maximal}:  The minimum edge in each directed cut is oriented in disagreement with its reference orientation.
\item {\bf Strongly connected}: There are no directed cuts.
\end{enumerate}
The poset of these four classes ordered by containment is isomorphic to the Boolean lattice on two elements. By intersecting min-edge cut and cycle classes of total orientations we realize all values of $T(x,y)$ for integral $0 \leq x,y \leq 2$ as explained in the unifying work of Gioan~\cite{gioan2007enumerating} and Bernardi~\cite{bernardi2008tutte}. Bernardi connects this $3 \times 3$ table with a corresponding table of classical subgraph enumerations. Note, however, that the input data of~\cite{bernardi2008tutte} is different than what we are working with here: Bernardi uses an embedding of the graph into a surface rather than~$\Oref$ and~$<$ to define his classes and in particular to define a notion of internal and external activity.  The middle row and column of the $3 \times 3$ table have various equivalent descriptions (see~\S\ref{subsec:history} above):
\pagebreak
\begin{itemize}
\item Middle row:
	\begin{itemize}
	\item cycle minimal total orientations;
	\item cycle maximal total orientations;
	\item equivalence classes of total orientations modulo cycle reversals;
	\item indegree sequences of total orientations.
	\end{itemize}
\item Middle column:
	\begin{itemize}
	\item cut minimal total orientations;
	\item cut maximal total orientations;
	\item equivalence classes of total orientations modulo cocycle reversals;
	\item $q$-connected total orientations.
	\end{itemize}
\end{itemize}
\noindent Informally, the indegree sequence of an orientation is the list of the numbers of incoming edges at each vertex, and a $q$-connected orientation is one with a directed path from~$q$ to every other vertex. We now recall the cycle/cocycle reversal systems of Gioan~\cite{gioan2007enumerating}. Given a total orientation~$\O$ of~$\G$, a \emph{(co)cycle reversal} is the operation of replacing~$\O$ by~$(\O \setminus \mathbb{E}(\OC)) \cup \mathbb{E}(-\OC)$ for some directed cycle (cut) $\OC$ of $\O$. We write~$\O \overset{\Cy} \sim \O'$ ($\O \overset{\Cu} \sim \O'$) if $\O$ is related to~$\O'$ by a series of (co)cycle reversals, and write $\O \sim \O'$ if~$\O$ is related to $\O'$ by a series of cycle and cocycle reversals. The three equivalence relations~$ \overset{\Cy} \sim$, $\overset{\Cu} \sim$, and $\sim$ define the \emph{cycle}, \emph{cocycle}, and \emph{cycle-cocycle} reversal systems respectively. Each equivalence class in the (co)cycle reversal system contains a unique cycle (cut) minimal orientation, and each equivalence class in the cycle-cocycle reversal system contains a unique cut minimal-cycle minimal orientation.

\subsection{Tables of enumerations}
Figure~\ref{fig:fourtables} displays four tables recording the generalized Tutte polynomial evaluations that, as a consequence of Theorem~\ref{thm:main}, enumerate the various min-edge classes of generalized orientations.

\subsection{Subgraphs} \label{subsec:subgraphs}

If we set $(k,l,m) := (0,1,1)$ we get fourientations with no oriented edges. We may identify such fourientations with subgraphs by thinking of the bioriented edges as belonging to our subgraph and the unoriented edges as being absent. Here a subgraph just means a subset $H\subseteq E(G)$ of the edges of $G$. In other words, all subgraphs $H$ of $G$ under consideration have $V(H) = V(G)$ and $E(H)\subseteq E(G)$, so we can identify the subgraph $H$ with its set of edges $E(H)$. Let $H \subseteq E(G)$ be a subgraph of $G$. Then a \emph{cut} $\Cu$ of $H$ is a cut $\Cu$ of $G$ with $E(\Cu) \cap H = \emptyset$. And a \emph{cycle} $\Cy$ of $H$ is a cycle $\Cy$ of $G$ with $E(\Cy) \subseteq H$. The min-edge cut classes for subgraphs become:
\begin{enumerate}
\item {\bf Cut general}: There are no restrictions on cuts.
\item {\bf Spanning}: The subgraph has no cuts.
\end{enumerate}
Of course the poset of these two classes ordered by containment is isomorphic to the Boolean lattice on one element. Dually, the min-edge cycle classes become:
\begin{enumerate}
\item {\bf Cycle general}: There are no restrictions on cuts.
\item {\bf Forest}: The subgraph has no cycles.
\end{enumerate}
These classes make up two rows and two columns of the~$3 \times 3$ table of classical subgraph Tutte polynomial enumerations mentioned above. In order to recover the other row and column of this table we need to consider internal and external activities: this will be the main project of a sequel paper. We now briefly sketch this approach.

\subsection{Future work: activities} \label{subsec:future}
For a fourientation $\O$ of $\G$ let $\O^{o}$ denote the set of oriented edges of $\O$, $\O^{u}$ the set of unoriented edges, and $\O^{b}$ the set of bioriented edges as in the statement of Theorem~\ref{thm:main}. Also, let $\O^{+}$ denote the set of edges oriented in agreement with the reference orientation and $\O^{-}$ the set of edges oriented in disagreement with the reference orientation so that $\O^{o} = \O^{+} \cup \O^{-}$.

Las Vergnas developed a notion of orientation activity that allows one to recapture the enumerations in the bottom $3 \times 3$ table in Figure~\ref{fig:fourtables}. Specifically, let us say~$e \in E(G)$ is \emph{internally active} in the total orientation $\O$ if it is the minimum edge in some directed cut of $\O$. Dually, we say $e \in E(G)$ is \emph{externally active} in $\O$ if it is the minimum edge in some directed cycle of $\O$. Let $I(\O)$ denote the set of internally active edges in~$\O$ and~$L(\O)$ the set of externally active edges. For ease of notation set~$I(\O^{+}) := I(\O) \cap \O^{+}$ and so on. Las Vergnas~\cite[Theorem 3.1]{las1984tutte}~\cite[Theorem~3.1]{las2012tutte} proved
\begin{align} \label{eqn:orientationactivity}
T_G(x+w,y+z) = \sum_{\O} x^{| I(\O^{+})|} w^{|I(\O^{-})|} y^{|L(\O^{+})|} z^{|L(\O^{-})|}
\end{align}
where the sum is over all total orientation $\O$ of $\G$.

There is a very analogous story for subgraphs. We say an edge $e \in E(G)$ is \emph{internally active} in the subgraph $H \subseteq E(G)$ if $e$ is the minimum edge in a cut of $H\setminus \{e\}$. Dually, we say~$e\in E(G)$ is \emph{externally active} in $H$ if $e$ is the minimum edge in a cycle of $H \cup \{e\}$. Let $I(H)$ denote the set of internally active edges of $H$ and $L(H)$ the set of externally active edges. In the case where $H$ is a spanning tree of $G$, these notions of activity go back to the original work of Tutte~\cite{tutte1954contribution}; but in the case of arbitrary subgraphs $H$ the earliest reference for these notions of activity we are aware of is Gordon-Traldi~\cite{gordon1990generalized} (but see also~\cite{bari1979chromatic}). Gordon-Traldi~\cite[Theorem 3]{gordon1990generalized} (see also~\cite[Theorem~3.5]{las2013tutte}) proved
\begin{align} \label{eqn:subgraphactivity}
T_G(x_*+w_*,y_*+z_*) = \sum_{H \subseteq E(G)} {x_*}^{|I(H) \cap H|} {w_*}^{|I(H)\setminus H|} {y_*}^{|L(H) \setminus H|} {z_*}^{|L(H) \cap H|}.
\end{align}

\medskip

Of course the equations~~(\ref{eqn:orientationactivity}) and~(\ref{eqn:subgraphactivity}) give different expressions for the same Tutte polynomial evaluation when the variables with stars equal those without. Giving a bijective proof of this fact that matches terms in the two sums is one aim of the so-called ``active bijection'' of Gioan-Las Vergnas~\cite{gioan2009active}. We now discuss a different approach to understanding the relationship between~(\ref{eqn:orientationactivity}) and~(\ref{eqn:subgraphactivity}) via fourientations. The ultimate goal is to derive a fourientation activity expression for the Tutte polynomial that specializes to both~(\ref{eqn:orientationactivity}) and~(\ref{eqn:subgraphactivity}). So let us say an edge $e \in E(G)$ is \emph{internally active} in a fourientation $\O$ of $\G$ if
\begin{itemize}
\item $e$ is oriented in $\O$ and is the minimum edge of some potential cut of $\O$;
\item $e$ is unoriented in $\O$ and is the minimum edge of some potential cut of $\O \cup \{e^{-}\}$.
\end{itemize}
Dually, we say~$e \in E(G)$ is \emph{externally active} in~$\O$ if
\begin{itemize}
\item $e$ is oriented in $\O$ and is the minimum edge of some potential cycle of $\O$;
\item $e$ is bioriented in $\O$ and is the minimum edge of some potential cycle of $\O \setminus \{e^{+}\}$.
\end{itemize}
As before, let $I(\O)$ denote the set of internally active edges in $\O$ and $L(\O)$ the set of externally active edges. Set~$I(\O^{u}) := I(\O) \cap \O^{u}$ and so on. Then the techniques developed in this paper allow one to prove that
\begin{gather} \label{eqn:fouractivity}
(k+m)^{n-1}(k+l)^{g}T_G\left(\frac{kx+kw+lw_*+m}{k+m},\frac{ky+kz+l+mz_*}{k+l}\right) = \\
\sum_{\O} k^{|\O^{o}|}l^{|\O^{u}|}m^{|\O^{b}|}x^{| I(\O^{+}) |} w^{| I(\O^{-}) |} {w_*}^{| I(\O^{u}) |} y^{| L( \O^{+}) |} z^{| L(\O^{-}) |}{z_*}^{| L(\O^{b}) |} \nonumber
\end{gather}
where the sum is over all fourientations $\O$ of $\G$. Taking~$x,w,w_*,y,v,v_* \in \{0,1\}$ in~(\ref{eqn:fouractivity}) recovers the enumeration of min-edge classes of fourientations from Theorem~\ref{thm:main}. Also, specializing~$(k,l,m) := (1,0,0)$ in~(\ref{eqn:fouractivity}) recovers Las Vergnas's formula~(\ref{eqn:orientationactivity}). Moreover, specializing~$(k,l,m) := (0,1,1)$ in~(\ref{eqn:fouractivity}) yields
\[T_G(1+w_*,1+z_*) = \sum_{H \subseteq E(G)} {w_*}^{|I(H)\setminus H|} {z_*}^{|L(H) \cap H|},\]
which is~(\ref{eqn:subgraphactivity}) with $x_* := 1$ and $y_* := 1$. It would be extremely interesting to introduce two extra variables, $x_*$ and $y_*$, into~(\ref{eqn:fouractivity}) so that we can fully recover the Gordon-Traldi formula~(\ref{eqn:subgraphactivity}) when we set~$(k,l,m) := (0,1,1)$. This amounts to extending the definition of fourientation activity so that bioriented edges can be internally active and unoriented edges can be externally active. This extension is the project of our future work. Ultimately we should also be able to realize the active bijection as a map from the set of fourientations to itself that explains the symmetric role played by the variables in~(\ref{eqn:fouractivity}) when $(k,l,m) := (1,1,1)$.

\afterpage{
\newgeometry{top=1in, bottom=1in, left=0.75in, right=0.75in}

\begin{figure}

\setlength{\tabcolsep}{0.3em}

\def\arraystretch{2}
\begin{tabular}{c | c | c | c | c | }
	\multicolumn{5}{c}{{\bf Fourientations}} \\
	\multicolumn{1}{r}{} & \multicolumn{1}{c}{General} & \multicolumn{1}{c}{\parbox{1in}{\begin{center} {\color{blue} Cut pos./neg.} \\  {\color{red} Cut directed}\end{center}}} & \multicolumn{1}{c}{\parbox{1in}{\begin{center} {\color{blue} Cut neutral} \\ {\color{red} Cut (co)-con. }\end{center}}}  & \multicolumn{1}{c}{Cut internal}   \\ \hhline{|~|-|-|-|-|}
	General & \cellcolor{mygray}  $2^{|E|} T(2,2)$ & $2^{|E|} T(\frac{3}{2},2)$ &  \cellcolor{mygray}  $2^{|E|} T(1,2)$ &   $2^{|E|} T(\frac{1}{2}, 2)$   \\ \hhline{|~|-|-|-|-|}
	\parbox{0.75in}{\begin{center}{\color{blue} Cycle pos./neg.} \\ {\color{red} Cycle directed}\end{center}} & $2^{|E|} T(2,\frac{3}{2})$ & $2^{|E|} T(\frac{3}{2},\frac{3}{2})$ & $2^{|E|} T(1,\frac{3}{2})$ & $2^{|E|} T(\frac{1}{2},\frac{3}{2})$ \\ \hhline{|~|-|-|-|-|}
	\parbox{0.75in}{\begin{center} {\color{blue} Cycle neutral} \\ {\color{red} Cycle (co)-con.} \end{center}} & $2^{|E|} T(2,1)$ & $2^{|E|} T(\frac{3}{2},1)$ & $2^{|E|} T(1,1)$ & $2^{|E|} T(\frac{1}{2},1)$ \\ \hhline{|~|-|-|-|-|}
	\parbox{0.75in}{\begin{center}Cycle external\end{center}} & $2^{|E|} T(2,\frac{1}{2})$ & $2^{|E|} T(\frac{3}{2},\frac{1}{2})$ & $2^{|E|} T(1,\frac{1}{2})$ & $2^{|E|} T(\frac{1}{2},\frac{1}{2})$ \\ \hhline{|~|-|-|-|-|}
\end{tabular}

\vspace{0.3in}

\setlength{\tabcolsep}{3pt}
\def\arraystretch{1}
\begin{tabular}{c|c} 
	\begin{minipage}{.49\linewidth}
		\begin{tabular}{c | c | c | c | c |}
			\multicolumn{5}{c}{{\bf Type A classes of partial orientations}}\\
			\multicolumn{1}{r}{} & \multicolumn{1}{c}{General} & \multicolumn{1}{c}{\parbox{0.65in}{\begin{center} {\color{blue} Cut pos./neg.} \\ {\color{red} Cut dir.}\end{center}} } & \multicolumn{1}{c}{\parbox{0.65in}{\begin{center} {\color{blue} Cut neutral} \\ {\color{red} Cut (co)-con.}\end{center}} } &\multicolumn{1}{c}{\parbox{0.5in}{\begin{center}Cut int. \end{center}}} \\ \hhline{|~|-|-|-|-|}
			Gen. &  \cellcolor{mygray} \parbox{0.5in}{\begin{center} $2^{g}$ \\ $T(3,\frac{3}{2})$ \end{center}} & \parbox{0.5in}{\begin{center} $2^{g}$ \\ $T(2,\frac{3}{2})$ \end{center}} & \parbox{0.5in}{\begin{center} $2^{g}$ \\ $T(1,\frac{3}{2})$ \end{center}} & \parbox{0.5in}{\begin{center} $2^{g}$ \\ $T(0,\frac{3}{2})$ \end{center}} \\ \hhline{|~|-|-|-|-|}
			\parbox{0.5in}{\begin{center}{\color{blue} Cycle \\ min.} \\ {\color{red}Cycle  max.}\end{center}} & \cellcolor{mygray}  \parbox{0.5in}{\begin{center} $2^{g}$ \\ $T(3,1)$ \end{center}} & \parbox{0.5in}{\begin{center} $2^{g}$ \\ $T(2,1)$ \end{center}} & \parbox{0.5in}{\begin{center} $2^{g}$ \\ $T(1,1)$ \end{center}} & \parbox{0.5in}{\begin{center} $2^{g}$ \\ $T(0,1)$ \end{center}} \\ \hhline{|~|-|-|-|-|}
			Acyc. & \cellcolor{mygray} \parbox{0.5in}{\begin{center} $2^{g}$ \\ $T(3,\frac{1}{2})$ \end{center}} & \parbox{0.5in}{\begin{center} $2^{g}$ \\ $T(2,\frac{1}{2})$ \end{center}} & \cellcolor{mygray} \parbox{0.5in}{\begin{center} $2^{g}$ \\ $T(1,\frac{1}{2})$ \end{center}} & \parbox{0.5in}{\begin{center} $2^{g}$ \\ $T(0,\frac{1}{2})$ \end{center}} \\ \hhline{|~|-|-|-|-|}
		\end{tabular}  
	\end{minipage} & \begin{minipage}{.42\linewidth}
		\begin{tabular}{c | c | c | c |}
			\multicolumn{4}{c}{{\bf Type B classes of partial orientations}}\\
			\multicolumn{1}{r}{} & \multicolumn{1}{c}{General} & \multicolumn{1}{c}{\parbox{0.7in}{\begin{center}Cut \\ {\color{blue} min.}/{\color{red} max.} \end{center}}} & \multicolumn{1}{c}{\parbox{0.5in}{\begin{center}Strong.\\ con.\end{center}}} \\ \hhline{|~|-|-|-|}
			Gen. & \cellcolor{mygray} \parbox{0.5in}{\begin{center} $2^{n-1}$ \\ $T(\frac{3}{2},3)$ \end{center}} & \cellcolor{mygray} \parbox{0.5in}{\begin{center} $2^{n-1}$ \\ $T(1,3)$ \end{center}} & \cellcolor{mygray} \parbox{0.5in}{\begin{center} $2^{n-1}$ \\ $T(\frac{1}{2},3)$ \end{center}}  \\ \hhline{|~|-|-|-|}
			\parbox{0.6in}{\begin{center}{\color{blue} Cycle pos./neg.} \\ {\color{red} Cycle dir.}\end{center}} & \parbox{0.5in}{\begin{center} $2^{n-1}$ \\ $T(\frac{3}{2},2)$ \end{center}} & \parbox{0.5in}{\begin{center} $2^{n-1}$ \\ $T(1,2)$ \end{center}} & \parbox{0.5in}{\begin{center} $2^{n-1}$ \\ $T(\frac{1}{2},2)$ \end{center}}  \\ \hhline{|~|-|-|-|}
			\parbox{0.6in}{\begin{center} {\color{blue} Cycle neutral} \\ {\color{red} Cycle (co)-con.} \end{center}}  &  \cellcolor{mygray} \parbox{0.5in}{\begin{center} $2^{n-1}$ \\ $T(\frac{3}{2},1)$ \end{center}} & \parbox{0.5in}{\begin{center} $2^{n-1}$ \\ $T(1,1)$ \end{center}} &  \cellcolor{mygray} \parbox{0.5in}{\begin{center} $2^{n-1}$ \\ $T(\frac{1}{2},1)$ \end{center}}    \\ \hhline{|~|-|-|-|}
			\parbox{0.5in}{\begin{center}Cycle ext.\end{center}} & \parbox{0.5in}{\begin{center} $2^{n-1}$ \\ $T(\frac{3}{2},0)$ \end{center}} & \parbox{0.5in}{\begin{center} $2^{n-1}$ \\ $T(1,0)$ \end{center}} & \parbox{0.5in}{\begin{center} $2^{n-1}$ \\ $T(\frac{1}{2},0)$ \end{center}}  \\ \hhline{|~|-|-|-|}
		\end{tabular}  
	\end{minipage}
\end{tabular}

\vspace{0.3in}

\def\arraystretch{1.8}
\begin{tabular}{c |c| c| c |}
	\multicolumn{4}{c}{{\bf Total orientations }} \\
	\multicolumn{1}{r}{} & \multicolumn{1}{c}{General} & \multicolumn{1}{c}{\parbox{0.75in}{\begin{center}Cut \\ {\color{blue} min.}/{\color{red} max.} \end{center}}} & \multicolumn{1}{c}{\parbox{0.75in}{\begin{center}Strongly \\connected\end{center}}} \\ \hhline{|~|-|-|-|}
	General & \cellcolor{mygray} $T(2,2)$ & \cellcolor{mygray} $T(1,2)$ & \cellcolor{mygray} $T(0,2)$  \\ \hhline{|~|-|-|-|}
	\parbox[c]{0.75in}{\begin{center}Cycle \\ {\color{blue} min.}/{\color{red} max.} \end{center}} & \cellcolor{mygray} $T(2,1)$ & \cellcolor{mygray} $T(1,1)$ & \cellcolor{mygray} $T(0,1)$  \\ \hhline{|~|-|-|-|}
	Acyclic & \cellcolor{mygray} $T(2,0)$ & \cellcolor{mygray} $T(1,0)$ & \cellcolor{mygray} $T(0,0)$ \\ \hhline{|~|-|-|-|}
\end{tabular}
\caption{Four tables showing how the min-edge classes of generalized orientations are enumerated by generalized Tutte polynomial evaluations. Cells shaded in gray are enumerations that had been obtained in some form prior to this work or are trivial (see~\S\ref{subsec:history}).} \label{fig:fourtables}

\end{figure}
\restoregeometry
}

\restoregeometry

\section{Connections between min-edge classes and geometric, combinatorial, and algebraic objects} \label{sec:connections}

The mid-edge classes of fourientations enumerated in our main Theorem~\ref{thm:main} are not simply formal combinatorial objects.  In this section we illustrate their broader significance by highlighting connections to several different mathematical topics such as bigraphical arrangements, cycle-cocycle reversal systems, Riemann-Roch theory for graphs, graphic Lawrence ideals, zonotopal algebras, and the reliability polynomial.  In the future, we hope that a more unified understanding of these various relationships will arise which incorporates additional min-edge classes.

Throughout this section we fix an ordered, oriented graph~$\G = (G,\Oref,<)$ which has~$n := |V(G)|$ vertices and cyclomatic number~$g := |E(G)| - |V(G)| + 1$.  We will suppress mention of the reference orientation and edge order where it is not necessary (and thus for example speak of fourientations or partial orientations of $G$).

\subsection{Bi(co)graphical arrangements and cycle (cut) neutral partial orientations} \label{subsec:bigraph}

Cycle neutral partial orientations are related to the bigraphical arrangements originally defined by the second author and Perkinson~\cite{hopkins2012bigraphical}. We explain this relationship precisely here, and also define for the first time the object dual to the bigraphical arrangement, namely the bicographical arrangement. The bi(co)graphical arrangement depends on~$G$ as well as a \emph{parameter list} $A = (a_{e^{\pm}}) \in \mathbb{R}_{>0}^{\mathbb{E}(G)}$, which is a list of positive real parameters~$a_{e^{+}}, a_{e^{-}} \in \mathbb{R}_{>0}$ for each $e \in E(G)$ subject to the technical restriction that $a_{e_1^{\delta_1}} \neq a_{e_2^{\delta_2}}$ for $e_1^{\delta_1} \neq e_2^{\delta_2} \in \mathbb{E}(G)$ with $\Oref(e_1^{\delta_1}) = \Oref(e_2^{\delta_2})$. For an appropriate choice of parameters, the regions of the bi(co)graphical arrangement are in bijection with cycle (cut) neutral partial orientations; moreover the regions that avoid a certain generic hyperplane are in bijection with cut minimal-cycle neutral (cycle minimal-cut neutral) partial orientations, and the bounded regions are in bijection with strongly connected-cycle neutral (acyclic-cut neutral) partial orientations. The result concerning bounded regions of the bigraphical arrangement was essentially already proved in~\cite{hopkins2012bigraphical}, albeit in slightly different language. In general, for any hyperplane arrangement $\mathcal{A}$ these three region counts (total number of regions, number of regions avoiding a generic hyperplane, number of bounded regions) are given (up to sign) by evaluating the characteristic polynomial $\bigchi_{\mathcal{A}}(t)$ at~$t = -1,0,1$. These three characteristic polynoial evaluations allow us to explain an entire row (resp., column) in one of the tables in Figure~\ref{fig:fourtables} in terms of the bi(co)graphical arrangement.

The degenerate case of the bi(co)graphical arrangement where we set all the parameters $a_{e^{\pm}}$ to~$0$ recovers the (co)graphical arrangement. Many of the results here are extensions from total orientations to partial orientations of results obtained by Greene and Zaslavsky in~\cite{greene1983interpretation}, especially~\S8 of that paper which explores the cographical arrangement. In the proofs in this subsection we assume some familiarity with the theory of hyperplane arrangements, especially the notions of the intersection poset and characteristic polynomial of a hyperplane arrangement; see~\cite{stanley2007hyperplane} for all the relevant definitions and background information. 

\begin{definition}
Let $A = (a_{e^{\pm}}) \in \mathbb{R}_{>0}^{\mathbb{E}(G)}$ be a parameter list. Let~$W \simeq \mathbb{R}^{V(G)}$ be a real vector space with basis~$x_v$ for $v \in V$. Let $U \subseteq W$ be the subspace of $W$ where~$\sum_{v \in V(G)} x_{v} = 0$. The \emph{bigraphical arrangement} $\Sigma_{(G,\Oref)}(A) \subseteq U$ is
\[ \Sigma_{(G,\Oref)}(A) := \{H_{e^{+}} \cap U, H_{e^{-}} \cap U\colon e \in E(G) \textrm{ with $e$ not a loop}\}\]
where for a non-loop $e \in E(G)$ with $e^{\pm} = (u,v)$ we define \mbox{$H_{e^{\pm}} := x_v - x_u = a_{e^{\pm}}$}. Note that the bigraphical arrangement is an essential arrangement of $2|E(G)|$ hyperplanes in~$(n-1)$-dimensional space.
\end{definition}

\begin{definition}
Let $A \in \mathbb{R}_{>0}^{\mathbb{E}(G)}$ be a parameter list. Let $W \simeq \mathbb{R}^{E(G)}$ be a real vector space with basis $x_{e^{+}}$ for $e \in E(G)$, with the convention that $x_{e^{-}} = -x_{e^{+}}$. Let $U \subseteq W$ be the subspace of $W$ where for every $v \in V(G)$ we have $\sum_{e^{\pm} \in \mathbb{E}(\{v\},V(G)\setminus\{v\})} x_{e^{\pm}} = 0$. The \emph{bicographical arrangement} $\Sigma^{*}_{(G,\Oref)}(A) \subseteq U$ is
\[ \Sigma^{*}_{(G,\Oref)}(A) := \{H_{e^{+}} \cap U, H_{e^{-}} \cap U\colon e \in E(G)\}\]
where for $e \in E(G)$ we define $H_{e^{\pm}} := x_{e^{\pm}} = a_{e^{\pm}}$. Note that the bicographical arrangement is an essential arrangement of $2|E(G)|$ hyperplanes in $g$-dimensional space. (This is because $U$ is determined by $n$ linear equations, any $n-1$ of which are linearly independent, so its dimension is $|E(G)| - (n-1) = g$.)
\end{definition}

A region of a hyperplane a hyperplane arrangement $\mathcal{A}$ in $\mathbb{R}^{k}$ is a connected component of $\mathbb{R}^{k} \setminus \mathcal{A}$. In both the bigraphical and bicographical arrangements, the hyperplanes~$H_{e^{+}}$ and $H_{e^{-}}$ cut out a ``sandwich'' in space for each~$e \in E(G)$, so that for any region~$R$ of the arrangement exactly one of the following holds:
\begin{enumerate}[label=(\alph*),ref=(\alph*)]
\item $R$ is in the half-space of $U \setminus H_{e^{+}}$ opposite from $H_{e^{-}}$; \label{cond:hypplus}
\item $R$ is in the half-space of $U \setminus H_{e^{-}}$ opposite from $H_{e^{+}}$; \label{cond:hypminus}
\item $R$ is between $H_{e^{+}}$ and $H_{e^{-}}$. \label{cond:hypneut}
\end{enumerate}
Thus there is a natural map $R \mapsto \O_R$ that associates to any region $R$ of either the bigraphical or bicographical arrangement a partial orientation $\O_R$ of $(G,\Oref)$ whereby~$e \in E(G)$ is oriented as $e^{+}$ in case~\ref{cond:hypplus}, it is oriented as $e^{-}$ in case~\ref{cond:hypminus}, and it is left neutral in case~\ref{cond:hypneut}.\footnote{In the case of the bigraphical arrangement, if $e$ is a loop we did not include hyperplanes $H_{e^{\pm}}$ for~$e$ because they would lead to contradictory equations, but we can in fact consider these as hyperplanes ``at infinity'' and thus treat any region as ``between'' $H_{e^{+}}$ and $H_{e^{-}}$. Thus a loop will always be neutral in $\O_R$ for $R$ a region of $\Sigma_{(G,\Oref)}(A)$. It can similarly be seen that  an isthmus will always be neutral in~$\O_R$ for~$R$ a region of~$\Sigma^{*}_{(G,\Oref)}(A)$.} The second author and Perkinson~\cite{hopkins2012bigraphical} show that for a generic parameter list $A$ the number of regions of $\Sigma_{(G,\Oref)}(A)$ is given by a generalized Tutte polynomial evaluation. In order to make their input data compatible with the edge order $<$ used to define classes of partial orientation above we will fix a particular choice of generic parameters, namely, exponential parameters. We define the \emph{exponential parameter list} associated to~$<$ to be $A^< := (a^<_{e^{\pm}})$ where for each~$e \in E(G)$ we set $a^<_{e^{+}} := a^<_{e^{-}} := (1/2)^{i}$ if $e$ is the $i$th smallest edge according to $<$. That is, if~$e_1 < e_{2} < \cdots < e_m$ are the elements of $E(G)$, then $a^<_{e_i^{+}} = a^<_{e_i^{-}} = (1/2)^{i}$.

\begin{prop} \label{prop:bigbij}
The map $R \mapsto \O_R$ is a bijection between the regions of $\Sigma_{(G,\Oref)}(A^<)$ and the cycle neutral partial orientations of $\G$.
\end{prop}

\begin{proof}
Cleary $R \mapsto \O_R$ is injective as a map to partial orientations. It is shown in~\cite[Theorem 1.6]{hopkins2012bigraphical} that for a bigraphical arrangement $\Sigma_{(G,\Oref)}(A)$ with arbitrary parameter list $A$, the image of this map $R \mapsto \O_R$ is the set of so-called \emph{$A$-admissible partial orientations}. A partial orientation $\O$ is $A$-admissible if every potential cycle of $\O$ has a positive score with respect to $A$, where the score $\nu_A(C,\O)$ of a potential cycle $\OCy$ is given by
\[ \nu_A(\OCy,\O) := \sum_{\substack{e^{\pm} \in \mathbb{E}(\OCy), \\ e \notin \O}} a_{e^{\pm}} - \sum_{\substack{e^{\pm} \in \mathbb{E}(\OCy), \\ e^{\pm} \in \O}} a_{e^{\mp}}. \]
We are interested in the case of the exponential parameter list $A^<$. There is a simpler description of admissibility in this special case: a partial orientation is $A^<$-admissible precisely when the minimum edge in every potential cut is neutral because the contribution of this minimum edge in a potential cycle dominates the score of that cycle. In other words, a partial orientation is $A^<$-admissible precisely when it is cycle neutral. So indeed the image of $R \mapsto \O_R$ is the set of cycle neutral orientations.
\end{proof}

\begin{prop}
The map $R \mapsto \O_R$ is a bijection between the regions of $\Sigma^{*}_{(G,\Oref)}(A^<)$ and the cut neutral partial orientations of $\G$.
\end{prop}

\begin{proof}
This proposition is formally dual to the previous one. Using the same techniques as in~\cite{hopkins2012bigraphical} we can describe when a partial orientation is in the image of  $R \mapsto \O_R$ in terms of scores associated to potential cuts, and we will see that with $A^<$ the $A$-coadmissible partial orientations will be precisely the cut neutral ones. Alternatively, one could also prove, as in~\cite[Theorem 3.2]{hopkins2012bigraphical}, that because $A^<$ is generic the characteristic polynomial of~$\Sigma^{*}_{(G,\Oref)}(A^<)$ is $\bigchi_{\Sigma^{*}_{(G,\Oref)}(A^<)}(t) = (-2)^{g} T_G(1,1-t/2)$, which would show via Zaslavsky's theorem~\cite{zaslavsky1975facing}~\cite[Theorem 2.5]{stanley2007hyperplane} that there are at least the same number of cut neutral partial orientations as regions of $\Sigma^{*}_{(G,\Oref)}(A^<)$. Then it is easy to see that~$\O_R$ for~$R$ a region of~$\Sigma^{*}_{(G,\Oref)}(A^<)$ cannot have a potential cut whose minimum edge is directed, proving that the map is indeed a bijection.
\end{proof}

Compare the following propositions to~\cite[Corollary 8.2]{greene1983interpretation}.

\begin{prop} \label{prop:bigmin}
Let $M \gg 0$ be some large positive constant and define the hyperplane~$H_0 \subseteq \mathbb{R}^{n-1}$ by $H_0 := \sum_{e \in E(G), e^{+}=(u,v)} a^{<}_{e^{+}} (x_v - x_u) =-M$. Then the map~$R \mapsto \O_R$ is a bijection between the regions $R$ of $\Sigma_{(G,\Oref)}(A^<)$ with $R \cap H_0 = \emptyset$ and the cut minimal-cycle neutral partial orientations of $\G$.
\end{prop}

\begin{proof}
Let $\mathcal{A}$ be an essential arrangement of hyperplanes in $\mathbb{R}^k$. Let us say the hyperplane $H$ is \emph{generic with respect to $\mathcal{A}$} if for any $H_1,\ldots,H_m \in \mathcal{A}$, we have that $H$ has nonempty intersection with $H_1 \cap \ldots \cap H_m$ if and only if $\mathrm{dim}(H_1 \cap \ldots \cap H_m) \geq 1$. Suppose $H$ is generic with respect to $\mathcal{A}$. Then the number of regions $R$ of $\mathcal{A}$ such that~$H \cap R = \emptyset$ is given by $(-1)^{k}\bigchi_{\mathcal{A}}(0)$, where $\bigchi_{\mathcal{A}}$ is the characteristic polynomial of~$\mathcal{A}$. This assertion is~\cite[Theorem 3.1]{greene1983interpretation}.

We say that a parameter list $A$ is \emph{generic} if the arrangement $\Sigma_{(G,\Oref)}(A)$ is generic in the sense of~\cite[\S2]{stanley2007hyperplane}. It is shown in~\cite[Theorem 3.2]{hopkins2012bigraphical} that for generic $A$, the characteristic polynomial of~$\Sigma_{(G,\Oref)}(A)$ is $\bigchi_{\Sigma_{(G,\Oref)}(A)}(t) = (-2)^{n-1} T_G(1-t/2,1)$. Thus the previous claim tells us that the number of regions of $\Sigma_{(G,\Oref)}(A)$ for generic~$A$ that avoid a generic hyperplane is $2^{n-1}T_G(1,1)$, which is precisely the number of cut minimal-cycle positive partial orientations of $G$ by Theorem~\ref{thm:main}. It is easy to see that~$A^<$ is a generic parameter list and~$H_0$ is generic with respect to  $\Sigma_{(G,\Oref)}(A^<)$.

To finish the proof, we show that if $\O_R$ is not cut minimal, then $R$ must have nonempty intersection with $H_0$. Suppose $\O_R$ is not cut minimal. Then there is a directed cut $\OCu = (W,W^c)$ of $\O_R$ such that the orientation of the minimum edge in~$E(\OCu)$ disagrees with $\Oref$. Let $\mathbf{1}_{W^c} := \sum_{v \in W^c} x_v$ and let~$p$ be a point in $R$. Let~$L_0$ be the linear form $ \sum_{e \in E, e^{+}=(u,v)} a^{<}_{e^{+}} (x_v - x_u)$. Then~$L_0(p + t\mathbf{1}_{W^c}) = L_0(p) + Nt$ where 
\[N :=  \sum_{e^{+} \in \mathbb{E}(\OCu)} a^{<}_{e^{+}} \; -  \sum_{e^{-} \in \mathbb{E}(\OCu)}a^{<}_{e^{+}}.\] 
Note that $N$ is negative because the orientation of the minimum edge in $E(\OCu)$ disagrees with~$\Oref$. And note also that $p + t\mathbf{1}_{W^c} \in R$ for all $t \in [0,\infty)$. Thus we can find a point~$q \in R$ with~$L_0(q)$ arbitrarily small. This means $R$ intersects $H_0$ nontrivially as long as $M$ is taken to be sufficiently large.
\end{proof}

\begin{prop} \label{prop:cobigmin}
Let $M \gg 0$ be some large positive constant and  define the hyperplane~$H_0 \subseteq \mathbb{R}^{g}$ by $H_0 := \sum_{e \in E(G)} a^{<}_{e^{+}} x_{e^{+}} =-M$. Then the map $R \mapsto \O_R$ is a bijection between the regions $R$ of $\Sigma^{*}_{(G,\Oref)}(A^<)$ with $R \cap H_0 = \emptyset$ and the cycle minimal-cut neutral partial orientations of $\G$.
\end{prop}

\begin{proof}
Again, this proposition is formally dual to the previous one and the key is to compute the characteristic polynomial of~$\Sigma^{*}_{(G,\Oref)}(A^<)$.
\end{proof}

The following propositions should be seen as analogous to the fact that there are no strongly connected-acyclic total orientations, which agrees with there being no bounded regions of the (co)graphical arrangement.

\begin{prop}
The map $R \mapsto \O_R$ is a bijection between the regions of~$\Sigma_{(G,\Oref)}(A^<)$ that are bounded and the strongly connected-cycle neutral partial orientations of $\G$.
\end{prop}

\begin{proof}
It is shown in~\cite[Theorem 1.8]{hopkins2012bigraphical} that the bounded regions $R$ of $\Sigma_{(G,\Oref)}(A)$ for any parameter list $A$ are those for which $\O_R$  is $A$-admissible and such that every oriented edge in~$\O_R$ belongs to a potential cycle. Although they did not describe it in these terms, that is equivalent to saying that the bounded regions $R$ are those for which $\O_R$ is $A$-admissible and strongly connected because, in light of Proposition~\ref{prop:ordecomp}, each oriented edge in a partial orientation either belongs to a potential cycle or to a directed cut, but not both. Thus indeed the bounded regions $R$ of $\Sigma_{(G,\Oref)}(A^<)$ are those for which $\O_R$ is strongly connected and cycle neutral.
\end{proof}

\begin{prop}
The map $R \mapsto \O_R$ is a bijection between the regions of~$\Sigma^{*}_{(G,\Oref)}(A^<)$ that are bounded and the acyclic-cut neutral partial orientations of $\G$.
\end{prop}

\begin{proof}
This proposition is again dual to the previous one.
\end{proof}

\begin{example} \label{ex:hyper}
 Let $G$ be the triangle graph as below:
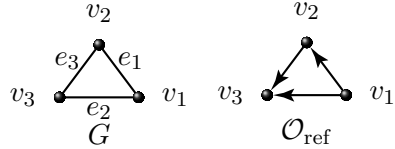
\begin{center}
\begin{tikzpicture}[scale=0.7]
	\SetFancyGraph
	\Vertex[LabelOut,Lpos=90, Ldist=.1cm,x=0,y=1]{v_2}
	\Vertex[LabelOut,Lpos=180, Ldist=.1cm,x=-0.75,y=0]{v_3}
	\Vertex[LabelOut,Lpos=0, Ldist=.1cm,x=0.75,y=0]{v_1}
	\Edges[style={thick}](v_1,v_2)
	\Edges[style={thick}](v_1,v_3)
	\Edges[style={thick}](v_2,v_3)
	\node at (-0.6,0.7) {$e_3$};
	\node at (0,-0.25) {$e_2$};
	\node at (0.6,0.7) {$e_1$};
	\node at (0,-0.7){$G$};
\end{tikzpicture} \begin{tikzpicture}[scale=0.7]
	\SetFancyGraph
	\Vertex[LabelOut,Lpos=90, Ldist=.1cm,x=0,y=1]{v_2}
	\Vertex[LabelOut,Lpos=180, Ldist=.1cm,x=-0.75,y=0]{v_3}
	\Vertex[LabelOut,Lpos=0, Ldist=.1cm,x=0.75,y=0]{v_1}
	\Edges[style={->,thick,>=mytip}](v_2,v_3)
	\Edges[style={->,thick,>=mytip}](v_1,v_3)
	\Edges[style={->,thick,>=mytip}](v_1,v_2)
	\node at (0,-0.7) {$\Oref$};
\end{tikzpicture}
\end{center}
Take $\Oref$ as above and let $<$ be given by $e_1 < e_2 < e_3$. The Tutte polynomial of~$G$ is~$T_G(x,y) = x^2 + x + y$. Figure~\ref{fig:bigraph} shows the bigraphical arrangement of~$\G$ together with a labeling of its regions by partial orientations. Note that there are~$2^{n-1}T_G(\frac{3}{2},1) = 19$ regions of $\Sigma_{(G,\Oref)}(A^<)$ and their labels are the cycle neutral partial orientations. There are~$2^{n-1}T_G(1,1) = 12$ regions of $\Sigma_{(G,\Oref)}(A^<)$ that avoid $H_0$, shaded in light and dark gray, and their labels are the cut minimal-cycle neutral partial orientations. There are~$2^{n-1}T_G(\frac{1}{2},1) = 7$ bounded regions of $\Sigma_{(G,\Oref)}(A^<)$, shaded in dark gray, and their labels are the strongly connected-cycle neutral partial orientations. Figure~\ref{fig:bicograph} depicts the bicographical arrangement of $G$ similarly labeled.
\end{example}

\begin{figure}
\begin{tikzpicture}
	
	\draw [fill = gray,opacity=0.5] (4,3) -- (-3.5,3) -- (-1.5,-1) -- (-3.75,-3.25) -- (3.625,-3.25) -- (1.675,0.675) -- cycle;

	\draw [fill = gray,opacity=0.25] (-5.75,3) -- (-3.5,3) -- (-1.5,-1) -- (-3.75,-3.25) -- (3.625,-3.25) -- (4.125,-4.25) -- (-5.75,-4.25) -- cycle;

	\draw [style=thick,color=red] (-0.5,4.1) -- (7,3.5);
	\node [color=red] at (0.9,4.2) {$H_0$};
	
	\draw [style=thick] (-4,4) -- (0.125,-4.25);
	\node at (-3.75,4.45) {$x_{v_3}-x_{v_1}=\frac{1}{4}$};
	\draw [style=thick] (0,4) -- (4.125,-4.25);
	\node at (-0.5,4.45) {$x_{v_1}-x_{v_3}=\frac{1}{4}$};
	
	\draw [style=thick] (3.5,4) -- (-4.75,-4.25);
	\node at (5.75,4.45) {$x_{v_2}-x_{v_3}=\frac{1}{8}$};
	\draw [style=thick] (5,4) -- (-3.25,-4.25);
	\node at (3.25,4.45) {$x_{v_3}-x_{v_2}=\frac{1}{8}$};
	
	\draw [style=thick] (-5.75,3) -- (5.5,3);
	\node at (5,2.75) {$x_{v_1}-x_{v_2}=\frac{1}{2}$};
	\draw [style=thick] (-5.75,-3.25) -- (5.5,-3.25);
	\node at (4.75,-3) {$x_{v_2}-x_{v_1}=\frac{1}{2}$};
	
	\Triangle[0,0,0,0,0]
	\Triangle[-1,1.5,-1,0,0]
	\Triangle[1,-1.5,1,0,0]
	\Triangle[3,0,1,1,0]
	\Triangle[-3,0,-1,-1,0]
	\Triangle[1.7,1.75,0,1,0]
	\Triangle[-1.8,-1.8,0,-1,0]
	\Triangle[1.23,2.6,-1,1,0]
	\Triangle[-1.25,-2.75,1,-1,0]
	\Triangle[-1.5,-3.75,1,-1,-1]
	\Triangle[-3.75,-3.75,0,-1,-1]
	\Triangle[-5.25,-3.75,-1,-1,-1]
	\Triangle[2.25,-3.75,1,0,-1]
	\Triangle[4.5,-3.75,1,1,-1]
	\Triangle[1.25,3.75,-1,1,1]
	\Triangle[-2.25,3.75,-1,0,1]
	\Triangle[-5,3.75,-1,-1,1]
	\Triangle[3.6,3.75,0,1,1]
	\Triangle[5.25,3.75,1,1,1]
	
\end{tikzpicture}
\caption{The bigraphical arrangement $\Sigma_{(G,\Oref)}(A^<)$ in Example~\ref{ex:hyper}. The hyperplane $H_0 = \frac{1}{2}(x_{v_2}-x_{v_1}) + \frac{1}{4}(x_{v_3}-x_{v_1}) + \frac{1}{8}(x_{v_3}-x_{v_2}) = -M$ from Proposition~\ref{prop:bigmin} is depicted in red.} \label{fig:bigraph}
\end{figure}
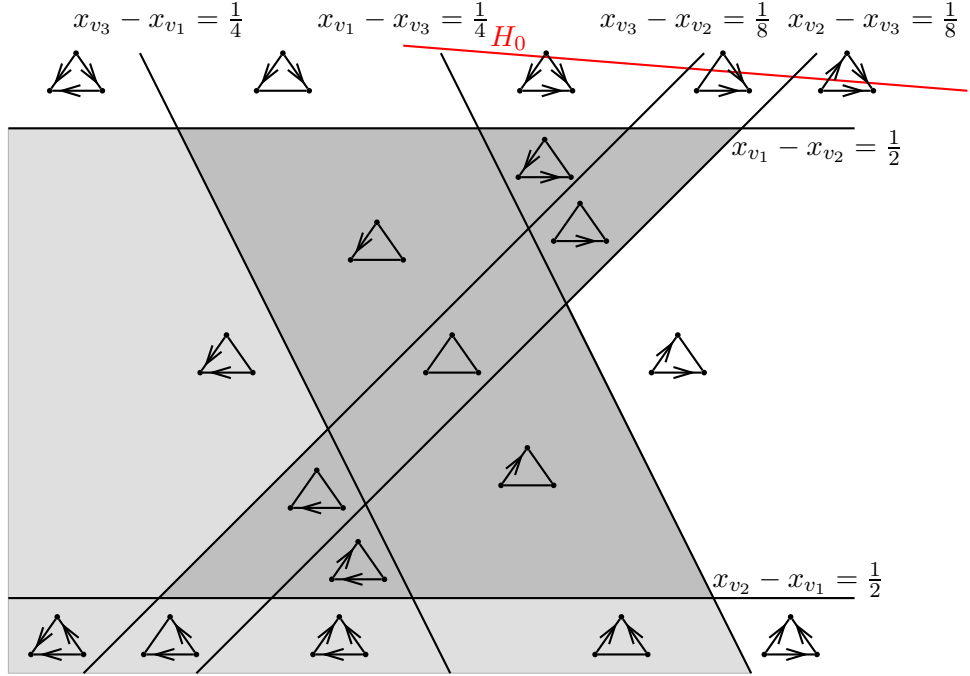

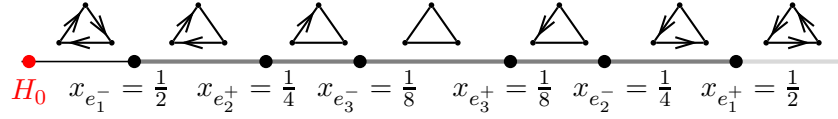
\begin{figure}
\begin{tikzpicture}
	\draw [style=semithick] (-5.5,0) -- (5.5,0);
	\draw [color = gray,style = ultra thick] (-4,0) -- (4,0);
	\draw [color = mygray, style = ultra thick] (4,0) -- (5.5,0);

	\node[circle,fill=black,minimum size=5pt,inner sep=0pt] at (1,0){};
	\node at (0.9,-0.4) {$x_{e^+_3} = \frac{1}{8}$};
	\node[circle,fill=black,minimum size=5pt,inner sep=0pt] at (-1,0){};
	\node at (-0.9,-0.4) {$x_{e^-_3} = \frac{1}{8}$};
	\node[circle,fill=black,minimum size=5pt,inner sep=0pt] at (2.25,0){};
	\node at (2.5,-0.4) {$x_{e^-_2} = \frac{1}{4}$};
	\node[circle,fill=black,minimum size=5pt,inner sep=0pt] at (-2.25,0){};
	\node at (-2.5,-0.4) {$x_{e^+_2} = \frac{1}{4}$};
	\node[circle,fill=black,minimum size=5pt,inner sep=0pt] at (4,0){};
	\node at (4.2,-0.4) {$x_{e^+_1} = \frac{1}{2}$};
	\node[circle,fill=black,minimum size=5pt,inner sep=0pt] at (-4,0){};
	\node at (-4.2,-0.4) {$x_{e^-_1} = \frac{1}{2}$};
	\Triangle[-0.2,0.5,0,0,0]
	\Triangle[1.5,0.5,-1,0,0]
	\Triangle[-1.7,0.5,1,0,0]
	\Triangle[3.1,0.5,-1,1,0]
	\Triangle[-3.3,0.5,1,-1,0]
	\Triangle[4.6,0.5,-1,1,-1]
	\Triangle[-4.8,0.5,1,-1,1]
	
	\node[circle,fill=red,minimum size=5pt,inner sep=0pt] at (-5.4,0){};
	\node[color=red] at (-5.4,-0.4) {$H_0$};
	
\end{tikzpicture}
\caption{The bicographical arrangement $\Sigma^{*}_{(G,\Oref)}(A^<)$ in Example~\ref{ex:hyper}. The hyperplane $H_0 = \frac{1}{2}(x_{e^+_1}) + \frac{1}{4}(x_{e^{+}_2}) + \frac{1}{8}(x_{e^{+}_3}) = -M$ from Proposition~\ref{prop:cobigmin} is depicted in red. } \label{fig:bicograph}
\end{figure}

\begin{remark} \label{rem:param}
Let $W \simeq \mathbb{R}^{E(G)}$ be a real vector space with basis $x_{e^{+}}$ for $e \in E(G)$ with the convention that $x_{e^{-}} = -x_{e^{+}}$. Let $\mathcal{E}_{(G,\Oref)}(A) := \{H_{e^{\pm}}\}$ be the ``perturbed coordinate hyperplane arrangement'' in $W$ with $H_{e^{\pm}} := x_{e^{\pm}} = a_{e^{\pm}}$. For a directed cut or cycle $\OC$ of $G$ define the vector $x_{\OC} := \sum_{e^{\pm} \in \mathbb{E}(\OC)}x_{e^{\pm}} \in W$. The bicographical arrangement is the projection of $\mathcal{E}_{(G,\Oref)}(A)$ to the subspace $U\subseteq W$ where $x_{\OCu} = 0$ for all directed cuts $\OCu$ of $G$. In order to make the bigraphical arrangement look more like the bicographical arrangement, one can also view it as the projection of~$\mathcal{E}_{(G,\Oref)}(A)$ to the subspace $U'\subseteq W$ where $x_{\OCy} = 0$ for all directed cycles $\OCy$ of $G$. Consequently one might wonder which other min-edge classes of partial orientations can be described by projecting~$\mathcal{E}_{(G,\Oref)}(A^<)$ to various subspaces: for instance, the set of all partial orientations is naturally in bijection with the regions of~$\mathcal{E}_{(G,\Oref)}(A^<)$. 
\end{remark}

\begin{remark}
There is another notion of acyclicity for partial orientations which is not to be confused with our acyclic partial orientations. In a recent paper of Iriarte~\cite{iriarte2014acyclic} this other kind of acyclic partial orientation is called a  ``partial acyclic orientation.'' A partial acyclic orientation is one for which the contraction of all neutral edges yields an acyclic total orientation. By contrast, the acyclic partial orientations studied in this paper are those such that the \emph{deletion} of all neutral edges yields an acyclic total orientation. There is a bijection between the partial acyclic orientations of a graph and the faces (i.e., the regions and the faces of lower dimension) of its ordinary graphical arrangement (see Greene-Zaslavksy~\cite[Lemma 7.2]{greene1983interpretation} or Zaslavsky~\cite[Corollary 4.6]{zaslavsky1991orientation}, who proves a stronger version of this result that holds at the level of signed graphs). Recast in our terminology, these partial acyclic orientations are the partial orientations whose only potential cycles consist of all neutral edges. Apparently the partial acyclic orientations are not enumerated in a simple way by the Tutte polynomial. However, we remark that these partial acyclic orientations are precisely the partial orientations that are cycle neutral for all choices of edge order~$<$. They are also the partial orientations that are $A$-admissible for all choices of parameter list~$A\in\mathbb{R}_{>0}^{\mathbb{E}(G)}$.

\end{remark}

\subsection{The cycle/cocycle reversal systems}

Gioan \cite{gioan2007enumerating} investigated the set of total orientations modulo directed cycle and/or directed cut (cocycle) reversals and he used this setup to give a unified framework for understanding the evaluations~$T(x,y)$ for~$0 \leq x,y \leq 2$ integral.  Each equivalence class in the cycle/cocycle reversal systems contains a unique cycle/cut minimal orientation and so these objects give distinguished representatives.  Thus the~$3 \times 3$ table at the bottom of Figure \ref{fig:fourtables} is equivalent to Gioan's~$3 \times 3$ square.  Furthermore, Gioan also showed that the two orientations are in the same equivalence class of the cycle-cocycle reversal system if and only if their associated indegree sequences are equivalent by \emph{chip-firing moves}, which we now describe: given a chip configuration, which is simply a function from the vertices to the integers, a vertex \emph{fires} by sending a chip to each of its neighbors and losing its degree number of chips in the process; we say that two chip configurations~$D$ and~$D'$ are \emph{chip-firing equivalent} if we can get from one to the other by a sequence of chip-firings moves.  Equivalently, if we view~$D$ and~$D'$ as vectors, then they are chip-firing equivalent when their difference is in the integer span of the columns of the Laplacian matrix of~$G$.

In \cite{backman2014partial} and ~\cite{backman2014riemann} the first author investigated two different extensions of Gioan's cycle-cocycle reversal systems for partial orientations.  One extension, which we call the \emph{cycle/cocycle reversal systems for partial orientations} describes the set of partial orientations modulo cycle and/or cocycle reversals. The definition of (co)cycle reversals for partial orientations are exactly the same as for total orientations: given a partial orientation~$\O$ of $\G$, a \emph{(co)cycle reversal} is the operation of replacing~$\O$ by~$(\O \setminus \mathbb{E}(\OC)) \cup \mathbb{E}(-\OC)$ for some directed cycle (cut) $\OC$ of $\O$. These cycle/cocycle reversal systems are related to the graphic and cographic Lawrence ideals from combinatorial commutative algebra and in~\cite{backman2014partial} it was demonstrated that they define equivalence classes of partial orientations counted by generalized Tutte polynomial evaluations. Each equivalence class in the (co)cycle reversal system contains a unique cycle (cut) minimal partial orientation. The (co)graphic Lawrence ideals and their connection to fourientations are discussed in~\S\ref{sec:lawrence}. The other extension, which we call the \emph{generalized cycle/cocycle reversal systems for partial orientations} was introduced in~\cite{backman2014riemann} for the study of chip-firing in the context of Baker and Norine's combinatorial Riemann-Roch theorem~\cite{baker2007riemann}.  In the next section we explain how this extension allows for a direct and aesthetically pleasing interpretation of the graphical Riemann-Roch duality in terms of fourientations.  At the time of writing, the precise connection between the Tutte polynomial and the generalized cycle/cocycle reversal systems remains a mystery.

\subsection{The generalized cycle/cocycle reversal systems and Riemann-Roch theory for fourientations}

In~\cite{backman2014riemann}, an \emph{edge pivot} for partial orientations was defined as follows: given an edge~$e$ oriented towards a vertex~$v$ and~$e'$ a neutral edge incident to~$v$, we may unorient~$e$ and orient~$e'$ towards~$v$.  This name is motivated by the image of an oriented edge nailed down at its head which can pivot to other unoriented edges.  The generalized cycle, cocycle and cycle-cocycle reversal systems for partial orientations are defined to be these systems extended to partial orientations by the inclusion of edge pivots.  

We now introduce generalized edge pivots for fourientations, which we will refer to as simply edge pivots. Let~$e$ and~$e'$ be a pair of edges incident to~$v$.  Suppose that~$e$ is bioriented or is oriented towards~$v$ and~$e'$ is either unoriented or oriented away from~$v$.  Then we can remove the orientation of~$e$ towards~$v$ and add an orientation of~$e'$ towards~$v$.  That is, if $\O$ is a fourientation with~$e_1^{\delta_1} = (v,u) \in \O$ but~$e_2^{\delta_2} = (w,u) \notin \O$, then an \emph{edge pivot} is the operation of replacing $\O$ by~$\O' = (\O \setminus  \{e_1^{\delta_1}\}) \cup \{e_2^{\delta_2}\}$. See Figure~\ref{fig:edgepivots} for the different combinatorial types of edge pivots. The generalized cycle, cocycle and cycle-cocycle reversal systems for fourientations are defined to be these systems extended to fourientations by the inclusion of generalized edge pivots.  To clarify, we can only reverse a directed cut or cycle in a fourientation if none of the edges are unoriented or bioriented.  We write~$\O \sim \O'$ if the fourientations~$\O$ and~$\O'$ are equivalent in the generalized cocycle reversal system.

\begin{figure}
\begin{tikzpicture}[scale=0.8]
	\SetFancyGraph
	\Vertex[NoLabel,x=-0.75,y=0]{v_1}
	\Vertex[NoLabel,x=0,y=1]{v_2}
	\Vertex[NoLabel,x=0.75,y=0]{v_3}
	\Edges[style={thick,->,>=mytip,dash pattern=on 0pt off 100pt}](v_1,v_2)
	\Edges[style={thick,->,>=mytip}](v_2,v_1)
	\Edges[style={thick,->,>=mytip}](v_2,v_3)
\end{tikzpicture} \parbox[b][0.3in][c]{0.4in}{\begin{center} \huge $\Leftrightarrow$ \end{center} \vfill} \begin{tikzpicture}[scale=0.8]
	\SetFancyGraph
	\Vertex[NoLabel,x=-0.75,y=0]{v_1}
	\Vertex[NoLabel,x=0,y=1]{v_2}
	\Vertex[NoLabel,x=0.75,y=0]{v_3}
	\Edges[style={thick,->,>=mytip,dash pattern=on 0pt off 100pt}](v_3,v_2)
	\Edges[style={thick,->,>=mytip}](v_2,v_1)
	\Edges[style={thick,->,>=mytip}](v_2,v_3)
\end{tikzpicture} \qquad \begin{tikzpicture}[scale=0.8]
	\SetFancyGraph
	\Vertex[NoLabel,x=-0.75,y=0]{v_1}
	\Vertex[NoLabel,x=0,y=1]{v_2}
	\Vertex[NoLabel,x=0.75,y=0]{v_3}
	\Edges[style={thick,->,>=mytip}](v_1,v_2)
	\Edges[style={thick}](v_2,v_3)
\end{tikzpicture} \parbox[b][0.3in][c]{0.4in}{\begin{center} \huge $\Leftrightarrow$ \end{center} \vfill} \begin{tikzpicture}[scale=0.8]
	\SetFancyGraph
	\Vertex[NoLabel,x=-0.75,y=0]{v_1}
	\Vertex[NoLabel,x=0,y=1]{v_2}
	\Vertex[NoLabel,x=0.75,y=0]{v_3}
	\Edges[style={thick,->,>=mytip}](v_3,v_2)
	\Edges[style={thick}](v_2,v_1)
\end{tikzpicture} \qquad \begin{tikzpicture}[scale=0.8]
	\SetFancyGraph
	\Vertex[NoLabel,x=-0.75,y=0]{v_1}
	\Vertex[NoLabel,x=0,y=1]{v_2}
	\Vertex[NoLabel,x=0.75,y=0]{v_3}
	\Edges[style={thick,->,>=mytip,dash pattern=on 0pt off 100pt}](v_1,v_2)
	\Edges[style={thick,->,>=mytip}](v_2,v_1)
	\Edges[style={thick}](v_2,v_3)
\end{tikzpicture} \parbox[b][0.3in][c]{0.4in}{\begin{center} \huge $\Leftrightarrow$ \end{center} \vfill} \begin{tikzpicture}[scale=0.8]
	\SetFancyGraph
	\Vertex[NoLabel,x=-0.75,y=0]{v_1}
	\Vertex[NoLabel,x=0,y=1]{v_2}
	\Vertex[NoLabel,x=0.75,y=0]{v_3}
	\Edges[style={thick,->,>=mytip}](v_2,v_1)
	\Edges[style={thick,->,>=mytip}](v_3,v_2)
\end{tikzpicture}
\caption{The various types of edge pivots.} \label{fig:edgepivots}
\end{figure}
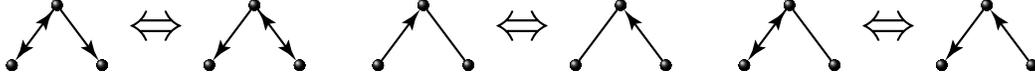

\begin{remark} A cycle reversal in a fourientation can be performed by a sequence of generalized edge pivots as depicted in Figure~\ref{fig:cyclerev}.  Thus the generalized cycle-cocycle reversal system and the generalized cocycle reversal system for fourientations agree.  
\end{remark}

\begin{figure}
\begin{tikzpicture}[scale=0.8]
	\SetFancyGraph
	\Vertex[NoLabel,x=-0.75,y=0]{v_1}
	\Vertex[NoLabel,x=0,y=1]{v_2}
	\Vertex[NoLabel,x=0.75,y=0]{v_3}
	\Edges[style={thick,->,>=mytip}](v_3,v_1)
	\Edges[style={thick,->,>=mytip}](v_2,v_3)
	\Edges[style={thick,->,>=mytip}](v_1,v_2)
	\tikzset{VertexStyle/.style = {shape = circle,fill = black,minimum size = 0pt,inner sep=0pt}}
	\Vertex[NoLabel,x=-0.25,y=0.9]{v_7}
	\Vertex[NoLabel,x=0.25,y=0.9]{v_8}
	\Edges[style={thick,->--,color=blue,out=90,in=90,looseness=2}](v_7,v_8)
	\Vertex[NoLabel,x=0.8,y=0.2]{v_9}
	\Vertex[NoLabel,x=0.55,y=-0.1]{v_a}
	\Edges[style={thick,->--,color=blue,out=-20,in=-70,looseness=3.5}](v_9,v_a)
	\Vertex[NoLabel,x=-0.55,y=-0.1]{v_b}
	\Vertex[NoLabel,x=-0.8,y=0.2]{v_c}
	\Edges[style={thick,->--,color=blue,out=-110,in=200,looseness=3.5}](v_b,v_c)
	\node at (-0.4,-0.4) {};
\end{tikzpicture} \parbox[b][0.4in][c]{0.4in}{\begin{center} \color{blue} \huge $\Rightarrow$ \end{center} \vfill} \begin{tikzpicture}[scale=0.8]
	\SetFancyGraph
	\Vertex[NoLabel,x=-0.75,y=0]{v_1}
	\Vertex[NoLabel,x=0,y=1]{v_2}
	\Vertex[NoLabel,x=0.75,y=0]{v_3}
	\Edges[style={thick,->,>=mytip}](v_1,v_3)
	\Edges[style={thick,->,>=mytip}](v_3,v_2)
	\Edges[style={thick,->,>=mytip}](v_2,v_1)
	\node at (-0.4,-0.4) {};
\end{tikzpicture}
\caption{A cycle reversal performed by edge pivots.} \label{fig:cyclerev}
\end{figure}
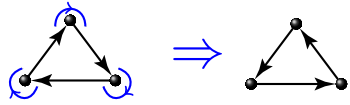

For a fourientation~$\O$ of~$G$ and a vertex~$v \in V(G)$ we define the \emph{indegree of~$\O$ at~$v$} to be~$\mathrm{indeg}_{\O}(v) := |\{e^{\pm} = (u,v) \in \O\}|$. In keeping with the terminology of algebraic geometry, we define the \emph{divisor} associated to the fourientation~$\O$ to be~$D_{\O} := \sum_{v \in V(G)} (\mathrm{indeg}_{\O}(v)-1)(v)$ viewed as a formal sum of the vertices with integer coefficients.  Similarly, given two divisors~$D$ and~$D'$ we write~$D \sim D'$ if they are equivalent by chip-firing moves and say they are \emph{linearly equivalent}. See~\cite{perkinson2013primer} for background on linear equivalence of divisors.  We note that our terminology is justified by the rich connection between combinatorial divisor theory for graphs and chip-firing~\cite{baker2007riemann}~\cite{cools2012tropical}~\cite{mikhalkin2006tropical}~\cite{bacher1997lattice}.  Lemma 3.1 of~\cite{backman2014riemann} says that two partial orientations are equivalent in the generalized cycle reversal system if and only if they have the same associated divisors, which extends Gioan's~\cite[Proposition 4.10]{gioan2007enumerating} from total to partial orientations.  We now further extend this result to the setting of fourientations.  

\begin{lemma} \label{lem:edgepivot}
Two fourientations~$\O$ and~$\O'$ are equivalent by edge pivots if and only if~$D_{\O} = D_{\O'}$.
\end{lemma}

\begin{proof}
It is clear that if~$\O$ and~$\O'$ are equivalent in the generalized cycle reversal system then~$D_{\O} = D_{\O'}$.  We now demonstrate the converse.  First suppose that there exists some edge~$e = (u,v)$ such that~$e$ is oriented towards~$v$ in~$\O$, but~$e$ is bioriented in~$\O'$.  Because~$D_{\O} = D_{\O'}$ we know that there exists some~$e'$ incident to~$u$ such that~$e'$ is not oriented towards~$u$ in~$\O$.  We can perform a pivot from~$e$ to~$e'$ in~$\O$.  By induction on the symmetric difference of~$\O$ and~$\O'$ we may assume that no such edge exists.  Therefore their symmetric difference is a Type~A fourientation and we reduce to Lemma 3.1 of~\cite{backman2014riemann}.
\end{proof}

In~\cite{backman2014riemann} the first author introduced a ``nonlocal" extension of an edge pivot called a \emph{Jacob's ladder cascade} and employed this operation repeatedly.  We now extend this operation to fourientations.  Let~$P$ be a directed path from~$u$ to~$v$ in the fourientation~$\O$~(i.e.,~$P$ is a path from~$u$ to~$v$ that walks along oriented edges). Let~$e_1$ and~$e_2$ be edges not in~$P$ such that~$e_1^{\delta_1} = (x,u)$,~$e_2^{\delta_2} = (y,v)$ with~$e_1^{\delta_1} \in \O$ and~$e_2^{\delta_2} \notin \O$. Then we can perform successive edge pivots along~$P$ to so that~$e_1^{\delta_1} \notin \O$ and~$e_2^{\delta_2} \in \O$ and we call this operation a \emph{Jacob's ladder cascade}; see Figure~\ref{fig:jacob}.  We note that our definition allows for~$e_1 = e_2 = \{u,v\}$ and hence a cycle reversal may be viewed as a special case of a Jacob's ladder cascade.

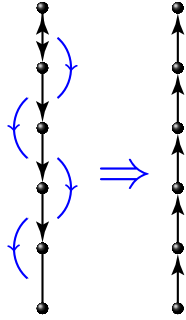
\begin{figure}
\begin{tikzpicture}[scale=0.8]
	\SetFancyGraph
	\Vertex[NoLabel,x=0,y=0]{v_1}
	\Vertex[NoLabel,x=0,y=1]{v_2}
	\Vertex[NoLabel,x=0,y=2]{v_3}
	\Vertex[NoLabel,x=0,y=3]{v_4}
	\Vertex[NoLabel,x=0,y=4]{v_5}
	\Vertex[NoLabel,x=0,y=5]{v_6}
	\Edges[style={thick,->,>=mytip,dash pattern=on 0pt off 100pt}](v_5,v_6)
	\Edges[style={thick,->,>=mytip}](v_6,v_5)
	\Edges[style={thick,->,>=mytip}](v_5,v_4)
	\Edges[style={thick,->,>=mytip}](v_4,v_3)
	\Edges[style={thick,->,>=mytip}](v_3,v_2)
	\Edges[style={thick}](v_2,v_1)
	\tikzset{VertexStyle/.style = {shape = circle,fill = black,minimum size = 0pt,inner sep=0pt}}
	\Vertex[NoLabel,x=0.25,y=4.5]{v_7}
	\Vertex[NoLabel,x=0.25,y=3.5]{v_8}
	\Edges[style={thick,->--,color=blue,bend left=50}](v_7,v_8)
	\Vertex[NoLabel,x=-0.25,y=3.5]{v_9}
	\Vertex[NoLabel,x=-0.25,y=2.5]{v_a}
	\Edges[style={thick,->--,color=blue,bend right=50}](v_9,v_a)
	\Vertex[NoLabel,x=0.25,y=2.5]{v_b}
	\Vertex[NoLabel,x=0.25,y=1.5]{v_c}
	\Edges[style={thick,->--,color=blue,bend left=50}](v_b,v_c)
	\Vertex[NoLabel,x=-0.25,y=1.5]{v_d}
	\Vertex[NoLabel,x=-0.25,y=0.5]{v_e}
	\Edges[style={thick,->--,color=blue,bend right=50}](v_d,v_e)
\end{tikzpicture} \parbox[b][0.8in][c]{0.4in}{\begin{center} \color{blue} \huge $\Rightarrow$ \end{center} \vfill} \begin{tikzpicture}[scale=0.8]
	\SetFancyGraph
	\Vertex[NoLabel,x=0,y=0]{v_1}
	\Vertex[NoLabel,x=0,y=1]{v_2}
	\Vertex[NoLabel,x=0,y=2]{v_3}
	\Vertex[NoLabel,x=0,y=3]{v_4}
	\Vertex[NoLabel,x=0,y=4]{v_5}
	\Vertex[NoLabel,x=0,y=5]{v_6}
	\Edges[style={thick,->,>=mytip}](v_5,v_6)
	\Edges[style={thick,->,>=mytip}](v_4,v_5)
	\Edges[style={thick,->,>=mytip}](v_3,v_4)
	\Edges[style={thick,->,>=mytip}](v_2,v_3)
	\Edges[style={thick,->,>=mytip}](v_1,v_2)
\end{tikzpicture}
\caption{A Jacob's ladder cascade.} \label{fig:jacob}
\end{figure}

Given a fourientation~$\O$, we define~$\O^c$ to be the fourientation obtained by reversing the orientation of each directed edge, replacing each unoriented edge with a bioriented edge, and replacing each bioriented edge with an unoriented edge.  In other words, we simply set~$\O^c := \mathbb{E}(G) \setminus \O$. Recall that the canonical divisor of~$G$ is~$K = \sum_{v \in V(G)} (\mathrm{deg}(v)-2)(v)$ where the \emph{degree} of~$v \in V(G)$ is~$\mathrm{deg}(v) := |\{e = \{u,v\}\colon e \in E(G), u \in V(G)\}|$.  Baker and Norine's Riemann-Roch formula for graphs~\cite{baker2007riemann} investigates the \emph{rank} of a divisor~$D$, written~$r(D)$, in comparison to~$r(K-D)$.  We do not review the definition of rank here, nor the Riemann-Roch formula, but we note the following important observation.

\begin{remark}\label{RRrmk}
 If~$\O$ is a fourientation, then~$K-D_{\O}= D_{\O^c}$. Thus the divisors associated to complementary fourientations are Riemann-Roch dual.
 \end{remark}

\begin{lemma}\label{lem:dual}
If~$\O$ and~$\O'$ are fourientations, then~$\O \sim \O'$ if and only if~$\O^c \sim \O^{c'}$.
\end{lemma}
\begin{proof}
This is trivial.
\end{proof}

\begin{lemma}\label{lem:typered}
Let~$\O$ be a fourientation, then 
\begin{enumerate}[label=(\roman*)]
\item~$\O \sim \O'$ with~$\O'$ a Type A fourientation if and only if~${\rm deg}(D_{\O}) \leq g-1$; \label{cond:rrtypea}
\item~$\O \sim \O'$ with~$\O'$ a Type B fourientation if and only if~${\rm deg}( D_{\O}) \geq g-1$. \label{cond:rrtypeb}
\end{enumerate}
\end{lemma}

\begin{proof}
We have that~${\rm deg}(D_{\O}) \leq g-1$ if and only if~${\rm deg}(D_{\O^c}) \geq g-1$, and~$\O$ is Type~A if and only if~$\O^c$ is Type~B, thus Lemma~\ref{lem:dual} shows that~\ref{cond:rrtypeb} is equivalent to~\ref{cond:rrtypea}.  We now verify~\ref{cond:rrtypea}.  It is clear that if~$\O \sim \O'$ with~$\O'$ a type~$A$ partial orientation, then~${\rm deg}( D_{\O}) \leq g-1$.  Conversely, suppose~${\rm deg}( D_{\O}) \leq g-1$ and~$\O$ is not a Type A fourientation.  Let~$S$ be the set of vertices incident to a bioriented edge and~$T$ be the set of edges incident to an unoriented edge.  By assumption, both~$S$ and~$T$ are non-empty.  Furthermore, we take~$\overline{S}$ to be the set of vertices which are reachable from~$S$ by a (possibly empty) directed path.  If~$\overline{S} \cap T \neq \emptyset$ then we may perform a Jacob's ladder cascade to decrease the number of bioriented edges.  By induction on the number of bioriented edges in~$\O$ we can assume that eventually~$\overline{S} \cap T = \emptyset$.  Therefore~$(\overline{S}^c, \overline{S})$ is fully oriented towards~$\overline{S}$, and we can reverse this directed cut enlarging~$\overline{S}$.  By induction on~$|\overline{S}^c|$ this process must terminate.
\end{proof}

Theorem 3.4 of~\cite{backman2014riemann} states that two partial orientations are equivalent in the generalized cycle-cocycle reversal system if and only if their associated divisors are chip-firing equivalent.  This extends Gioan's~\cite[Proposition 4.13]{gioan2007enumerating} from total orientations to partial orientations.  We now extend this theorem further to the setting of fourientations.  
 
\begin{thm}\label{thm:cocycle}
If~$\O$ and~$\O'$ are fourientations, then~$\O \sim \O'$ if and only if~$D_{\O} \sim D_{\O'}$.
\end{thm}

\begin{proof}
It is clear that if~$\O$ and~$\O'$ are equivalent in the generalized cocycle reversal system, then~$D_{\O} \sim D_{\O'}$.  We now demonstrate the converse.  Lemma~\ref{lem:dual} in conjunction with the fact that~$D_{\O} \sim D_{\O'}$ if and only if~$D_{\O^c} \sim D_{(\O')^c}$ allows us to assume that~${\rm deg}(D_{\O}) \leq g-1$.  By Lemma \ref{lem:typered}~$\O \sim \O_1$ and~$\O' \sim \O_2$ such that both~$\O_1$ and~$\O_2$ are Type~A fourientations.  We know that~$D_{\O} \sim D_{\O_1}$ and~$D_{\O'} \sim D_{\O_2}$, thus by transitivity~$D_{\O_1} \sim D_{\O_2}$.  Now by Theorem 3.4 in \cite{backman2014riemann} we have~$\O_1 \sim \O_2$ and again by transitivity~$\O \sim \O'$.
\end{proof}

\subsection{Indegree sequences of partial orientations} 

For a fourientation~$\O$ of~$G$, define~$\overline{D}_{\O} :=  \sum_{v \in V(G)} \mathrm{indeg}_{\O}(v)(v) \in \mathbb{Z}V(G)$. We call~$\overline{D}_{\O}$ the \emph{indegree sequence} of~$\O$. Recall the divisor associated to~$\O$ is~$D_{\O} := \sum_{v \in V(G)} (\mathrm{indeg}_{\O}(v)-1)(v) \in \mathbb{Z}V(G)$; the distinction between the divisor associated to a fourientation and its indegree sequence is just one of normalization. As mentioned in the last section, results of the first author imply that studying partial orientations up to ``having the same indegree sequence'' is the same as studying equivalence classes of partial orientations in the generalized cycle reversal system (see the discussion above Lemma~\ref{lem:edgepivot}). In this subsection we explore the number of indegree sequences in various classes of partial orientations. Here we will assume all graphs are loopless since loops only affect indegree sequences in a trivial way. With this assumption any acyclic partial orientation of a graph can be extended to an acyclic total orientation. We might hope that the number of indegree sequences among partial orientations in a min-edge class is also given by a generalized Tutte polynomial evaluation. But, as observed in~\cite{backman2014partial}, the number of indegree sequences among all partial orientations of~$G$ cannot be a generalized evaluation of the Tutte polynomial of~$G$ itself: for example, the path on three edges has~$21$ indegree sequences among its partial orientations while the star on three edges has~$20$, but the Tutte polynomials of all trees on~$n$ vertices are the same. This also shows that the number of indegree sequences of acyclic partial orientations of~$G$ is not a generalized Tutte polynomial evaluation since all partial orientations of a tree are of course acyclic. One way to get around this obstruction is by considering the Tutte polynomial of graphs related to~$G$. Let us denote by~$G^{\bullet}$ the \emph{cone over~$G$}, which is the graph obtained from~$G$ by adding an extra vertex~$v_0$ and connecting it by an edge to every other vertex in~$G$. Note that the cone over the path on three edges and the cone over the star on three edges have different Tutte polynomials. It turns out that the set~$\{\overline{D}_{\O} \colon \O \textrm{ an acyclic partial orientation of~$G$} \}$ is the set of~$(G^{\bullet},v_0)$-parking functions and thus the cardinality of this set is~$T_{G^{\bullet}}(1,1)$. We first need some terminology to explain why this is.

\begin{definition}
Let~$G$ be a graph and designate a special \emph{sink} vertex~$q \in V(G)$. Set~$V^{q}(G) := V(G) \setminus \{q\}$. A \emph{$(G,q)$-parking function} is an element~$c = \sum_{v \in V^{q}(G)}c_v (v)$ of~$\mathbb{Z}V^{q}(G)$ so that for every non-empty~$U \subseteq V^{q}(G)$, there is~$u \in U$ with~$0\leq c_u < d^G_U(u)$, where~$d^G_U(u) := |\{ e=\{u,v\} \in E(U, V(G)\setminus U)\}|$. We denote the set of $(G,q)$-parking functions by $\mathrm{PF}(G,q)$. The set of $(G,q)$-parking function inherits a natural partial order from~$\mathbb{Z}V^{q}(G)$. A \emph{maximal} $(G,q)$-parking function is one that is maximal among elements of~$\mathrm{PF}(G,q)$ with respect to this order.
\end{definition}

 A \emph{source} of a total orientation is a vertex with no incoming directed edges. The following lemma, which is the main tool that will allow us to count indegree sequences of partial orientations, is well-known.
 
 \begin{lemma}[{See~\cite[Theorem~3.1]{benson2010g}}]\label{lem:maxpfs}
 There is a bijection between acyclic total orientations of $G$ with unique source $q$ and maximal $(G,q)$-parking functions given by~$\O \mapsto (D_{\O})_{\mathbb{Z}V^{q}(G)}$.
 \end{lemma}

Note the unfortunate, but traditional, conflict between the terms sink and source. Here~$(\cdot)_{ \mathbb{Z}V^{q}(G)}$ means ignore the~$-1$~coefficient of~$q$ and treat the expression as an element of~$\mathbb{Z}V^{q}(G)$. The inverse map of the bijection in Lemma~\ref{lem:maxpfs} is essentially given by Dhar's burning algorithm~\cite{dhar1990self}. By~\cite[Theorem~7.3]{greene1983interpretation}, acyclic total orientations of $G$ with unique source $q$ (and consequently, maximal $(G,q)$-parking functions) are enumerated by~$T_G(1,0)$.

\begin{prop}[{See~\cite[Proposition 2,4]{hopkins2012bigraphical}}] \label{prop:numacyclicpos}
The number of indegree sequences of acyclic partial orientations of $G$ is $T_{G^{\bullet}}(1,1)$.
\end{prop}
\begin{proof} Observe that the set~$\{\overline{D}_{\O}\colon \O \textrm{ an acyclic total orientation of $G$} \}$ is also equal to~$\{(D_{\O})_{V(G)}\colon \O  \textrm{ an acyclic total orientation of $G^{\bullet}$ with unique source $v_0$}\}$ and so by Lemma~\ref{lem:maxpfs} the indegree sequences of acyclic total orientations of $G$ are the maximal $(G^{\bullet},v_0)$-parking functions. Then observe~$\{\overline{D}_{\O}\colon \O \textrm{ an acyclic partial orientation of $G$} \}$ is the same as~$\{c \in \mathbb{Z}V\colon 0 \leq c \leq \overline{D}_{\O} \textrm{ for some acyclic total orientation $\O$ of $G$}\}$ because any acyclic partial orientation can be completed to an acyclic total orientation. It is a simple fact that~$c \in \mathbb{Z}V^{q}(G)$ is a~$(G,q)$-parking function if and only if \mbox{$0 \leq c \leq c'$} for some maximal~$(G,q)$-parking function $c'$. Thus indeed the set of~$(G^{\bullet},v_0)$-parking functions is~$\{\overline{D}_{\O}\colon \O \textrm{ an acyclic partial orientation of $G$} \}$. It is a classical fact (again, see~\cite{benson2010g}) that the number of~$(G,q)$-parking functions is~$T_{G}(1,1)$, the number of spanning trees of~$G$. So the number of indegree sequences of acyclic partial orientations of~$G$ is~$T_{G^{\bullet}}(1,1)$.\end{proof}

\begin{prop}[{See~\cite[Corollary 2.10]{hopkins2012bigraphical}}]
The number of indegree sequences of cycle neutral partial orientations of $G$ is $T_{G^{\bullet}}(1,1)$.
\end{prop}
\begin{proof}
The main result of Hopkins and Perkinson~\cite[Corollary 2.10]{hopkins2012bigraphical} is that for any parameter list~$A \in \mathbb{R}_{>0}^{\mathbb{E}(G)}$ (in the sense of~\S\ref{subsec:bigraph}) the set of indegree sequences of acyclic partial orientations of $G$ is also equal to~$\{\overline{D}_{\O_R}\colon R \textrm{ a region of } \Sigma_{(G,\Oref)}(A)\}$.  (See also the work of Mazin~\cite{mazin2015multigraph} extending this result, which was originally proven only for simple graphs, to multigraphs; and for more on the connection between parking functions and partial orientations when $G = K_n$ is the complete graph, see~\cite{beck2014parking}.) Therefore by Propositions~\ref{prop:bigbij} and~\ref{prop:numacyclicpos} the number of indegree sequences of cycle neutral partial orientations of $G$ is also given by~$T_{G^{\bullet}}(1,1)$. 
\end{proof}

It would be interesting to see if we can count indegree sequences for other classes of partial orientations by evaluating the Tutte polynomial of graphs related to $G$, or by using more complicated expressions involving the Tutte polynomial of~$G$ itself. Another way to obtain Tutte polynomial enumerations of indegree sequences for min-edge classes of partial orientations is by restricting to special input data. Recall that a partial orientation~$\O$ is \emph{$q$-connected} if every vertex $v \in V^{q}(G)$ is reachable from $q$, i.e.,~there is a directed path from $q$ to $v$ for each~$v \in V^{q}(G)$. As mentioned in~\S\ref{subsec:history}, $q$-connected (or ``initially connected'') fourientations and partial orientations were previously investigated by Gessel and Sagan~\cite{gessel1996tutte} in the context of depth-first search. First we consider a slight variation of Proposition~\ref{prop:numacyclicpos}  (which appears implicity in ~\cite{mohammadi2013divisors2} and explicitly in ~\cite[Lemma 5.6]{backman2014riemann} and ~\cite[Theorem 3.10]{mohammadi2014divisors}).  The following result in some sense extends Lemma~\ref{lem:maxpfs} to all, not necessarily maximal, $G$-parking functions, and it follows more-or-less immediately from that classical result.  Perhaps the main reason why this result did not appear earlier in the literature is simply that many authors consider total orientations to be more well-behaved or natural than partial orientations.  One of the main goals of this paper is to convince the reader that this instinctual desire to restrict attention to total orientations is not always beneficial.

\begin{prop} \label{prop:numacyclicqpos}
The number of indegree sequences of acyclic, $q$-connected partial orientations of $G$ is $T_G(1,1)$.
\end{prop}
\begin{proof}
Note that the acyclic,~$q$-connected total orientations of $G$ are the same as the acyclic total orientations of $G$ with unique source~$q$. So thanks to Lemma~\ref{lem:maxpfs}, and arguing as in the proof of Proposition~\ref{prop:numacyclicpos}, the set of $(G,q)$-parking functions is also equal to~$\{(D_{\O})_{\mathbb{Z}V^{q}(G)} \colon \O \textrm{ an acyclic, $q$-connected partial orientation of $\G$}\}$. As before the cardinality of this set is $T_G(1,1)$.
\end{proof}

Recall that we are always assuming our graph $G$ is connected and so in particular $G$ has at least one spanning tree. Choose a sink~$q \in V(G)$ and choose an \emph{ordered, $q$-rooted spanning tree} $T$ of $G$. By this we mean that $T$ is a directed spanning tree of~$G$ rooted at $q$, with edges oriented away from~$q$ and totally ordered in some way consistent with the partial order of ancestry so that edges closer to $q$ in $T$ are less than those further away. Let us say that $\Oref$ and $<$ are \emph{compatible} with the data of $(q,T)$ if reference orientation~$\Oref$ is obtained by extending the orientation of the edges of~$T$ to all the edges in $E(G)$ arbitrarily, and the edge order~$<$ is obtained by extending the order on the edges of $T$ to an order of all the edges in $E(G)$ in some way so that any edge not in~$T$ is greater than all edges in~$T$. If $\Oref$ and $<$ are compatible with~$(q,T)$ then we call the ordered, oriented graph $\G = (G,\Oref,<)$ a \emph{$(q,T)$-connected graph}.  Let us say that $\G$ is \emph{$q$-connected} if the set of cut connected partial orientations of $\G$ is equal to the set of $q$-connected partial orientations of $\G$. The point of studying $(q,T)$-connected graphs $\G$ is explained by the following proposition.

\begin{prop}[{See~\cite[\S3]{backman2014partial}}] \label{prop:qconnected}
If $\G$ is $(q,T)$-connected then $\G$ is $q$-connected.
\end{prop}
\begin{proof}
Let $\G=(G,\Oref,<)$ be~$(q,T)$-connected and let $\O$ be a partial orientation of~$\G$. Suppose $\O$ is not cut connected and let $\OCu = (U,U^c)$ be a bad potential cut that witnesses this. Let~$e$ be the minimum edge in $E(\OCu)$. Since $\OCu$ is bad we have~$\Oref(e) = (v,u)$ with~$u \in U$ and~$v \in U^c$. Observe that~$e$ belongs to~$T$. Moreover, because $T$ is ordered in a way consistent with ancestry and because its edges are oriented away from $q$, all the edges in $T$ closer to $q$ than~$e$ are between vertices in~$U^c$. Thus it must be that~$q \in U^c$. But this means that~$u \in U$ is not reachable from $q$ and so $\O$ is not $q$-connected. Now suppose $\O$ is not $q$-connected. Let $U$ be set of vertices not reachable from $q$. Then $\OCu = (U,U^c)$ is a bad potential cut with respect to the cut connected property: the minimum edge $e$ in $E(\OCu)$ belongs to~$T$ and since $q \in U^c$ and $e$ is oriented away from $q$ we have~$\Oref(e) = (v,u)$ with~$u \in U$ and~$v \in U^c$.
\end{proof}

\begin{cor}
Suppose $\G$ is $(q,T)$-connected. Then the number of indegree sequences of acyclic-cut connected partial orientations of $\G$ is $T_G(1,1)$.
\end{cor}
\begin{proof}
This follows immediately from~Propoisitions~\ref{prop:numacyclicqpos} and~\ref{prop:qconnected}. 
\end{proof}

The~$(q,T)$-connectedness of $\G$ is really essential here: the number of indegree sequences of acyclic-cut connected partial orientations of an arbitrary ordered, oriented graph $\G = (G,\Oref,<)$ is not necessarily given by $T_G(1,1)$. In fact, we have the following converse to Proposition~\ref{prop:qconnected} which justifies restricting our attention to $(q,T)$-connected graphs.

\begin{prop}
Suppose $\G = (G,\Oref,<)$ is $q$-connected. Then there exists an ordered, $q$-rooted spanning tree $T$ such that for any min-edge cut property $X$ and any $(q,T)$-connected graph $\G' = (G,\Oref',<')$, the set of good fourientations of $\G$ with respect to $X$ is equal to the set of good fourientations of $\G'$ with respect to $X$.\footnote{For $\O$ a fourientation of $\G$ and $\O'$ a fourientation of $\G'$ we write $\O = \O'$ to mean that we have an equality of multisets $\{\Oref(e^{\pm})\colon e^{\pm} \in \O\} = \{\Oref'(e^{\pm})\colon e^{\pm} \in \O'\}$.}
\end{prop}
\begin{proof}
Let $\G = (G,\Oref,<)$ be $q$-connected. It is enough to prove the following:
\begin{equation*} \label{claim:tree}
 \parbox{5in}{There is an ordered, $q$-rooted spanning tree $T$ of $G$ such that for any simple cut $\Cu$ of $G$ the minimum edge $e$ in $E(\Cu)$ with respect to $<$ is the minimum edge in $E(\Cu) \cap T$ with respect to the tree order and $\Oref(e)$ agrees with the orientation of $e$ in $T$.} \tag{\dag}
\end{equation*}
This is because, as explained in the proof of Theorem~\ref{thm:minaretutte}, it suffices to check min-edge cut properties on simple cuts and because the minimum edge in any cut $\Cu$ with respect to $<'$ for a $(q,T)$-connected graph $\G = (G,<',\Oref')$ will always be the minimum edge in $E(\Cu) \cap T$. For~$U \subseteq V(G)$ let $G[U]$ denote the restriction of $G$ to vertex set $U$ and define $\G[U]$ analogously. A \emph{cut vertex} of~$G$ is~$v \in V(G)$ such that $G[V(G)\setminus\{v\}]$ is disconnected. Suppose~$v \in V(G)$ is a cut vertex of $G$. Let $V_0 \subseteq V(G)$ be the connected component of $G[V(G)\setminus\{v\}]$ containing~$q$; set~$V_1 := V_0 \cup \{v\}$ and $V_2 := V(G) \setminus V_0$. If~$\Cu$ is a simple cut of $G$ then either~$E(\Cu) \subseteq E(G[V_1])$ or~$E(\Cu) \subseteq E(G[V_2])$. So if $T_1$ is an ordered, $q$-rooted spanning tree of~$G[V_1]$ satisfying~(\ref{claim:tree}) for~$\G[V_1]$ and $T_2$ is a $v$-rooted, ordered spanning tree of~$G[V_2]$ satisfying~(\ref{claim:tree}) for~$\G[V_2]$ then $T := T_1 \cup T_2$ (where we declare all edges of~$T_2$ to be greater than those of~$T_1$) is an ordered, $q$-rooted spanning tree of $G$ satisfying~(\ref{claim:tree}) for~$\G$. Thus by induction on the number of vertices we can reduce to the case where $G$ has no cut vertices.

Now assume that $G$ has no cut vertices. To construct the appropriate $T$ we build up a chain of vertices $\{q\} = S_1 \subsetneq S_2 \subsetneq \cdots \subsetneq S_{n} = V(G)$ and $T_1,\ldots,T_n$ such that~$T_i$ is an ordered, $q$-rooted spanning tree of $G[S_i]$ for each $1 \leq i \leq n$ and~$T_{i+1}$ is obtained from~$T_i$ by adding $e_i$, the minimum edge (with respect to~$<$) in $E(G) \setminus E(G[S_i])$. To show that this is possible we need to show that the minimum edge in $E(G) \setminus E(G[S_i])$ never belongs to~$E(G[V(G)\setminus S_i])$. Suppose to the contrary that for some $i$ the minimum edge in $E(G) \setminus E(G[S_i])$ is $e_i = \{u,v\}$ with~$u,v \in V(G) \setminus S_i$ and~$e_i^{+} = (u,v)$. Since~$G$ has no cut vertices, $G[V(G) \setminus \{u\}]$ is connected and thus there exists a $q$-connected partial orientation $\O_u$ of~$G[V(G) \setminus \{u\}]$. Set $\O := \O_u \cup \{e^{\pm}\colon e^{\pm} = (w,u), e \in E(G)\}$. Then $\O$ remains $q$-connected; but $\O$~is not a cut connected partial orientation of~$\G$ because~$\Cu = (V(G) \setminus \{u\},\{u\})$ is a bad potential cut. This contradicts the assumption that $\G$ is $q$-connected. So indeed the minimum edge $e_i = \{u,v\}$ in $E(G) \setminus E(G[S_i])$ is always between a vertex $u  \in S_i$ and a vertex $v \in V(G)\setminus S_i$. A very similar argument shows that $\Oref(e_i) = (u,v)$. Thus we can construct the desired $T_i$. Now set $T := T_n$. We claim that this~$T$ satisfies~(\ref{claim:tree}). So let~$\Cu = \{U,U^c\}$ be a simple cut of $G$ and suppose~$e_i$ is the minimum edge in~$E(\Cu) \cap T$. Because $T_i$ spans $S_i$ and all the edges in~$T_i$ are less than~$e_i$, we must have that $S_i \subseteq U$. But then $E(U,U^c) \subseteq E(G) \setminus E(G[S_i])$, so~$e_i$ must be the minimum edge in $E(U,U^c)$. As mentioned above, $\Oref(e_i)$ agrees with the orientation of $e_i$ in $T$. Therefore~$T$ satisfies~(\ref{claim:tree}).
\end{proof}

\begin{remark}
There is a natural notion of $q$-connected fourientation as well: we say a fourientation $\O$ is \emph{$q$-connected} if for every $v \in V^{q}(G)$ there is a potential path from $q$ to $v$. Note that~$\G$ being $q$-connected is equivalent to the set of cut connected fourientations of $\G$ being equal to the set of $q$-connected fourientations of $\G$. This is because a fourientation is cut connected (respectively,~$q$-connected) if and only if it contains a cut connected (resp.,~$q$-connected) partial orientation. In fact, the minimal cut connected fourientations of $\G$ under the partial order of containment are ``oriented spanning trees'' of $G$; each spanning tree appears appears exactly once in this set with some orientation. In the case where $\G$ is $q$-connected, these minimal cut connected fourientations are precisely the $q$-rooted spanning trees.
\end{remark}

\begin{remark}
In contrast to the complicated situation with partial orientations described above, the number of indegree sequences among a min-edge class of total orientations is certainly given by a Tutte polynomial evaluation as outlined in~\S\ref{subsec:history}. In the other direction, it also might be interesting to investigate the number of indegree sequences among fourientations in a min-edge class. Again, this value is not necessarily a generalized Tutte polynomial evaluation. However, we can nevertheless sometimes get a simple expression for this value: for instance, it is easily seen that the number of indegree sequences among all fourientations of a graph $G$ is $\prod_{v \in V(G)} \mathrm{deg}(v)$.
\end{remark}

\begin{remark}
For ~$\G$ a $(q,T)$-connected graph, the cut minimal-cycle minimal total orientations enumerated by $T_G(1,1)$ become the cycle minimal,~$q$-connected total orientations.  These objects are in bijection with their associated divisors which were introduced by Gioan~\cite{gioan2007enumerating} and further investigated by Bernardi~\cite{bernardi2008tutte}.  These divisors were rediscovered by An, Baker, Kuperberg, and Shokrieh~\cite{an2014canonical}, who proved that they are exactly the \emph{break divisors} of Mikhalkin and Zharkov \cite{mikhalkin2006tropical} offset by a chip at $q$.  These break divisors were discovered originally in the context of divisor theory for tropical curves: they provide canonical representatives for the set of divisors of degree $g$ modulo linear equivalence.  In particular, this implies that by adding a chip at $q$ to the divisors associated to $q$-connected partial orientations, we lose all dependence on $q$.  Interestingly, there exist tropical proofs of the existence and uniqueness of break divisors which are not combinatorial in nature.
\end{remark}

\subsection{Monomizations of power ideals and cut internal partial orientations} \label{subsec:cutin}

One can extend the enumeration of $(G,q)$-parking functions via the Tutte polynomial to an expression for the generating function of $(G,q)$-parking functions by degree. For~$c = \sum_{v \in V^{q}(G)} c_v (v) \in \mathbb{Z}V^{q}(G)$ define~$\mathrm{deg}(c) := \sum_{v \in V^{q}(G)} c_v$. A famous result of Merino~\cite{merino1997chip} is that~$T_G(1,y) = \sum_{c \in \mathrm{PF}(G,q)} y^{g - \mathrm{deg}(c)}$. (Merino~\cite{merino2001chip} used this interpretation of the Tutte polynomial to resolve a special case of a 1977 conjecture of Stanley~\cite{stanley1977cohen} about the $h$-vectors of matroid complexes.) Merino's theorem can also be expressed succinctly using commutative algebra. Fix some field $\mathbf{k}$ and let $R := \mathbf{k}[x_v\colon v \in V^{q}(G)]$ be the polynomial ring with generators indexed by non-sink vertices. For $U \subseteq V^{q}(G)$, define~$\mathbf{x}^{U} := \prod_{u \in U} x_u^{d^G_U(u)}$ where as before we have~$d^G_U(u) := |\{e= \{u,v\} \in E(U,U^c)\}|$. Then define the monomial ideal~$I_{(G,q)} := \langle \mathbf{x}^{U}\colon U\subseteq V^{q}(G) \textrm{ with $U \neq \emptyset$} \rangle$. We use the notation~$\mathbf{x}^c := \prod_{v \in V^{q}(G)} x_v^{c_v}$ for~$c = \sum_{v \in V^{q}(G)} c_v (v) \in \mathbb{N}V^{q}(G)$. It is not difficult to see that a linear basis of~$R/I_{(G,q)}$ is~$\{\mathbf{x}^c\colon c \textrm{ a $(G,q)$-parking function}\}$. A restatement of Merino's theorem is then that the Hilbert series of the~$R$-module $R/I_{(G,q)}$ is~$\mathrm{Hilb}(R/I_{(G,q)};y) = y^{g} \cdot T_G(1,1/y)$.

Motivated by questions in Schubert calculus~\cite{postnikov1999differential}, Postnikov and Shapiro~\cite{postnikov2004trees} studied the monomial ideal $I_{(G,q)}$ as well as a deformation of this ideal generated by powers of homogenous linear forms. Specifically, setting $d^G_U := \sum_{u \in U} d^G_U(u)$, they defined the power ideal $J_{(G,q)} := \langle \left( \sum_{u \in U} x_u\right)^{d^G_U}\colon U\subseteq V^{q}(G) \textrm{ with $U \neq \emptyset$} \rangle$. One of the main results of~\cite{postnikov2004trees} is that $R/I_{(G,q)}$ and $R/J_{(G,q)}$ have the same Hilbert series. The (Macaulay) inverse system of~$J_{(G,q)}$ is a so-called \emph{central zonotopal algebra}~\cite{holtz2011zonotopal}. There are two other closely related zonotopal algebras associated to a graph. Specifically, define the ideals~$J^{r}_{(G,q)} := \langle ( \sum_{u \in U} x_u)^{d^G_U+r}\colon U\subseteq V^{q}(G) \textrm{ with $U \neq \emptyset$} \rangle$ for~$r \geq -1$. Then the inverse systems of $J^{r}_{(G,q)}$ in the special cases $r = +1,0,-1$ yield the \emph{external}, \emph{central} and \emph{internal zonotopal algebras} associated to $G$. These spaces of polynomials are related to the complexity of box splines~\cite{dahmen1985local}~\cite{deconcini2008hyperplane}. Recall that (at least when $\mathbf{k}$ has characteristic zero) the dimension of an inverse system of an ideal is equal to the dimension of the quotient by that ideal and that when the ideal is homogenous, as is the case for these power ideals, the same is true of the graded parts. Thus the Hilbert series of~$R/J^{r}_{(G,q)}$ are important in the theory of zonotopal algebras. Ardila and Postnikov~\cite{ardila2010combinatorics} show that 
 \begin{align*}
 \mathrm{Hilb}(R/J^{+1}_{(G,q)};y) &= y^g\cdot T_G(1+y,1/y); \\
 \mathrm{Hilb}(R/J^{-1}_{(G,q)};y) &= y^g\cdot T_G(0,1/y).
 \end{align*}
One might wonder whether there are some analogous monomial ideals~$I^{r}_{(G,q)}$ for~$r = \pm1$ with $\mathrm{Hilb}(R/J^{r}_{(G,q)};y) = \mathrm{Hilb}(R/I^{r}_{(G,q)};y)$; in this case we say the monomial ideal is a \emph{monomization} of the corresponding power ideal.\footnote{More precisely, for $\mathcal{I}$ a monomial ideal of $R$ and $\mathcal{J}$ any ideal of $R$, we say that $\mathcal{I}$ is a \emph{monomization} of $\mathcal{J}$ if the standard monomials of $\mathcal{I}$ give a linear basis of $R/\mathcal{J}$. We will however ignore the distinction between having the same Hilbert series and having the same linear basis. Note that for any given $\mathcal{J}$ such an $\mathcal{I}$ can in principle be found using the theory of Gr\"{o}bner bases but when $\mathcal{J}$ is a power ideal it is computationally expensive to compute a Gr\"{o}bner basis of $\mathcal{J}$. The monomizations we discuss here are not initial ideals of their corresponding power ideals with respect to any term order.} We want this monomization to be ``natural'' in some sense: for instance, the generators of $I_{(G,q)}$ correspond to the generators of $J^{0}_{(G,q)}$ in an obvious way; we would want the~$I^{r}_{(G,q)}$ to also have this property with respect to the~$J^{r}_{(G,q)}$.

For the complete graph $G = K_{n}$, Postinkov-Shapiro-Shaprio~\cite{postnikov1998chern} found  such an external monomial ideal $I^{+1}_{(G,q)}$ (and indeed this external case was the one they were originally interested in). Desjardins~\cite[\S3]{desjardins2010monomization} extended their construction to obtain an external monomial ideal $I^{+1}_{(G,q)}$ for any $G$. Specifically, let $\prec$ be a total order on~$V^{q}(G)$ and define~$I^{+1}_{(G,q,\prec)} :=  \langle x_{\mathrm{min}_{\prec}(U)} \cdot \mathbf{x}^{U}\colon U\subseteq V^{q}(G) \textrm{ with $U \neq \emptyset$} \rangle$ where $\mathrm{min}_{\prec}(U)$ is the minimal element of~$U$ according to~$\prec$. Then Desjardins showed~$I^{+1}_{(G,q,\prec)}$ is a monomization of~$J^{+1}_{(G,q)}$. But he also showed that certain assumptions on $G$ were necessary to mimic this construction and obtain an appropriate internal monomial ideal~$I^{-1}_{(G,q)}$. We can try to define~$I^{-1}_{(G,q,\prec)} :=  \langle x_{\mathrm{min}_{\prec}(U)}^{-1} \cdot \mathbf{x}^{U}\colon U\subseteq V^{q}(G) \textrm{ with $U \neq \emptyset$} \rangle$. However, it is important to observe that $I^{-1}_{(G,q,\prec)}$ does not always make sense because $x_{\mathrm{min}_{\prec}(U)}^{-1} \cdot \mathbf{x}^{U}$ may be a Laurent monomial rather than an honest monomial. But $I^{-1}_{(G,q,\prec)}$ does make sense at least when there is an edge between $q$ and each vertex in $V^{q}(G)$. Desjardins~\cite[\S4]{desjardins2010monomization} showed that when $G$ is \emph{saturated},~i.e.,~when there is at least one edge between any pair of vertices in~$V(G)$, then $I^{-1}_{(G,q,\prec)}$ is a monomization of~$J^{-1}_{(G,q)}$ for any choice of~$\prec$.

Here we present an approach to the problem of finding a monomization of $J^{-1}_{(G,q)}$ for all $G$ using acyclic-cut internal partial orientations. As before, we must work with a~$(q,T)$-connected graph~$\G$.  For a subset~$U \subseteq V^{q}(G)$, define~$\mathbf{x}^{U,T} := \prod_{u \in U} x_u^{d^{G}_{U,T}(u)}$ with~$d^{G}_{U,T}(u) := |\{ e=\{u,v\} \in E(U,U^c) \textrm{ and $e$ is not the minimum edge in $E(U,U^c)$}\}|$. So in particular the degree of $\mathbf{x}^{U,T}$ is one less than the degree of $\mathbf{x}^{U}$. Then define the monomial ideal $I^{-1}_{(G,q,T)} := \langle \mathbf{x}^{U,T}\colon U\subseteq V^{q}(G) \textrm{ with $U \neq \emptyset$} \rangle$.  Note that an ordered, $q$-rooted spanning tree $T$ of $G$ gives rise to a total order $\prec_T$ on $V^{q}(G)$ whereby~$u \prec_T v$ if the minimum edge in $T$ containing $u$ is less than the minimum edge in $T$ containing~$v$. Furthermore, we have $I^{-1}_{(G,q,T)} = I^{-1}_{(G,q,\prec_T)}$; and moreover, the total orders $\prec$ arising from spanning trees in this way are precisely the ones for which $I^{-1}_{(G,q,\prec)}$ does make sense. In analogy to the definition of $(G,q)$-parking functions, let us say $c \in \mathbb{Z}V^{q}(G)$ is a \emph{$(G,q,T)$-subparking function} if~$\mathbf{x}^{c}$ is a standard monomial of $I^{-1}_{(G,q,T)}$ (i.e., a monomial in $R \setminus I^{-1}_{(G,q,T)}$). Equivalently:
\begin{definition}
An element $c = \sum_{v \in V^{q}(G)}c_v(v) \in \mathbb{Z}V^{q}(G)$ is a \emph{$(G,q,T)$-subparking function} if for every non-empty~$U \subseteq V^{q}(G)$, there is~$u \in U$ with~$0\leq c_u < d^{G}_{U,T}(u)$. We denote the set of~$(G,q,T)$-subparking functions by $\mathrm{PF}^{-}(G,q,T)$. (See also~\cite[Definition~1.5]{holtz2011zonotopal}, where Holtz and Ron define subparking functions, which they call ``internal parking functions,'' in the case of the complete graph $G=K_{n}$.)
\end{definition}
\noindent Also, let~$\mathrm{ACI}(G,q,T)$ denote the set of acyclic-cut internal partial orientations of~$\G$. (Notice that both~$\mathrm{PF}^{-}(G,q,T)$ and~$\mathrm{ACI}(G,q,T)$ depend only on $(G,q,T)$.)

\begin{conj} \label{conj:subpark}
For any graph $G$ and choice of sink $q \in V(G)$, there exists an ordered, $q$-rooted spanning tree $T$ of $G$ such that 
\begin{enumerate}[label=(\alph*)]
\item $\mathrm{Hilb}(R/I^{-1}_{(G,q,T)};y) = y^{g} \cdot T_G(0,1/y)$; \label{statement:hilbertseries}
\item $\mathrm{PF}^{-}(G,q,T) = \{(D_{\O})_{\mathbb{Z}V^{q}(G)}\colon \O \in \mathrm{ACI}(G,q,T) \}$. \label{statement:acfpos}
\end{enumerate}
\end{conj}

\begin{example}
Let $(G,q,T)$ be as below, with the edges of $T$ oriented, in bold, and labeled according to order:
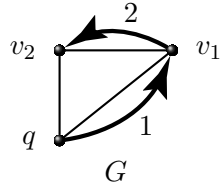
\begin{center}
\begin{tikzpicture}
	\SetFancyGraph
	\Vertex[LabelOut,Lpos=180, Ldist=.1cm,x=-0.75,y=-1.2]{q}
	\Vertex[LabelOut,Lpos=0, Ldist=.1cm,x=0.75,y=0]{v_1}
	\Vertex[LabelOut,Lpos=180, Ldist=.1cm,x=-0.75,y=0]{v_2}
	\Edges[style={thick}](q,v_1)
	\Edges[style={->,>=mytip,ultra thick,bend right}](q,v_1)
	\Edges[style={thick}](q,v_2)
	\Edges[style={thick}](v_1,v_2)
	\Edges[style={->,>=mytip,ultra thick, bend right}](v_1,v_2)
	\node at (0.2,0.5) {$2$};
	\node at (0.4,-1) {$1$};
	\node at (0,-1.6) {$G$};
\end{tikzpicture}
\end{center}
The $9$ acyclic-cut internal partial orientations in $\mathrm{ACI}(G,q,T)$ are the following:
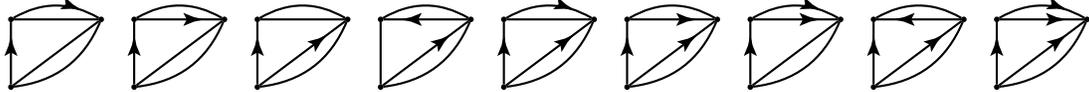
\begin{center}
\begin{tikzpicture}
	\SetBasicGraph
	\Vertex[NoLabel,x=-0.6,y=-0.9]{q}
	\Vertex[NoLabel,x=0.6,y=0]{v_1}
	\Vertex[NoLabel,x=-0.6,y=0]{v_2}
	\Edges[style={thick}](q,v_1)
	\Edges[style={thick,bend right}](q,v_1)
	\Edges[style={->-,>=mytip,thick}](q,v_2)
	\Edges[style={thick}](v_1,v_2)
	\Edges[style={->-,>=mytip,thick, bend left}](v_2,v_1)
\end{tikzpicture} \;
\begin{tikzpicture}
	\SetBasicGraph
	\Vertex[NoLabel,x=-0.6,y=-0.9]{q}
	\Vertex[NoLabel,x=0.6,y=0]{v_1}
	\Vertex[NoLabel,x=-0.6,y=0]{v_2}
	\Edges[style={thick}](q,v_1)
	\Edges[style={thick,bend right}](q,v_1)
	\Edges[style={->-,>=mytip,thick}](q,v_2)
	\Edges[style={->-,>=mytip,thick}](v_2,v_1)
	\Edges[style={thick, bend left}](v_2,v_1)
\end{tikzpicture} \;
\begin{tikzpicture}
	\SetBasicGraph
	\Vertex[NoLabel,x=-0.6,y=-0.9]{q}
	\Vertex[NoLabel,x=0.6,y=0]{v_1}
	\Vertex[NoLabel,x=-0.6,y=0]{v_2}
	\Edges[style={->-,>=mytip,thick}](q,v_1)
	\Edges[style={thick,bend right}](q,v_1)
	\Edges[style={->-,>=mytip,thick}](q,v_2)
	\Edges[style={thick}](v_1,v_2)
	\Edges[style={thick, bend right}](v_1,v_2)
\end{tikzpicture} \;
\begin{tikzpicture}
	\SetBasicGraph
	\Vertex[NoLabel,x=-0.6,y=-0.9]{q}
	\Vertex[NoLabel,x=0.6,y=0]{v_1}
	\Vertex[NoLabel,x=-0.6,y=0]{v_2}
	\Edges[style={->-,>=mytip,thick}](q,v_1)
	\Edges[style={bend right}](q,v_1)
	\Edges[style={thick}](q,v_2)
	\Edges[style={->-,>=mytip,thick}](v_1,v_2)
	\Edges[style={bend left}](v_2,v_1)
\end{tikzpicture} \;
\begin{tikzpicture}
	\SetBasicGraph
	\Vertex[NoLabel,x=-0.6,y=-0.9]{q}
	\Vertex[NoLabel,x=0.6,y=0]{v_1}
	\Vertex[NoLabel,x=-0.6,y=0]{v_2}
	\Edges[style={->-,>=mytip,thick}](q,v_1)
	\Edges[style={thick,bend right}](q,v_1)
	\Edges[style={->-,>=mytip,thick}](q,v_2)
	\Edges[style={thick}](v_1,v_2)
	\Edges[style={->-,>=mytip,thick, bend left}](v_2,v_1)
\end{tikzpicture} \;
\begin{tikzpicture}
	\SetBasicGraph
	\Vertex[NoLabel,x=-0.6,y=-0.9]{q}
	\Vertex[NoLabel,x=0.6,y=0]{v_1}
	\Vertex[NoLabel,x=-0.6,y=0]{v_2}
	\Edges[style={->-,>=mytip,thick}](q,v_1)
	\Edges[style={thick,bend right}](q,v_1)
	\Edges[style={->-,>=mytip,thick}](q,v_2)
	\Edges[style={->-,>=mytip,thick}](v_2,v_1)
	\Edges[style={thick, bend left}](v_2,v_1)
\end{tikzpicture} \;
\begin{tikzpicture}
	\SetBasicGraph
	\Vertex[NoLabel,x=-0.6,y=-0.9]{q}
	\Vertex[NoLabel,x=0.6,y=0]{v_1}
	\Vertex[NoLabel,x=-0.6,y=0]{v_2}
	\Edges[style={thick}](q,v_1)
	\Edges[style={thick,bend right}](q,v_1)
	\Edges[style={->-,>=mytip,thick}](q,v_2)
	\Edges[style={->-,>=mytip,thick}](v_2,v_1)
	\Edges[style={->-,>=mytip,thick, bend left}](v_2,v_1)
\end{tikzpicture} \;
\begin{tikzpicture}
	\SetBasicGraph
	\Vertex[NoLabel,x=-0.6,y=-0.9]{q}
	\Vertex[NoLabel,x=0.6,y=0]{v_1}
	\Vertex[NoLabel,x=-0.6,y=0]{v_2}
	\Edges[style={->-,>=mytip,thick}](q,v_1)
	\Edges[style={thick,bend right}](q,v_1)
	\Edges[style={->-,>=mytip,thick}](q,v_2)
	\Edges[style={->-,>=mytip,thick}](v_1,v_2)
	\Edges[style={thick, bend right}](v_1,v_2)
\end{tikzpicture} \;
\begin{tikzpicture}
	\SetBasicGraph
	\Vertex[NoLabel,x=-0.6,y=-0.9]{q}
	\Vertex[NoLabel,x=0.6,y=0]{v_1}
	\Vertex[NoLabel,x=-0.6,y=0]{v_2}
	\Edges[style={->-,>=mytip,thick}](q,v_1)
	\Edges[style={thick,bend right}](q,v_1)
	\Edges[style={->-,>=mytip,thick}](q,v_2)
	\Edges[style={->-,>=mytip,thick}](v_2,v_1)
	\Edges[style={->-,>=mytip,thick, bend left}](v_2,v_1)
\end{tikzpicture}
\end{center}
The Tutte polynomial of our graph is $T_G(x,y) = y^3 + x^2 + 2xy + 2y^2 + x + y$. We can compute~$I^{-1}_{(G,q,T)} = \langle x_{v_1}x_{v_2},x_{v_2}^2,x_{v_1}^3\rangle$. So indeed, 
\[\mathrm{Hilb}(R/ I^{-1}_{(G,q,T)};y) = 1 + 2y + y^2 = y^3 \cdot T_G(0,1/y).\] 
And indeed, the set of divisors associated to partial orientations in $\mathrm{ACI}(G,q,T)$ is the set of subparking functions, namely $\{0,(v_1),(v_2),2(v_1)\}$.
\end{example}

Finding the appropriate tree $T$ for each choice of graph $G$ and sink $q \in V(G)$ is a major part of resolving Conjecture~\ref{conj:subpark}. Desjardins~\cite[Example 21]{desjardins2010monomization} gave an example that shows statement~\ref{statement:hilbertseries} from Conjecture~\ref{conj:subpark} does not always hold for all choices of~$(G,q,T)$. That some restrictions on $T$ are necessary should be seen as similar to the fact that we need $\G$ to be $(q,T)$-connected in order to ensure that the number of indegree sequences of its acyclic-cut connected partial orientations is~$T_{G}(1,1)$.

We now prove some special cases of Conjecture~\ref{conj:subpark}. First we upgrade Desjardins' result about monomizations of the internal power ideal for saturated $G$ to a proof of Conjecture~\ref{conj:subpark} for saturated $G$.

\begin{thm}
Conjecture~\ref{conj:subpark} is true when $G$ is saturated.
\end{thm}
\begin{proof}
As we already explained, Desjardins showed that statement~\ref{statement:hilbertseries} holds for any choice of $T$. Thus we need only show that~\ref{statement:acfpos} holds for an appropriate choice of $T$, which we do now. We will take $T$ to be a star; i.e., the edges of $T$ are $e=\{q,v\}$ for all vertices~$v \in V^{q}(G)$. The order of the edges of $T$ can be arbitrary. Let~$<$ be some edge order compatible with~$T$. In what follows for convenience we write~$D_\O$ in place of~$(D_{\O})_{\mathbb{Z}V^{q}(G)}$ when the choice of sink is clear from context.

First let us show $\{D_{\O}\colon \O \in \mathrm{ACI}(G,q,T) \} \subseteq \mathrm{PF}^{-}(G,q,T)$. So let~$\O \in \mathrm{ACI}(G,q,T)$ and set~$c = \sum_{v \in V^{q}(G)}c_v(v) := D_{\O}$. Suppose to the contrary that $c \notin \mathrm{PF}^{-}(G,q,T)$: specifically, suppose there is some~$U \subseteq V^{q}(G)$ such that~$d^{G}_{U,T}(u) \leq c_u$ for all $u \in U$. Let $u_0$ be the vertex in $U$ adjacent to the minimal edge in $E(U,U^c)$. First suppose that there is some $e \in E(U,U^c)$ with $e^{\pm} = (v,u_0)$ such that $e^{\pm} \notin \O$. This means we have $|\{e^{\pm}=(u',u)\in\O\colon u' \in U, e\in E(G)\}| \geq 1$ for each~$u \in U$. But then~$\O$ contains a directed cycle involving vertices in $U$, contradicting the fact that $\O$ is acyclic. So now assume that for all $e \in E(U,U^c)$ with $e^{\pm} = (v,u_0)$ we have $e^{\pm} \in \O$. Then if~$\OCu = (U^c,U)$ is a potential cut, it is a bad potential cut with respect to the cut internal property because its minimum edge is oriented. So it cannot be a potential cut. Thus there must be~$e^{\pm} = (u_1,v_0) \in \O$ with~$u_1 \in U$ and~$v_0 \in U^c$. But also note that~$|\{e^{\pm}=(u',u)\in\O\colon u' \in U, e\in E(G)\}| \geq 1$ for each~$u \in U$ with~$u \neq u_0$, which in particular means there is a directed path in~$\O$ from~$u_0$ to~$u_1$ involving vertices in~$U$. Because $G$ is saturated, there is~$e = \{v_1,u_0\} \in E(G)$ and by assumption we have~$e^{\pm} = (v_1,u_0) \in \O$. Therefore there is a directed cycle in $\O$ that goes from $u_1$ to~$v_1$ to $u_0$ back to $u_1$, contradicting the fact that $\O$ is acyclic. So indeed~$c \in \mathrm{PF}^{-}(G,q,T)$.

Next let us show $\mathrm{PF}^{-}(G,q,T) \subseteq  \{D_{\O}\colon \O \in \mathrm{ACI}(G,q,T) \}$. So let~$c\in \mathrm{PF}^{-}(G,q,T)$. We want to find a $\O \in  \mathrm{ACI}(G,q,T)$ with $D_{\O} = c$. To do this we will apply the Cori-Le~Borgne variant~\cite{cori2003sand} of Dhar's burning algorithm~\cite{dhar1990self}. The algorithm proceeds as follows: we initialize $\O_0 := \emptyset$ and $B_0 := \{q\}$; at the $i$th step for~\mbox{$i=1,\ldots,\mathrm{deg}(c)+n-1$} we set~$\O_{i} := \O_{i-1} \cup \{e^{\pm}\}$, where $e^{\pm}$ is the maximum edge of~$\mathbb{E}(B_{i-1},B_{i-1}^c) \setminus \O_i$ according to~$<$ , and~$B_i := \{q\} \cup \{v\in V^{q}(G)\colon \mathrm{indeg}_{\O_i}(v) - 1 = c_v\}$; our output is~\mbox{$\O := \O_{\mathrm{deg}(c)+n-1}$}. The facts that $\mathbb{E}(B_{i-1},B_{i-1}^c) \setminus \O_i \neq \emptyset$ at each step, that~$D_{\O} = c$, and that $\O$ is a $q$-connected, acyclic partial orientation follow from the correctness of the Cori-Le Borgne algorithm~\cite{cori2003sand} (see also the description of this algorithm given in~\mbox{\cite[\S5.2]{baker2013chip}}). All we need to show is that $\O$ is cut neutral. Suppose to the contrary we have a potential cut~$\OCu = (U,U^c)$ of $\O$ and $e_{\mathrm{min}}^{\delta} \in \O$ where~$e_{\mathrm{min}}^{\delta} = (q,v)$ is the minimum edge in~$\mathbb{E}(U,U^c)$ with respect to~$<$. (We know $e_{\mathrm{min}}$ is of this form because~$T$ is a star.) Let~$i$ be such that~$\O_i = \O_{i-1} \cup \{e^{\delta}\}$. First suppose that~$U \subseteq B_i$. The minimum edge in any cut is of the form $e=\{q,v'\}$ because $T$ is a star, and since~$\{e=\{q,v'\}\colon v' \in B_i^c\}\subseteq \{e=\{q,v'\}\colon v' \in U^c\}$, the minimum edge in~$E(B_i,B_i^c)$ must in fact be~$e_{\mathrm{min}}$. But because we always choose the maximum edge in the cut to add at every step of the algorithm, we must have~$\mathbb{E}(B_i,B_i^c) \subseteq \O$. This in turn means that~$d^{G}_{B_i^c,T}(v') \leq \mathrm{indeg}(\O)-1$ for all~$v' \in B_i^c$, contradicting that~$c \in \mathrm{PF}^{-}(G,q,T)$. So now assume~$U \setminus B_i \neq \emptyset$.  Let $j$ be minimal so that $\O_j = \O_{j-1}\cup \{e_{*}^{\delta_*}\}$ with~$e_{*}^{\delta_*} = (u,w)$ for some $w \in U \setminus B_i$. Such a $j$ exists because $B_{\mathrm{deg}(c)+n-1} = V(G)$. Note that because we added an edge $e_{\mathrm{min}} \in T$ at step $i$, it must be that any $e \in \mathbb{E}(B_i,B_i^{c}) \setminus \O_i$ also belongs to $T$ (and is thus incident to $q$). So if $u \in B_i$ then $u = q$  and $e_{*} \in T$. But in fact because $G$ is saturated there is an edge $e=\{v,w\} \in E(G)$, and because the algorithm would not choose to add an edge in~$T$ when it could add an edge not in~$T$ this means~$u \neq q$. Thus $u \in U^c$, which means $e^{\delta_*}_{*} = (u,w) \in \O$ with $u \in U^c$ and $w \in U$, contradicting the fact that $\OCu$ is a potential cut of~$\O$. So indeed $\O$ is cut neutral.
\end{proof}

We will now prove Conjecture~\ref{conj:subpark} in a different special case than the case addressed by Desjardins, namely, when $G$ is an \emph{outerplanar} graph. This means that~$G$ can be drawn in the plane without crossings in such a way that all of its vertices lie in the boundary of the unbounded face of this drawing. These cases really are quite different:~$G$ being saturated means that~$G$ is ``dense'' while~$G$ being outerplanar means that $G$ is ``sparse.'' Also, the techniques we employ are very different from those used by Desjardins. Desjardins employed the theory of Monotone Monomial Ideals developed by Postinkov-Shapiro~\cite[\S5]{postnikov2004trees}. Instead, we build on Merino's deletion-contraction proof~\cite{merino1997chip} of his famous theorem. We remark that the ideals $I_{(G,q,T)}$ we obtain for outerplanar $G$ are in general not Monotone Monomial Ideals.

The main observation for what follows is that if $T$ is an ordered, $q$-rooted spanning tree of any graph~$G$ and~$e = \{q,v\} \in E(G)$ is some non-loop edge that does not belong to~$T$ then there are natural ways to obtain ordered, $q$-rooted spanning trees~$T /e$ and~$T \setminus e$ of~$G/e$ and $G\setminus e$, respectively. These are defined as follows. Let $f$ be the minimal edge in~$T$ that contains $v$; then $T/e$ is the the ordered, $q$-rooted spanning tree of $G/e$ consisting of all edges in~$T$ except for $f$ with the same relative order as in~$T$. And~$T\setminus e$ is just defined to be equal to~$T$. Figure~\ref{fig:treecondel} gives an example of this construction, with the edges of the $q$-rooted trees oriented, in bold, and labeled according to order.
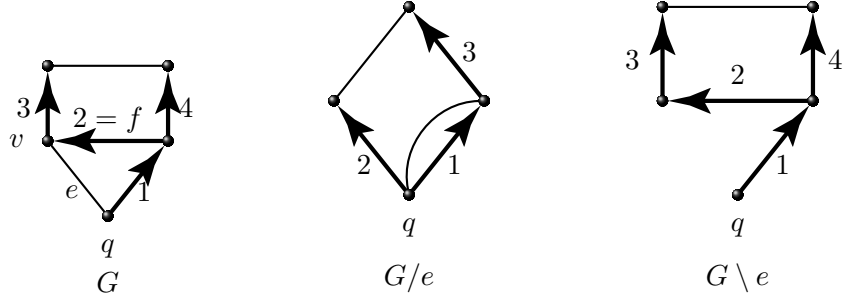
\begin{figure}
\begin{tikzpicture}[scale=0.8]
	\SetFancyGraph
	\Vertex[LabelOut,Lpos=270, Ldist=.1cm,x=0,y=0]{q}
	\Vertex[NoLabel,x=1,y=1.25]{v_1}
	\Vertex[NoLabel,x=1,y=2.5]{v_2}
	\Vertex[NoLabel,x=-1,y=2.5]{v_3}
	\Vertex[LabelOut,Lpos=180,Ldist=.1cm,,x=-1,y=1.25]{v}
	\Edges[style={->,>=mytip,ultra thick}](q,v_1)
	\Edges[style={->,>=mytip,ultra thick}](v_1,v_2)
	\Edges[style={->,>=mytip,ultra thick}](v_1,v)
	\Edges[style={->,>=mytip,ultra thick}](v,v_3)
	\Edges[style={thick}](q,v)
	\Edges[style={thick}](v_2,v_3)
	\node at (-0.6,0.4) {$e$};
	\node at (0.6,0.4) {$1$};
	\node at (0,1.6) {$2=f$};
	\node at (-1.4,1.8) {$3$};
	\node at (1.3,1.8) {$4$};
	\node at (0,-1.1) {$G$};
\end{tikzpicture} \qquad \qquad \begin{tikzpicture}
	\SetFancyGraph
	\Vertex[LabelOut,Lpos=270, Ldist=.1cm,x=0,y=0]{q}
	\Vertex[NoLabel,x=1,y=1.25]{v_1}
	\Vertex[NoLabel,x=0,y=2.5]{v_2}
	\Vertex[NoLabel,x=-1,y=1.25]{v_3}
	\Edges[style={->,>=mytip,ultra thick}](q,v_1)
	\Edges[style={->,>=mytip,ultra thick}](v_1,v_2)
	\Edges[style={thick,bend left=50}](q,v_1)
	\Edges[style={->,>=mytip,ultra thick}](q,v_3)
	\Edges[style={thick}](v_2,v_3)
	\node at (0.6,0.4) {$1$};
	\node at (-0.6,0.4) {$2$};
	\node at (0.8,1.9) {$3$};
	\node at (0,-1.1) {$G/e$};
\end{tikzpicture} \qquad \qquad \begin{tikzpicture}
	\SetFancyGraph
	\Vertex[LabelOut,Lpos=270, Ldist=.1cm,x=0,y=0]{q}
	\Vertex[NoLabel,x=1,y=1.25]{v_1}
	\Vertex[NoLabel,x=1,y=2.5]{v_2}
	\Vertex[NoLabel,x=-1,y=2.5]{v_3}
	\Vertex[NoLabel,x=-1,y=1.25]{v_4}
	\Edges[style={->,>=mytip,ultra thick}](q,v_1)
	\Edges[style={->,>=mytip,ultra thick}](v_1,v_2)
	\Edges[style={->,>=mytip,ultra thick}](v_1,v_4)
	\Edges[style={->,>=mytip,ultra thick}](v_4,v_3)
	\Edges[style={thick}](v_2,v_3)
	\node at (0.6,0.4) {$1$};
	\node at (0,1.6) {$2$};
	\node at (-1.4,1.8) {$3$};
	\node at (1.3,1.8) {$4$};
	\node at (0,-1.1) {$G\setminus e$};
\end{tikzpicture}
\caption{An example of contracting and deleting an ordered, $q$-rooted spanning tree $T$ along an edge $e \notin T$.} \label{fig:treecondel}
\end{figure}

\begin{prop} \label{prop:subparkisthmus}
Let $G$ be a graph, $q \in V(G)$ a sink, and $T$ an ordered, $q$-rooted spanning tree of $G$. Let $e \in E(G)$ be an isthmus. Then $\mathrm{Hilb}(R/I^{-1}_{(G,q,T)};y) = 0$.
\end{prop}
\begin{proof}
Let $\{U,U^c\}$ be the cut such that $E(U,U^c) = \{e\}$ and $q \in U^c$. Then $\mathbf{x}^{U,T} = 1$, so $R/I^{-1}_{(G,q,T)} = 0$.
\end{proof}

\begin{prop} \label{prop:subparkloop}
Let $G$ be a graph, $q \in V(G)$ a sink, and $T$ an ordered, $q$-rooted spanning tree of $G$. Let $e \in E(G)$ be a loop. Then $\mathrm{Hilb}(R/I^{-1}_{(G,q,T)};y) = \mathrm{Hilb}(R/I^{-1}_{(G/e,q,T)};y)$.
\end{prop}
\begin{proof}
This is trivial: loops do not affect $d^{G}_{U,T}(u) $ for any $u \in U \subseteq V^{q}(G)$, so in fact we have $I^{-1}_{(G,q,T)} = I^{-1}_{(G/e,q,T)}$.
\end{proof}

\begin{lemma} \label{lem:subparkdc}
Let $G$ be a graph, $q \in V(G)$ a sink, and $T$ an ordered, $q$-rooted spanning tree of $G$. Let $e = \{q,v\} \in E(G)$ be some non-loop edge that does not belong to $T$. Then $\mathrm{Hilb}(R/I^{-1}_{(G,q,T)};y) = \mathrm{Hilb}(R/(I^{-1}_{(G/e,q,T/e)}+ \langle x_v \rangle);y)+y\cdot \mathrm{Hilb}(R/I^{-1}_{(G\setminus e,q,T \setminus e)};y)$.
\end{lemma}
\begin{proof}
Let $F_1 \subseteq \mathrm{PF}^{-}(G,q,T)$ be the subset of subparking functions whose coefficient of $(v)$ is $0$. Let $F_2 := \mathrm{PF}^{-}(G,q,T) \setminus F_1$. Following Merino~\cite[Theorem 3.6]{merino1997chip}, we will construct bijections
\begin{align*}
\varphi_1\colon F_1 &\to \mathrm{PF}^{-}(G/e,q,T/e) \\
\varphi_2\colon F_2 &\to \mathrm{PF}^{-}(G\setminus e,q,T \setminus e)
\end{align*}
with $\mathrm{deg}(\varphi_1(c)) = c$ and $\mathrm{deg}(\varphi_2(c)) = \mathrm{deg}(c) - 1$, thus proving the desired identity.

We define $\varphi_1(c) := c$, i.e., $\varphi_1$ is the identity map. Clearly this is an invertible map, we just need to check that it and its inverse take subparking functions to subparking functions. Let $f$ be the minimal edge in $T$ adjacent to $v$. We need the following key claim:
\begin{equation*} \label{claim:contraction}
 \textrm{If $U \subseteq V^{q}(G)$ with $v \in U^c$ then $f$ is not the minimum edge in $E(U,U^c)$}. \tag{*}
\end{equation*}
To prove~(\ref{claim:contraction}), suppose $\emptyset \neq U \subseteq V^{q}(G)$ with $v \in U^c$ is such that $f \in E(U,U^c)$. Then consider the path from $v$ to $q$ in $T$: it crosses into $U$ at $f$ and therefore must cross back into~$U^c$ at some edge~$g$ closer to~$q$ than~$f$. But since the total order of edges in~$T$ is consistent with the partial order of ancestry, we must have $g < f$. So indeed~$f$ is not the minimum edge in~$E(U,U^c)$ because~$g \in E(U,U^c)$. A consequence of~(\ref{claim:contraction}) is that for any~$u \in U \subseteq V^{q}(G/e)$ we have~$d^{G/e}_{U,T/e}(u) = d^{G}_{U,T}(u)$. Thus clearly if~$c \in F_1$ we have~$\varphi_1(c) \in \mathrm{PF}^{-}(G/e,q,T/e)$. Conversely, let $c' \in \mathrm{PF}^{-}(G/e,q,T/e)$ and set $c := \varphi^{-1}(c')$. Because of~(\ref{claim:contraction}) there can be no $\emptyset\neq U \subseteq V^{q}(G)$ with $v \in U^c$ such that $c_u \geq d^{G}_{U,T}$ for all $u \in U$. But on the other hand, if $v \in U$ then $d^{G}_{U,T}(v) \geq 1$ since~$e \in E(U,U^c)$ and~$e$ is not minimal in any cut (as it does not belong to $T$) and so~$0 = c_v < d^{G}_{U,T}$. Thus, indeed $c \in F_1$.

Next we define $\varphi_2(c) := c - (v)$. Again, this is clearly an invertible map, we just need to check that it and its inverse take subparking functions to subparking functions. But observe that because $e$ is never the minimum edge in any cut, we have for all~$U \subseteq V^{q}(G)$ and all~$u \in U$ that
\[ d_{U,T\setminus e}^{G \setminus e}(u) = \begin{cases} d^{G}_{U,T}(u) - 1 &\textrm{if $u = v$} \\ d^{G}_{U,T}(u) &\textrm{otherwise}.\end{cases} \]
Thus, $\varphi_2$ and $\varphi_2^{-1}$ clearly take subparking functions to subparking functions.
\end{proof}

Let $\mathcal{G}$ be the smallest set of triples $(G,q,T)$ where $G$ is a (connected) graph, $q \in V(G)$ is a choice of sink, and $T$ is an ordered, $q$-rooted spanning tree of $G$ such that
\begin{itemize}
\item if $G$ is the graph with one vertex and no edges, then $(G,q,T) \in \mathcal{G}$;
\item if there exists $e \in E(G)$ with $e$ an isthmus then $(G,q,T) \in \mathcal{G}$;
\item if there exists $e \in E(G)$ with $e$ a loop and $(G/e,q,T) \in \mathcal{G}$ then $(G,q,T) \in \mathcal{G}$;
\item if there exists a non-loop $e = \{q,u\} \in E(G)$ with $e \notin T$ and if  $(G/e,q,T/e) \in \mathcal{G}$ and $(G\setminus e,q,T \setminus e) \in \mathcal{G}$ then~$(G,q,T) \in \mathcal{G}$.
\end{itemize}

\begin{cor} \label{cor:hilbertseries}
$\mathrm{Hilb}(I^{-1}_{(G,q,T)};y) = y^{g} \cdot T_G(0,1/y)$ for $(G,q,T) \in \mathcal{G}$.
\end{cor}
\begin{proof}
This follows from Propositions~\ref{prop:subparkisthmus} and~\ref{prop:subparkloop}, Lemma~\ref{lem:subparkdc}, and Theorem~\ref{thm:gentutte}.
\end{proof}

\begin{lemma} \label{lem:acfpos}
$\mathrm{PF}^{-}(G,q,T) = \{(D_{\O})_{\mathbb{Z}V^{q}(G)}\colon \O \in \mathrm{ACI}(G,q,T) \}$ for~$(G,q,T) \in \mathcal{G}$.
\end{lemma}
\begin{proof}
We will prove this by induction on the number of edges of $G$. If $e \in E(G)$ is an isthmus then both sets are empty. If $e \in E(G)$ is a loop, then, as mentioned earlier, the loop~$e$ has no effect on the set of subparking functions and also~$e$ must be neutral in any~$\O \in \mathrm{ACI}(G,q,T)$, so the claim reduces to the same claim for~$(G/e,q,T)$. If~$G$ is the graph on one vertex and no edges then both sets are equal to $\{0\}$. Thus, we may assume $G$ has no isthmuses and loops, that there is $e = \{q,v\} \in E(G)$ with $e \notin T$, and by induction that the claim holds for $(G/e,q,T/e)$ and $(G \setminus e, q, T \setminus e)$. In what follows for convenience we write~$D_\O$ in place of~$(D_{\O})_{\mathbb{Z}V^{q}(G)}$ when the choice of sink is clear from context.

First let us show $\mathrm{PF}^{-}(G,q,T) \subseteq \{D_{\O}\colon \O \in \mathrm{ACI}(G,q,T) \}$. So let~$c \in \mathrm{PF}^{-}(G,q,T)$. By the proof of Lemma~\ref{lem:subparkdc}, we know that either the coefficient of $(v)$ in $c$ is $0$ and we have~$c \in \mathrm{PF}^{-}(G/e,q,T/e)$, or else $c - (v) \in \mathrm{PF}^{-}(G\setminus e, q, T \setminus e)$. Assume we are in the first case where the coefficient of $(v)$ in $c$ is $0$. By induction we can find a partial orientation~$\O' \in \mathrm{ACI}(G/e,q,T/e)$ with~$D_{\O'} = c$. Define~$\O := \O' \cup \{e^{\delta} = (q,v)\}$. Then~$D_{\O} = c$. We claim that~$\O \in \mathrm{ACI}(G,q,T)$. Clearly~$\O$ remains acyclic and $q$-connected. Let us show it is cut neutral. So assume there is a potential cut $\OCu$ for~$\O$ that is bad for the cut neutral property. First suppose $\Cu = (\{q\} \cup U_1, \{v\} \cup U_2)$. Then note that the only edge in~$E(\OCu)$ adjacent to $v$ that is oriented in~$\O$ is~$e$, but $e$ it is not in~$T$ and so cannot be the minimum edge in~$E(\OCu)$ and the minimum edge of any bad potential cut must be oriented. Thus there is some edge in~$E(\OCu)$ that is less than any edge in~$E(\OCu)$ adjacent to~$v$. Therefore $(\{q,v\} \cup U_1, U_2)$ is also bad. So we may assume that~$\OCu = (\{q,v\} \cup U_1, U_2)$. But then by the key claim~(\ref{claim:contraction}) in the proof of Lemma~\ref{lem:subparkdc} the minimum edge of~$\OCu' :=(\{q\} \cup U_1, U_2)$ is also the minimum edge of $\OCu$, so $\OCu'$ is a bad potential cut of~$\O'$, a contradiction. Thus indeed $\O$ is cut neutral. Next assume we are in the second case where the coefficient of $(v)$ is greater than~$0$. Again by induction we can find~$\O'' \in \mathrm{ACI}(G\setminus e,q,T \setminus e)$ with $D_{\O''} = c - (v)$. Define $\O := \O'' \cup \{e^{\delta} = (q,v)\}$. Then $D_{\O} = c$. Furthermore, we have~$\O \in A(G,q,T)$: clearly $\O$ remains acyclic and $q$-connected, and it is cut neutral because~$e$ is never the minimum edge in a cut. 

Now let us show $\{D_{\O} \colon \O \in \mathrm{ACI}(G,q,T) \} \subseteq \mathrm{PF}^{-}(G,q,T)$. So let~$\O \in  \mathrm{ACI}(G,q,T)$. If $e^{\delta} = (q,v) \notin \O$, then~$\O \cup \{e^{\delta}\} \in  \mathrm{ACI}(G,q,T)$: certainly~$\O$ remains acyclic and $q$-connected, and it is still cut neutral because~$e$ is never the minimum edge in a cut. Moreover, if we show~$D_{\O \cup \{e^{\delta}\}} \in \mathrm{PF}^{-}(G,q,T)$ this will show~$D_{\O} \in \mathrm{PF}^{-}(G,q,T)$ because $0 \leq D_{\O} \leq D_{\O \cup \{e^{\delta}\}}$ and the set of subparking functions is downward closed: that is, if~$0 \leq c \leq c'$ with~$c'$ a subparking function, then~$c$ is a subparking function. So now assume that $e^{\delta} \in \O$. First suppose that $e^{\delta}$ is the only oriented edge in $\O$ pointing into $v$. Then we claim that~$\O/e \in \mathrm{ACI}(G/e,q,T/e)$: it is easy to see that this orientation is acyclic and $q$-connected, and any bad potential cut for the cut neutral property for~$\O/e$ lifts to a bad potential cut of~$\O$ because by the key claim~(\ref{claim:contraction}) in the proof of Lemma~\ref{lem:subparkdc} when~$q$ and~$v$ are on the same side of a cut the minimum edge of this cut is the same in $G$ and $G/e$. Now assume that there is at least one other oriented edge pointing to $v$. Then by walking backwards from $v$ along this edge we see that there is another directed path from~$q$ to~$v$ that does not use~$e$. We claim that in this case $\O \setminus e \in \mathrm{ACI}(G\setminus e, q, T\setminus e)$: again it is clearly acyclic and $q$-connected, and it is cut neutral because in any cut that $e$ belonged to, we must have another oriented edge in the same direction coming from the other path from~$q$ to~$v$. But then by induction we know that $D_{\O/e} \in \mathrm{PF}^{-}(G/e,q,T/e)$ or $D_{\O\setminus e} \in \mathrm{PF}^{-}(G\setminus e,q,T\setminus e)$, and so by the proof of Lemma~\ref{lem:subparkdc}, we know that $D_{\O} \in \mathrm{PF}^{-}(G,q,T)$.
\end{proof}

Now let us apply the above results to $G$ when $G$ is outerplanar. Let $q \in V(G)$ be a choice of sink. We say that $T$ is a \emph{$q$-rooted boundary tree} if it can be obtained as follows: draw $G$ in the plane without crossings and with all its vertices on the boundary of the unbounded face; walk counterclockwise along this boundary starting at~$q$ and add an oriented edge~$e^{\delta} = (u,v)$ to~$T$ whenever you walk along~$e^{\delta}$ and visit~$v$ for the first time; order the edges in $T$ in the order they were walked along. For example, Figure~\ref{fig:boundarytree} depicts an outerplanar graph $G$ together with a $q$-rooted boundary tree $T$, with the edges of~$T$ oriented, in bold, and labeled according to order.
\begin{figure}
\begin{tikzpicture}[scale=0.8]
	\SetFancyGraph
	\Vertex[LabelOut,Lpos=270, Ldist=.1cm,x=0,y=0]{q}
	\Vertex[NoLabel,x=1,y=1]{v_1}
	\Vertex[NoLabel,x=2.25,y=2]{v_2}
	\Vertex[NoLabel,x=1,y=2.5]{v_3}
	\Vertex[NoLabel,x=1,y=3.5]{v_4}
	\Vertex[NoLabel,x=-0.5,y=3]{v_5}
	\Vertex[NoLabel,x=-0.5,y=1.75]{v_6}
	\Vertex[NoLabel,x=-1.5,y=0.5]{v_7}
	\Edges[style={->,>=mytip,ultra thick}](q,v_1)
	\Edges[style={thick,bend left=50}](q,v_1)
	\Edges[style={->,>=mytip,ultra thick}](v_1,v_2)
	\Edges[style={->,>=mytip,ultra thick}](v_2,v_3)
	\Edges[style={->,>=mytip,ultra thick}](v_3,v_4)
	\Edges[style={->,>=mytip,ultra thick}](v_4,v_5)
	\Edges[style={thick}](v_5,v_3)
	\Edges[style={->,>=mytip,ultra thick}](v_3,v_6)
	\Edges[style={thick}](v_6,v_1)
	\Edges[style={thick}](v_6,q)
	\Edges[style={->,>=mytip,ultra thick,bend right=20}](q,v_7)
	\Edges[style={thick,bend left=20}](q,v_7)
	\node at (0.6,0) {$1$};
	\node at (1.6,1.1) {$2$};
	\node at (2.0,2.4) {$3$};
	\node at (1.4,2.9) {$4$};
	\node at (0.4,3.6) {$5$};
	\node at (-0.1,2.5) {$6$};
	\node at (-0.5,0.7) {$7$};
\end{tikzpicture}
\caption{An example of an outerplanar graph $G$ with a $q$-rooted boundary tree~$T$.} \label{fig:boundarytree}
\end{figure}
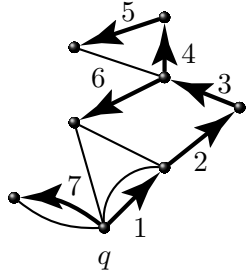

\begin{lemma} \label{lem:outerplanar}
If $G$ is an outerplanar graph, $q \in V(G)$ is a choice of sink, and $T$ is a $q$-rooted boundary tree, then $(G,q,T) \in \mathcal{G}$.
\end{lemma}
\begin{proof}
The proof is by induction on the number of edges of $G$. The case where $G$ has one vertex and no edges is trivial. If $G$ has a loop $e$ at $q$ then we can contract $e$ and $T$ is still a $q$-rooted boundary tree of the outerplanar graph $G/e$, so by induction~$(G/e,q,T) \in \mathcal{G}$ and thus~$(G,q,T) \in \mathcal{G}$ by the definition of $\mathcal{G}$. Suppose all edges containing $q$ in $E(G)$ belong to $T$: then we claim $G$ has an isthmus $e$. Indeed, this can happen only if as we are walking along the boundary we walk along an edge $e^{\delta} = (q,v)$ and then later walk along $e^{-\delta}$; in this case, $e$ must be an isthmus. So in this case $(G,q,T) \in \mathcal{G}$.

Thus, we can assume that there is some non-loop $e = \{q,v\}$ with $e \notin T$, and we can assume that $e$ is chosen to be the ``leftmost'' such edge, i.e., it is the last one we walk along when walking along the boundary of $G$. Then we claim that $T/e$ and $T \setminus e$ remain $q$-rooted boundary trees of the outerplanar graphs $G/e$ and $G\setminus e$, respectively. First consider~$T/e$: we can take the drawing that verifies $T$ is a $q$-rooted boundary tree and ``squish'' $e$ in this drawing to obtain a drawing of $G/e$; then (neglecting any loops produced, which are irrelevant) if we walk counterclockwise from $q$ around the boundary of this drawing of $G/e$ we walk along the same edges in the same order as in the walk for $T$; so this drawing verifies that $T/e$ is also a $q$-rooted boundary tree. Now consider~$T\setminus e$: again, we take the drawing for $T$ and delete $e$ from the drawing to obtain a drawing of $G\setminus e$; now as we walk counterclockwise from~$q$ along the boundary of this drawing of~$G \setminus e$ we may walk along some new edges after we visit~$v$, but we will have already visited all vertices at that point; so this drawing verifies that~$T\setminus e$ is also a $q$-rooted boundary tree. Therefore by induction $(G/e,q,T/e), (G\setminus e,q,T\setminus e) \in \mathcal{G}$, which means $(G,q,T) \in \mathcal{G}$ by the definition of $\mathcal{G}$.
\end{proof}

\begin{thm}
Conjecture~\ref{conj:subpark} is true when $G$ is outerplanar.
\end{thm}
\begin{proof}
Let $q \in V(G)$. Certainly there exists a $q$-rooted boundary tree $T$. Then by Lemma~\ref{lem:outerplanar} we have $(G,q,T) \in \mathcal{G}$. So statement~\ref{statement:hilbertseries} of Conjecture~\ref{conj:subpark} holds by Corollary~\ref{cor:hilbertseries}, and statement~\ref{statement:acfpos} holds by Lemma~\ref{lem:acfpos}.
\end{proof}

\begin{remark}
Gessel and Sagan~\cite{gessel1996tutte} count the number of acyclic partial orientations of $G$ by decomposing the poset of acyclic partial orientations into a number of intervals equal to the number of forests of $G$. The interval corresponding to a forest $F$ is a Boolean lattice of order $|E\setminus (F \cup L(F))|$ where $L(F)$ denotes the set of externally active edges in $F$ as in~\S\ref{subsec:future} (albeit with a different notion of external activity). The partial orientations belonging to an interval corresponding to a spanning tree are precisely the acyclic, $q$-connected partial orientations. It should be possible to fit the acyclic-cut internal partial orientations into this story: in particular, they should be precisely the partial orientations that belong to an interval corresponding to a spanning tree with no internal activity. Indeed, extending the interval decomposition of~\cite{gessel1996tutte} in this manner seems like a promising approach to proving Conjecture~\ref{conj:subpark}. The main issue is that there are so many choices of data, including even which notion of activity to use.
\end{remark}

\begin{remark} \label{rem:minres}
Recently there has been a great deal of interest in understanding minimal free resolutions of~$I_{(G,q)}$ and minimal free resolutions of a certain binomial ideal for which~$I_{(G,q)}$ is a distinguished initial ideal~\cite[\S7]{perkinson2013primer}~\cite{manjunath2013monomials}~\cite{manjunath2014minimal}~\cite{dochtermann2012laplacian}~\cite{mohammadi2013divisors}~\cite{hopkins2014another}~\cite{mohammadi2013divisors2}~\cite{mohammadi2014divisors}. It would be interesting to find a combinatorial description of a minimal free resolution of $I^{-1}_{(G,q,T)}$ or to compute its Betti numbers combinatorially. Even more interesting would be to find some combinatorially-meaningful binomial ideal which has $I^{-1}_{(G,q,T)}$ as an initial ideal for an appropriate choice of term order.
\end{remark}

\subsection{(Co)graphic Lawrence ideals, cut (cycle) connected fourientations and cut (cycle) minimal partial orienations} \label{sec:lawrence}

As in Remark~\ref{rem:param}, let~$W \simeq \mathbb{R}^{E(G)}$ be the vector space with basis~$x_{e^+} = -x_{e^-} $ for~\mbox{$e \in E(G)$}. Consider the lattice~$\mathbb{Z}^{E(G)}$ inside of~$W$. Given~$u = \sum_{e \in E(G)} c_e (x_{e^+})  \in \mathbb{Z}^{E(G)}$, let us define~$u^+ := \sum_{c_e \geq 0} c_e (x_{e^+})$ and~$u^- := - \sum_{c_e \leq 0} c_e (x_{e^+})$ to be the positive and negative parts of~$u$. Fix a field~$\mathbf{k}$ and let~$S = \mathbf{k}[y^{+}_e,y^{-}_e\colon e \in E(G)]$ be a polynomial ring in~$2|E(G)|$ variables. To a lattice element~$u \in \mathbb{Z}^{E(G)}$ we associate a binomial~$b(u) := y_{+}^{u^+}y_{-}^{u^-} - y_{-}^{u^+}y_{+}^{u^-} \in S$, where we use the notation~$y_{\pm}^{c} := \prod_{e \in E(G)} (y^{\pm}_e)^{c_e}$ for~$c = \sum_{e \in E(G)} c_e (x_{e^+}) \in \mathbb{N}^{E(G)}$. Let~$L$ be a sublattice of~$\mathbb{Z}^{E(G)}$. To~$L$ we associate the binomial \emph{Lawrence ideal}~$I_L := \langle b(u) : u \in L \rangle$. 

Recall that for a directed cut or cycle~$\OC$ of~$G$ we defined~$x_{\OC} := \sum_{e^{\pm} \in \mathbb{E}(\OC)}x_{e^{\pm}} \in W$. We define the \emph{cut lattice of~$(G,\Oref)$} to be~$\langle x_{\OCu}: \OCu \textrm{ is a directed cut} \rangle_{\mathbb{Z}}$ and the \emph{cycle lattice} to be~$\langle x_{\OCy}: \OCy \textrm{ is a directed cycle} \rangle_{\mathbb{Z}}$.  See~\cite{bacher1997lattice} for a more organic homological description of the cut and cycle lattices.  The graphic and cographic Lawrence ideals, which we will denote~$I^{\OCu}_{(G,\Oref)}$ and~$I^{\OCy}_{(G,\Oref)}$, are the Lawrence ideals associated to the cut and cycle lattices respectively.  The observation which relates these ideals to cycle/cocycle reversal systems is the following: we can encode a fourientation~$\O$ of~$G$ as a squarefree monomial~$y^{\O} := \prod_{e^{\pm} \in \O} y^{\pm}_e \in S$; then multivariate division of~$y^{\O}$ by some~$b(x_{\OCy})$ ($b(x_{\OCu})$) corresponds to a (co)cycle reversal of~$\OCy$ ($\OCu$) in~$\O$.  These ideals have been previously studied in the context of algebraic combinatorics and algebraic statistics~\cite{bayer1999syzygies}~\cite{drton2008lectures}~\cite{kateri2014family}~\cite{mohammadi2013divisors}~\cite{mohammadi2013divisors2}~\cite{mohammadi2014divisors}.  A theorem of Sturmfels~\cite[Theorem 7.1]{sturmfels1996grobner} about binomial generating sets for Lawrence ideals implies that~$\{ b(x_{\OCu}): \OCu \textrm{ is a directed cut} \}$ and~$\{ b(x_{\OCu}): \OCy \textrm{ is a directed cycle} \}$  are universal Gr\"obner bases for the ideals that they generate.

Mohammadi and Shokrieh~\cite{mohammadi2013divisors2} investigate the graphic Lawrence ideal~$I^{\OCu}_{(G,\Oref)}$ and construct a minimal free resolution of~$I^{\OCu}_{(G,\Oref)}$ with the aim of relating it to~$I_{(G,q)}$ and the binomial ideal mentioned in Remark~\ref{rem:minres} of which~$I_{(G,q)}$ is a distinguished initial ideal. They relate these ideals via regular sequences, as we will now explain. Define~$\mathrm{in}_{<}(I^{\OCu}_{(G,\Oref)})$ to be the initial ideal of $I^{\OCu}_{(G,\Oref)}$ with respect to lexicographic term order $\prec$ where~$y^{-}_{e_m} \prec y^{+}_{e_m} \prec y^{-}_{e_{m-1}} \prec y^{+}_{e_{m-1}} \prec \cdots \prec y^{-}_{e_1} \prec y^{+}_{e_1}$ is  the order we choose on the generators of~$S$ if~$e_1 < e_2 < \cdots < e_m$ are the edges of $G$. Explicitly, by the aforementioned theorem of Sturmfels, we have
\[ \mathrm{in}_{<}(I^{\OCu}_{(G,\Oref)}) := \left\langle m(x_{\OCu})\colon \parbox{3in} {\begin{center}$\OCu$ is a directed cut with $e^{+}_{\mathrm{min}} \in \mathbb{E}(\OCu)$, \\ where $e_{\mathrm{min}}$ is the minimal element of $E(\OCu)$\end{center}} \right\rangle \]
where to $u \in \mathbb{Z}^{E(G)}$ we associate the monomial $m(u) := y_{+}^{u^+}y_{-}^{u^-} \in S$. Choose a sink~$q \in V(G)$ and suppose $\G$ is~$(q,T)$-connected. In this case, Mohammadi and Shokrieh term~$\mathrm{in}_{<}(I^{\OCu}_{(G,\Oref)})$ the \emph{graphic oriented matroid ideal} of $G$. 

For each~$v \in V(G)$ choose some $e^{\delta_v}_{v} \in \mathbb{E}(G)$ so that~$e^{\delta_v}_{v} = (u,v)$ for some $u \in V(G)$ and then define the set~$\mathcal{L}_v := \{e^{\pm} - e^{\delta_v}_v\colon e^{\pm} = (u,v) \textrm{ for some $u \in V(G)$ and $e^{\pm} \neq e^{\delta_v}_v$}\}$. Set~$\mathcal{L} := \cup_{v \in V} \mathcal{L}_v$ and~$\mathcal{L}^{q} := \mathcal{L} \cup \{e^{\delta_q}_q\}$. Mohammadi and Shokrieh~\cite[Proposition~9.6]{mohammadi2013divisors2} (see also~\cite{hopkins2014another}) prove~$\mathcal{L}^{q}$ is a permutable regular sequence for~$S/\mathrm{in}_{<}(I^{\OCu}_{(G,\Oref)})$. Tensoring with $S/\langle\mathcal{L}^{q}\rangle$ should be seen as the algebraic version of taking the indegree of an orientation: it moves from edge orientation variables to vertex variables. It is clear from the presentation of these ideals by generators that~$S/\mathrm{in}_{<}(I^{\OCu}_{(G,\Oref)}) \otimes_S S/\langle\mathcal{L}^{q} \rangle \simeq R/I_{(G,q)}$.  That $\mathcal{L}^{q}$ is a regular sequence in particular implies
\[\mathrm{Hilb}(R/I_{(G,q)};y) = (1-y)^{|\mathcal{L}^{q}| } \cdot \mathrm{Hilb}(S/\mathrm{in}_{<}(I^{\OCu}_{(G,\Oref)});y);\] 
but~$|\mathcal{L}^{q}| = 2|E| - n +1$, so from the theorem of Merino we mentioned at the beginning of~\S\ref{subsec:cutin} one concludes
\[\mathrm{Hilb}(S/I^{\OCu}_{(G,\Oref)};y) = \frac{y^g \cdot T_G(1,1/y)}{(1-y)^{2|E|-n+1}}.\]

We now offer a different way to compute the Hilbert series of $S/I^{\OCu}_{(G,\Oref)}$ ($S/I^{\OCy}_{(G,\Oref)}$) in terms of the Tutte polynomial using cut (cycle) connected fourientations that avoids the use of regular sequences and instead studies the polyhedral combinatorics of these ideals directly (and which does not depend on choosing $(q,T)$-connected data).

\begin{prop} \label{prop:lawrencehilbertseries}
We have
\begin{align*}
\mathrm{Hilb}(S/I^{\OCu}_{(G,\Oref)};y) &= \frac{y^g \cdot T_G(1,1/y)}{(1-y)^{2|E|-n+1}}; \\
\mathrm{Hilb}(S/I^{\OCy}_{(G,\Oref)};y) &= \frac{y^{n-1} \cdot T_G(1/y,1)}{(1-y)^{2|E|-g}}.
\end{align*}
\end{prop}
\begin{proof}
One can easily see that the squarefree standard monomials of $ \mathrm{in}_{<}(I^{\OCu}_{(G,\Oref)})$ are $y^{\O}$ for~$\O \in \Delta^{\OCu}_{\G}$ where~$\Delta^{\OCu}_{\G} := \{\O^c\colon \textrm{$\O$ is a cut connected fourientation of $\G$}\}$.\footnote{It has previously been observed by Manjunath and Sturmfels~\cite{manjunath2013monomials} that Alexander duality for the $G$-parking function ideal $I_{(G,q)}$ is closely related to Reimann-Roch duality.  When we instead work with edge variables rather than vertex variables,  Alexander duality with respect to $y^{\mathbb{E}(G)}$ is precisely Riemann-Roch duality.} In other words, the Stanley-Reisner ring of the simplicial complex~$\Delta^{\OCu}_{\G}$ is~$S/\mathrm{in}_{<}(I^{\OCu}_{(G,\Oref)})$. Note that Theorem~\ref{thm:main} implies that
\[ \sum_{\sigma \in \Delta^{\OCu}_{\G}} y^{|\sigma|} = \sum_{\substack{\textrm{$\O$ a cut connected} \\ \textrm{fourientation of $\G$}}} y^{|\O^{o}|+2\cdot|\O^{u}|} = (y+1)^{n-1}(y^2+y)^{g}\cdot T_G\left(1,1+\frac{1}{y}\right)\]
and so~$\mathrm{dim}(\Delta^{\OCu}_{\G}) = |E(G)| + g - 1$. Therefore, by again applying Theorem~\ref{thm:main} we get that the $f$-polynomial of~$\Delta^{\OCu}_{\G}$ is
\begin{align*}
F_{\Delta^{\OCu}_{\G}}(y) &= \sum_{\sigma \in \Delta^{\OCu}_{\G}} y^{\mathrm{dim}(\Delta^{\OCu}_{\G}) -\mathrm{dim}(\sigma)} \\
&= y^{|E(G)| + g} \cdot \sum_{\substack{\textrm{$\O$ a cut connected} \\ \textrm{fourientation of $\G$}}} y^{-|\O^{o}|-2\cdot|\O^{u}|} \\
&= y^{|E(G)| + g} \left(\frac{1+y}{y}\right)^{n-1} \left(\frac{1+y}{y^2}\right)^{g} \cdot T_G\left(1, \frac{\left(\frac{1+y}{y}\right)^2}{\frac{1+y}{y^2}}\right) \\
&= (1+y)^{|E(G)|}\cdot T_G(1,1+y).
\end{align*}
And so the $h$-polynomial of~$\Delta^{\OCu}_{\G}$ is $H_{\Delta^{\OCu}_{\G}}(y) = F_{\Delta^{\OCu}_{\G}}(y-1) = y^{|E(G)|}\cdot T_G(1,y)$. Basic combinatorial commutative algebra~\cite[Corollary 1.5]{miller2005combinatorial} then implies that
\[\mathrm{Hilb}(S/\mathrm{in}_{<}(I^{\OCu}_{(G,\Oref)});y) = \frac{y^{\mathrm{dim}(\Delta^{\OCu}_{\G})+1} \cdot H_{\Delta^{\OCu}_{\G}}(1/y)}{(1-y)^{\mathrm{dim}(\Delta^{\OCu}_{\G})+1}} = \frac{y^g \cdot T_G(1,1/y)}{(1-y)^{2|E|-n+1}}.\]
Of course this means that Tutte polynomial expression is the Hilbert series of $S/I^{\OCu}_{(G,\Oref)}$ as well. An analogous argument for the cycle case establishes that $S/\mathrm{in}_{<}(I^{\OCy}_{(G,\Oref)})$ is the Stanley-Reisner ring of~$\Delta^{\OCy}_{\G} := \{-\O\colon \textrm{$\O$ is a cycle connected fourientation of $\G$}\}$ and that in particular the Hilbert series of the cographic Lawrence ideal is as claimed.
\end{proof}

We recall that Gessel and Sagan~\cite{gessel1996tutte} obtained the generating function for cut connected fourientations of $\G$ by size in the case where~$\G$ is~$(q,T)$-connected. Their result could be used in the proof of Propoisiton~\ref{prop:lawrencehilbertseries} instead of Theorem~\ref{thm:main}. In fact, the aforementioned theorem of Sturmfels implies that this generating function is the same for all choices of total order and reference orientation.

Let $\mathfrak{o} := \langle\{y_e^{+}y_e^{-}\colon e \in E(G)\}\rangle$, an ideal of $S$. Tensoring with $S/\mathfrak{o}$ is the algebraic version of passing from fourientations to partial orientations: it kills bioriented edges. We can compute the Hilbert series of $S/\mathrm{in}_{<}(I^{\OCu}_{(G,\Oref)}) \otimes_S S/\mathfrak{o}$ ($S/\mathrm{in}_<(I^{\OCy}_{(G,\Oref)}) \otimes S/\mathfrak{o}$) in terms of the Tutte polynomial using cut (cycle) minimal partial orientations. 

\begin{prop}
We have
\begin{align*}
\mathrm{Hilb}(S/\mathrm{in}_{<}(I^{\OCu}_{(G,\Oref)}) \otimes_S S/\mathfrak{o};y) &= \frac{y^g \cdot T_G(1,1+\frac{1}{y})}{(1-y)^{|E|}}; \\
\mathrm{Hilb}(S/\mathrm{in}_<(I^{\OCy}_{(G,\Oref)}) \otimes S/\mathfrak{o};y) &= \frac{y^{n-1} \cdot T_G(1+\frac{1}{y},1)}{(1-y)^{|E|}}.
\end{align*}
\end{prop}
\begin{proof}
Set $\widehat{I} :=   \left\langle \left\{ m(x_{\OCu})\colon \textrm{$\OCu$ a directed cut, $e^{+}_{\mathrm{min}} \in \mathbb{E}(\OCu)$} \right\} \cup \{y_e^{+}y_e^{-}\colon e \in E(G)\}  \right\rangle$. Clearly $S/\mathrm{in}_{<}(I^{\OCu}_{(G,\Oref)}) \otimes_S S/\mathfrak{o} \simeq S/\widehat{I}$. The squarefree standard monomials of $\widehat{I}$ are~$y^{\O}$ for~$\O \in \widehat{\Delta}^{\;\OCu}_{\G}$ where~$\widehat{\Delta}^{\;\OCu}_{\G} := \{-\O\colon \textrm{$\O$ is a cut minimal partial orientation of $\G$}\}$. In other words, the Stanley-Reisner ring of the simplicial complex~$\widehat{\Delta}^{\OCu}_{\G}$ is~$S/\widehat{I}$. Note that Theorem~\ref{thm:main} implies that~$\mathrm{dim}(\widehat{\Delta}^{\OCu}_{\G}) = |E(G)|-1$. Therefore, by again applying Theorem~\ref{thm:main} we get that the $f$-polynomial of~$\widehat{\Delta}^{\OCu}_{\G}$ is
\begin{align*}
F_{\widehat{\Delta}^{\OCu}_{\G}}(y) &= \sum_{\sigma \in \widehat{\Delta}^{\OCu}_{\G}} y^{\mathrm{dim}(\widehat{\Delta}^{\OCu}_{\G}) -\mathrm{dim}(\sigma)} \\
&= y^{|E(G)|} \cdot \sum_{\substack{\textrm{$\O$ a cut minimal} \\ \textrm{partial orientation of $\G$}}} y^{-|\O^{o}|} \\
&= y^{|E(G)|} \left(\frac{1+y}{y}\right)^{n-1} \left(\frac{1}{y}\right)^{g} \cdot T_G\left(1, \frac{\frac{2+y}{y}}{\frac{1}{y}}\right) \\
&= (1+y)^{n-1}\cdot T_G(1,2+y).
\end{align*}
And so the $h$-polynomial of~$\widehat{\Delta}^{\OCu}_{\G}$ is $H_{\widehat{\Delta}^{\OCu}_{\G}}(y) = F_{\widehat{\Delta}^{\OCu}_{\G}}(y-1) = y^{n-1}\cdot T_G(1,1+y)$. Again we conclude
\[\mathrm{Hilb}(S/\widehat{I};y) = \frac{y^{\mathrm{dim}(\widehat{\Delta}^{\OCu}_{\G})+1} \cdot H_{\widehat{\Delta}^{\OCu}_{\G}}(1/y)}{(1-y)^{\mathrm{dim}(\widehat{\Delta}^{\OCu}_{\G})+1}} = \frac{y^g \cdot T_G(1,1+\frac{1}{y})}{(1-y)^{|E|}}.\]
An analogous argument holds for the cycle case.
\end{proof}

\subsection{Cut connected fourientations and the reliability polynomial}

Suppose that we remove each edge of $G$ independently with probability $p$; then the \emph{reliability polynomial} $R_G(p)$ of $G$ is the probability that the resulting subgraph is connected. Note that this subgraph is connected if and only if it is \emph{spanning} in the sense of~\S\ref{subsec:subgraphs}.  It is well-known (see~\cite[(3.3)]{welsh1999tutte} and~\cite[{\S}V.(15)]{welsh2000potts}), and easy to prove using the rank generating function description of the Tutte polynomial, that 
\[R_G(p) = (1-p)^{n-1}p^g\cdot T_G\left(1,\frac{1}{p}\right).\]
In this section we discuss a strong relationship between the cut connected fourientations and the reliability polynomial. Let $k$, $l$, and $m$ be nonnegative real numbers such that~$2k+l+m=1$. By abuse of notation, by a ``$(k,l,m)$-fourientation'' we will mean a randomly chosen fourientation where the probability of choosing $\O$ is $k^{|\O^o|}l^{|\O^u|}m^{|\O^b|}$.

\begin{thm}\label{thm:reliable}
Let $k$, $l$, and $m$ be nonnegative real numbers with~$2k+l+m=1$. The probability that a $(k,l,m)$-fourientation of $G$ is cut connected is~$R_G(p)$ where $p := k+l$.
\end{thm}
\begin{proof}
By Theorem~\ref{thm:main}, the probability that a $(k,l,m)$-fourientation of $G$ is cut connected is 
\begin{align*}
\sum_{\substack{\textrm{$\O$ cut connected}\\\textrm{fourientation of $G$}} } k^{|\O^o|}l^{|\O^u|}m^{|\O^b|} &=  (k+m)^{n-1}(k+l)^g\cdot T_G\left(\frac {k + m}{k+m},\frac{2 k + l +  m}{k+l}\right)\\
&= (k+m)^{n-1}(k+l)^g\cdot T_G\left(1,\frac{2k + l + m}{k+l}\right) \\
&=  (1-p)^{n-1}p^g\cdot T_G\left(1,\frac{1}{p}\right) \\
&= R_G(p).
\end{align*}
\end{proof}

For a fixed probability $p$, Theorem~\ref{thm:reliable} gives a one-parameter family of combinatorial interpretations of $R_G(p)$: if we choose some probability~$0 < p < 1$ and some parameter~$t \in \left[\mathrm{max}\left(\frac{1}{p}, \, \frac{1}{1-p}\right) - 2,\infty\right)$ then there are unique nonnegative real numbers~$k$, $l$, and~$m$ with~$2k+l+m=1$ such that $p =  k+l$ and~\mbox{$t = \frac{l+m}{k}$}. We now present some specializations of Theorem~\ref{thm:reliable}, the first of which shows that we can recover the classical description of $R_G(p)$.

\begin{cor} \label{cor:classicreliable}
By setting $k=0$ in Theorem~\ref{thm:reliable} we recover the classical description of the reliability polynomial.
\end{cor}
\begin{proof}
If $k=0$ then the random fourientation will contain only bioriented edges and unoriented edges.  We can consider these objects as random subgraphs by saying that a bioriented edge is ``present" and an unoriented edge is ``absent" as in~\S\ref{subsec:subgraphs}.  In this situation the fourientation is cut connected precisely when the subgraph is spanning in the sense of~\S\ref{subsec:subgraphs}.
\end{proof}

The following specialization recovers a result of the first author.

\begin{cor}\label{cor:cutmin}\cite[Theorem 5.1]{backman2014partial}
Let $0 \leq p \leq 1/2$. Then there are unique nonnegative real numbers~$k$ and~$m$ with $2k+m = 1$ such that $p = k$. In this case,~$R_G(p)$ is the probability that a~$(k,0,m)$-fourientation of $G$ is cut minimal when viewed as a partial orientation in the sense of~\S\ref{subsec:partial}.
\end{cor}

At the time of writing~\cite{backman2014partial} it seemed odd that the probability was restricted to lie between $0$ and $1/2$, and that something else must lie on ``the other side of $1/2$".  The following dual specialization clarifies this strange range restriction.

\begin{cor}\label{cor:dualreliable}
Let $1/2 \leq p \leq 1$. Then there are unique nonnegative real numbers~$k$ and~$l$ with $2k+l = 1$ such that $p = k+l$. In this case,~$R_G(p)$ is the probability that a~$(k,l,0)$-fourientation of $G$ is cut connected when viewed as a partial orientation in the sense of~\S\ref{subsec:partial}.
\end{cor}

The first author's result Corollary~\ref{cor:cutmin} was inspired by the work of the second author and Perkinson who showed the following.

\begin{prop}\label{prop:reliableplanar}\cite[Corollary 3.3]{hopkins2012bigraphical}
Let $G$ be a planar graph and $G^*$ be its planar dual.  Let $A \in \mathbb{R}_{>0}^{\mathbb{E}(G)}$ be a generic parameter list in the sense of~\S\ref{subsec:bigraph}. The probability that a partial orientation of $G$ chosen uniformly at random is~$A$-admissible is~$R_{G^{*}}(2/3)$.
\end{prop}

Let us clarify the relationship between Proposition~\ref{prop:reliableplanar} and Theorem~\ref{thm:reliable}. Recall from~\S\ref{subsec:bigraph} that the exponential parameter list $A^{<}$ is generic and that being~$A^{<}$-admissible is the same as being cycle neutral. This property is planar dual to cut neutral (which was studied in~\S\ref{subsec:bigraph} in relation to the cobigraphical arrangement). Although the two sets of are not equal, the number of cut neutral partial orientations is equal to the number of cut connected partial orientations by Theorem~\ref{thm:main} and so Proposition~\ref{prop:reliableplanar} follows from Corollary~\ref{cor:dualreliable}.

\begin{remark}
Let $0 \leq p \leq 1$. Let $\O_D$ be a random fourientation obtained in the following manner: independently for each $e \in E(G)$ and each $\delta \in \{+,-\}$, include~$e^{\delta}$ in~$\O_D$ with probability $(1-p)$ and exclude $e^{\delta}$ from $\O_D$ with probability $p$. Let $\O_U$ be a different random fourientation obtained as follows: independently for each $e \in E(G)$, include both $e^{+}$ and $e^{-}$  in $\O_U$ with probability $(1-p)$ and exclude both $e^{+}$ and $e^{-}$ from $\O_U$ with probability $p$. Theorem~\ref{thm:reliable} implies that the probability that $\O_D$ is cut connected is the same as the probability that $\O_U$ is cut connected and in fact they are both equal to $R_G(p)$: for $\O_D$ we take $(k,l,m) := (p(1-p),p^2,(1-p)^2)$ and for~$\O_U$ we take $(k,l,m) := (0,p,(1-p))$. Suppose $\G$ is $(q,T)$-connected for some choice of sink~$q\in V(G)$. Our claims about $\O_D$ and $\O_U$ can be reinterpreted as follows. Let~$G_D$ be the \emph{directed} graph obtained from $G$ by including the directed edges~$(u,v)$, $(v,u)$ in~$G_D$ for each $e = \{u,v\} \in E(G)$. Remove each directed edge from $G_D$ independently with probability $p$; from the above claim about $\O_D$ we conclude that the probability the resulting ``subdigraph'' is $q$-connected is $R_G(p)$. On the other hand, as explained in Corollary~\ref{cor:classicreliable}, the claim about $\O_U$ just amounts to the classical definition of the reliability polynomial. Thus we see that the ``directed'' and ``undirected'' system reliability models have the same probability of failure. The connection between partial orientations and system reliability, and especially this intriguing fact that the ``directed'' and ``undirected'' system reliability models corresponding to a graph $G$ have the same probability of failure, were recently explored by Mohammadi~\cite{mohammadi2014divisors}.
\end{remark}

\bibliographystyle{plain}
\bibliography{Fourientations-and-the-Tutte-Polynomial}

\end{document}